\documentclass[a4paper,11pt]{amsart}
\usepackage[left=2.5cm, right=2.5cm,top=2.5cm,bottom=3cm]{geometry}

\usepackage{amsmath,amssymb,amsthm,mathtools}
\usepackage{booktabs,array}
\usepackage{enumitem}
\usepackage{placeins}

\usepackage{tikz}
\usepackage{algorithm}
\usepackage{algpseudocode}
\usetikzlibrary{arrows.meta}
\usepackage{pgfplots}
\pgfplotsset{compat=1.18}
\usepackage[hidelinks]{hyperref}
\usepackage{xcolor}
\DeclareTextSymbolDefault{\textsection}{OMS}

\newtheorem{theorem}{Theorem}[section]
\newtheorem{proposition}[theorem]{Proposition}
\newtheorem{lemma}[theorem]{Lemma}
\newtheorem{corollary}[theorem]{Corollary}

\theoremstyle{definition}
\newtheorem{definition}[theorem]{Definition}
\newtheorem{assumption}[theorem]{Assumption}
\newtheorem{example}[theorem]{Example}
\theoremstyle{remark}
\newtheorem{remark}[theorem]{Remark}

\newcommand{\R}{\mathbb{R}}
\newcommand{\Z}{\mathbb{Z}}
\newcommand{\C}{\mathbb{C}}
\newcommand{\RP}{\mathbb{RP}}
\newcommand{\T}{\mathbb{T}}
\newcommand{\RT}{\mathbb{RT}}
\DeclareMathOperator{\sgn}{sgn}
\DeclareMathOperator{\conv}{conv}
\DeclareMathOperator{\Trop}{Trop}
\newcommand{\Log}{\operatorname{Log}}
\newcommand{\cA}{\mathcal{A}}
\newcommand{\cB}{\mathcal{B}}
\newcommand{\E}{\mathbb{E}}
\newcommand{\cS}{\mathcal{S}}

\newcommand{\cO}{\mathcal{O}}

\title{Recovering the Topology of a Real Plane Algebraic Curve by Tropical Homotopy}
\author{Weixun Deng}
\address{Academy of Mathematics and Systems Science, Chinese Academy of Sciences, Beijing, China}
\email{wxdeng@amss.ac.cn}
\date{\today}

\begin{document}
\begin{abstract}
We present a tropical homotopy method for computing the numbers of ovals and unbounded components of a real plane algebraic curve, together with the pairing of its ends at infinity. Under explicit support and nondegeneracy assumptions, the homotopy starts from a tropical patchwork and encounters finitely many events. At Morse event points, the curve changes by local surgery. To determine the global attachments, we follow descending separatrices of a single gradient-like field to a fixed collection of base curves. The order of their landing points defines a finite combinatorial state, updated by local rules that also handle events at infinity and simultaneous event points. This recovers component data without tracing the intervening curves. For the dequantizing lifting, given by the negative logarithms of coefficient magnitudes, we prove an explicit event window when the lifting is generic. For fixed support and coefficient signs, with independent, identically distributed log-normal magnitudes, the expected number of events is bounded by a constant divided by the standard deviation of the logarithms. When all event points are Morse, correctness is conditional on complete, correctly classified event data, a sufficiently small starting parameter, and correctly labelled and ordered landing points. A numerical implementation illustrates its behaviour on sparse curves of high degree and on coefficients spanning many scales; these conditions are not certified.
\end{abstract}
\maketitle

\section{Introduction}\label{sec:intro}

Let $f=\sum_{a\in\cA}c_ax^a\in\R[x_1,x_2]$ and $C=\{f=0\}\subset\R^2$. Under the assumptions of \S\ref{sec:setting}, $C$ is a disjoint union of $n_o$ ovals and $n_u$ properly embedded lines. We compute $(n_o,n_u)$ and the pairing of the ends by the lines, but not the isotopy type of $(\R^2,C)$, which also depends on the nesting of the ovals. The number $n_u$ is read off from the edge polynomials; the global problem is to count the ovals and to pair the ends of the lines.

Standard algorithms project $C$ to a line and analyse the fibres over the critical values of the projection \cite{Col75,GVEK96,EKW07,BEKS13}; the elimination polynomials can have degree $d(d-1)$ for a curve of degree $d$, and coefficients of very different sizes put the features of $C$ at exponentially different scales, even for sparse $f$. We follow instead a homotopy from the tropical limit, in logarithmic coordinates: for a generic lifting $h\colon\cA\to\R$ let $f_t=\sum_ac_at^{h(a)}x^a$ and $C_t=\{f_t=0\}$, $0<t\le1$. (i) For small $t$, $C_t$ has the topology of the tropical patchwork $T(f,h)$, glued along the axes from the positive tropical curves of the four quadrants \cite{Vir84,Vir06,IMS07} (\S\ref{sec:tropical}). (ii) The topology of $C_t$ changes only at finitely many parameters, the events, at singular points of $C_t$ and tangencies with the divisors at infinity; if these event points are Morse, as for almost every lifting except on edges of the Newton polygon detached from the rest of the support, $C_t$ changes by local surgeries (\S\ref{sec:events}). (iii) A saddle, or the creation of two points at infinity, needs the components that contain its local arcs. The descending separatrices of one gradient-like field land on a fixed \emph{bottom}, the curve $C_{t_0}$ at a level $t_0$ before the first event, read on $T(f,h)$, and the small curves born later; their landing points cut it into intervals, which label the arcs of all later curves, and a finite state listing every component as a sequence of intervals is updated by local rules at every event (\S\ref{sec:global}). The resulting algorithm is correct if the list of events is complete and correctly classified, $t_0$ is small enough and the landing points are correctly labelled and ordered (Theorem~\ref{thm:algorithm}); our implementation does not certify them.

For the dequantizing lifting $h=-\log|c|$ every event satisfies $s\,g\le r$, where $s=1+\log(1/t)$, $g\ge0$ is a gap of the coefficients and $r$ an explicit constant (Theorem~\ref{thm:window}). For coefficients $\varsigma_ae^{\sigma u_a}$ with fixed signs and independent standard Gaussian $u_a$, the expected numbers of events and of Morse event points are $O(1/\sigma)$ (Theorem~\ref{thm:prob}), whereas for other liftings they need not tend to zero (Proposition~\ref{prop:otherrays}). In our experiments (\S\ref{sec:algorithm}) the solver that localizes the events by dominant terms was the fastest of the tested methods for few terms and high degree, and slower than projection methods for dense polynomials of moderate degree; every answer of our two solvers contradicting an exact reference was flagged, an observation, not a guarantee.

\emph{Relation to previous work.} Following a coefficient path through the discriminant by Morse theory is classical: for pencils of real curves \cite[\S\S4.6.7--4.6.9]{DK00}, from a chamber of known topology to bound numbers of components \cite[\S3]{FNR17}, \cite[\S3]{DRR25}, and for isotopy types of exponential sums over circuits \cite{BCDPRR23}; completed discriminant contours control isotopy types in \cite{RR17}. Separatrices were used in polygonizations of implicit surfaces to find the pieces that take part in a change of topology \cite{SH97}. Our states describe the Reeb graph of the height on the swept surface, and the separatrices belong to a Morse--Smale complex, which are computed combinatorially for piecewise-linear functions on triangulated surfaces \cite{EHZ03,CEHNP04}; here the surface is only implicit. The real polyhedral homotopy of \cite{EdW23} also starts at the tropical limit, but on paths certified to avoid the discriminant, and for T-curves the real scheme is computed combinatorially \cite{GJKMPSWZ26}; we cross the discriminant and update the global attachment at each crossing. As far as we know, new are the finite combinatorial state given by ordered landing points on a fixed bottom, which also handles the ends at infinity and simultaneous event points, and, for the dequantization path, the event window and the probability bound.

\section{Setting}\label{sec:setting}

\subsection{Support, quadrants and the curve at infinity}\label{ssec:support}\label{ssec:infinity}

Let $f=\sum_{a\in\cA}c_ax^a\in\R[x_1,x_2]$ with finite support $\cA\subset\Z_{\ge0}^2$ and $c_a\neq0$, $C=\{f=0\}\subset\R^2$, and $\Delta=\conv(\cA)$. For $\epsilon\in\{\pm1\}^2$ let $\cO_\epsilon=\{x\in\R^2:\sgn x_i=\epsilon_i\}$. A \emph{circuit} of $\cA$ is a minimal affinely dependent subset: three collinear points, or four points no three of which are collinear. Throughout we assume:

\begin{assumption}\label{ass:H0}
(H0) $0\in \cA$, and $\cA$ contains points $(k,0)$ and $(0,l)$ with $k,l\ge1$.
\end{assumption}

Thus $f(0)\neq0$, and neither coordinate axis is contained in $C$. The edges $E_1=\Delta\cap\{a_2=0\}$ and $E_2=\Delta\cap\{a_1=0\}$ are the \emph{axis edges}, the other edges of $\Delta$ the \emph{outer edges}. For non-constant $f$, (H0) holds after a generic affine change of coordinates, which preserves the topology of $C$ but can destroy sparsity.

Let $X_\Delta$ be the real toric surface of $\Delta$ \cite{Ful93}. By (H0) its affine chart at the vertex $0$ is $\R^2$, and $\partial X=X_\Delta\setminus\R^2=\bigcup_eD_e$ is the union of the toric divisors of the outer edges $e$. Let $\nu_e\in\Z^2$ be the primitive inner normal of $e$, $m_e=\min_{a\in\Delta}\langle a,\nu_e\rangle$, attained exactly on $e$, and $\mu_e\in\Z^2$ with $\det(\mu_e,\nu_e)=1$. The monomial change $x_i=y^{\mu_{e,i}}z^{\nu_{e,i}}$ is an automorphism of $(\R^*)^2$, and
\[
f=z^{m_e}g^e(y,z),\qquad g^e(y,z)=\sum_{a\in\cA}c_a\,y^{\langle a,\mu_e\rangle}z^{\langle a,\nu_e\rangle-m_e},
\]
with $g^e$ polynomial in $z$. The set $U_e=\R^*_y\times\R_z$ is a chart of $X_\Delta$ in which $z\ne0$ is the torus and $D_e^\circ=\{z=0\}$ is $D_e$ minus the torus-fixed points of the vertices of $e$. Let $\overline C\subset X_\Delta$ be the curve defined by $f$ ($C$ in $\R^2$, $\{g^e=0\}$ in $U_e$), \emph{smooth} at a point if this local equation has a non-zero gradient there. Its trace on $D_e^\circ$ is the zero set of the \emph{edge polynomial} $\psi^e(y)=g^e(y,0)=\sum_{a\in\cA\cap e}c_a\,y^{\langle a,\mu_e\rangle}$, of degree $L_e$, the lattice length of $e$, after division by a monomial. Near a point $(y_0,0)$, $y_0\neq0$, the sides $z>0$ and $z<0$ lie in two different quadrants, since $\nu_e$ is primitive and so has an odd coordinate.

\begin{lemma}[no corner points]\label{lem:fixed}
$\overline{C}$ contains no torus-fixed point of $X_\Delta$. Hence $\overline{C}$ lies in the smooth locus of $X_\Delta$, and $\overline{C}\cap\partial X=\bigsqcup_{e\text{ outer}}\{y\in\R^*:\ \psi^e(y)=0\}$.
\end{lemma}

\begin{proof}
Let $p_v$ be the fixed point of a vertex $v$ of $\Delta$. On the affine chart of $p_v$ the function $x^{-v}f=\sum_ac_ax^{a-v}$ is regular, and $x^{a-v}$, $a\ne v$, vanishes at $p_v$, since $a-v$ is a non-zero element of the pointed cone generated by $\Delta-v$; so $x^{-v}f(p_v)=c_v\neq0$. A normal toric surface is smooth outside its fixed points \cite{Ful93}, and the other points of $\partial X$ are those of the $D_e^\circ$.
\end{proof}

\begin{assumption}\label{ass:A1}
(A1) $\overline C$ is smooth and meets each $D_e$ transversally.
\end{assumption}

\begin{corollary}[ends]\label{cor:ends}
$\overline{C}$ is smooth and transverse to $\partial X$ at $p\in D_e^\circ$ if and only if $y(p)$ is a simple root of $\psi^e$; then $p$ is the limit of exactly two ends of $C$, one in each of the two quadrants adjacent to $D_e^\circ$ at $p$. Under (A1), $C$ is a smooth curve whose components are ovals (embedded circles) and lines (properly embedded copies of $\R$), and the number of lines is $n_u=\#(\overline C\cap\partial X)=\sum_{e\text{ outer}}\#\{y\in\R^*:\ \psi^e(y)=0\}$.
\end{corollary}

\begin{proof}
At a zero $p=(y_0,0)$ of $g^e$, $\partial_yg^e(p)=(\psi^e)'(y_0)$. If it is non-zero, $\overline{C}$ is smooth at $p$ and transverse to $\{z=0\}$, with one end on each side; otherwise it is singular or tangent to $D_e^\circ$ at $p$. Under (A1), $\overline C$ is a compact smooth curve, the closure of $C=\overline C\setminus\partial X$, so $C$ consists of circles and of arcs with ends at points of $\partial X$; each line has two ends and each point of $\overline C\cap\partial X$ is the limit of two, so $2n_u=2\,\#(\overline C\cap\partial X)$.
\end{proof}

\subsection{The data \texorpdfstring{$\Sigma$}{Sigma} and the homotopy}\label{ssec:sigma}\label{ssec:homotopy}

Under (A1) we orient $C$ so that $\{f>0\}$ lies on its left. An \emph{end} of $C$ is a pair $(p,Q)$ with $p\in\overline C\cap\partial X$ and $Q$ one of the two quadrants adjacent to $D_e^\circ$ at $p$, standing for the branch of $C$ in $Q$ that tends to $p$; each line runs from its \emph{initial} to its \emph{terminal} end. We compute the data $\Sigma(f)$: the components of $C$, as abstract ovals and lines, each line with its initial and its terminal end. They give $n_o$, $n_u$ and the pairing of the ends, but not the isotopy type of $(\R^2,C)$ (Remark~\ref{rem:isotopy}).

A \emph{lifting} is a function $h\colon\cA\to\R$, along which we deform the coefficients:
\begin{equation}\label{eq:deformation}
f_t(x)=\sum_{a\in \cA}c_a\,t^{h(a)}x^a,\qquad C_t=\{f_t=0\}\subset\R^2,\qquad 0<t\le1,
\end{equation}
so that $f_1=f$; we write $c(t)=(c_at^{h(a)})_{a\in\cA}$. All $f_t$ have the support $\cA$, hence the charts above, with local equations $g^e_t$ and edge polynomials $\psi^e_t$; where $f_t$ satisfies (A1) we write $\Sigma_t=\Sigma(f_t)$. Adding to $h$ an affine function $a\mapsto\lambda_0+\langle\lambda,a\rangle$ replaces $f_t(x)$ by $t^{\lambda_0}f_t(t^{\lambda_1}x_1,t^{\lambda_2}x_2)$; this changes neither the topology of any $C_t$, nor the events, nor the subdivision $\cS_h$ of \S\ref{sec:tropical}, but it changes the field of \S\ref{ssec:field} and hence the separatrices. The main lifting is the dequantizing one, $h=-\log|c|$, for which $f_t=f^{(s)}:=\sum_a\varsigma_a|c_a|^sx^a$ with $\varsigma=\sgn c$ and $s=1+\log(1/t)\ge1$ (Proposition~\ref{prop:dequant}).

Let $E_\cA=\sum_{\beta\in\cB}d_\beta c^\beta$, $\cB\subset\Z^{\cA}$ finite and all $d_\beta\neq0$, be the principal $\cA$-determinant \cite[Ch.~10]{GKZ94}. It vanishes exactly when, for some face $F$ of $\Delta$ (including $\Delta$), the truncation $f^F=\sum_{a\in F}c_ax^a$ has a singular point in $(\C^*)^2$, for an edge a multiple root in $\C^*$; up to a monomial it is a product of powers of the discriminants $\Delta_{\cA\cap F}$ of the faces \cite[Ch.~10, Thm.~1.2]{GKZ94}, and $\Delta_{\cA\cap e}$ is the discriminant of $\psi^e$. If $E_\cA(c)\neq0$, then (A1) holds: the face $\Delta$ excludes singular points of $C$ in the torus, the axis edges singular points on the axes, and the outer edges non-transverse points at infinity (Corollary~\ref{cor:ends}). The input is a polynomial with (H0) and $E_\cA(c)\neq0$, which makes the events finitely many (Lemma~\ref{lem:eventpoly}); the start needs a generic lifting (Definition~\ref{def:generic}), and the surgeries need Morse event points (Assumption~\ref{ass:morse}).

\section{The start: positive tropical curves and their gluing}\label{sec:tropical}

For small $t$ the topology of $C_t$ is determined by four positive tropical curves, one for each open quadrant, glued along the axes; only the explicit form of the gluing (\S\ref{ssec:gluing}) is new.

\subsection{Positive tropical curves}
For a quadrant $\cO_\epsilon$ put $C_{t,\epsilon}=C_t\cap\cO_\epsilon$. For $0<t<1$ the map $\Log_t\colon\cO_\epsilon\to\R^2$, $\Log_t(x)=(\log|x_1|,\log|x_2|)/\log(1/t)$, is a diffeomorphism with inverse $w\mapsto\epsilon\,t^{-w}:=(\epsilon_1t^{-w_1},\epsilon_2t^{-w_2})$; the axes correspond to $w_i\to-\infty$. With $\sigma_\epsilon(a)=\sgn(c_a)\,\epsilon_1^{a_1}\epsilon_2^{a_2}$,
\begin{equation}\label{eq:logform}
\varphi_t^\epsilon(w):=f_t(\epsilon\,t^{-w})=\sum_{a\in \cA}\sigma_\epsilon(a)\,|c_a|\,t^{-(\langle a,w\rangle-h(a))},
\end{equation}
so $C_{t,\epsilon}$ is the image of the zero set of $\varphi_t^\epsilon$, in which the terms with the largest $\langle a,w\rangle-h(a)$ dominate as $t\to0$ (the max-convention of \cite{Mik05,IMS07}). Let $F_h(w)=\max_{a\in \cA}(\langle a,w\rangle-h(a))$ and $\cA(w)$ the set of the $a$ attaining it; the tropical curve $\Trop(h)$ is the set of $w$ with $|\cA(w)|\ge2$. The regular subdivision $\cS_h$ of $\cA$ is the projection of the lower faces of $\conv\{(a,h(a)):a\in \cA\}$; the sets $\cA(w)$ are the sets of points of $\cA$ over the lower faces, and $w\mapsto\conv \cA(w)$ is the inclusion-reversing duality between the cells of $\Trop(h)$ (with its regions) and those of $\cS_h$ \cite[\S3]{MS15}. An edge of $\Trop(h)$ is orthogonal to its dual edge $e$ of $\cS_h$; it is a segment if $e$ is interior to $\Delta$, and otherwise a ray along the outer normal of the edge of $\Delta$ containing $e$.

\begin{definition}\label{def:generic}
The lifting $h$ is \emph{generic} if for every lower face the points $a\in \cA$ with $(a,h(a))$ on it are affinely independent. Then $\cS_h$ is a triangulation with vertices in $\cA$, the other points of $\cA$ being lifted strictly above the lower hull, $\Trop(h)$ is trivalent, and $|\cA(w)|\le3$ for all $w$.
\end{definition}

The non-generic liftings lie in finitely many hyperplanes of $\R^{\cA}$, since a non-generic lower face contains the lift of a circuit, a linear condition on $h$. A sign function $\sigma\colon \cA\to\{\pm1\}$ splits $\cA$ into $\cA^\pm_\sigma=\sigma^{-1}(\pm1)$; let $F_h^{\pm}(w)=\max_{a\in \cA^\pm_\sigma}(\langle a,w\rangle-h(a))$, with $\max\emptyset=-\infty$. For $\sigma=\sigma_\epsilon$, $F_h^+$ (resp.\ $F_h^-$) is the exponent of the largest positive (resp.\ negative) term of~\eqref{eq:logform}.

\begin{definition}\label{def:tropplus}
The \emph{positive tropical curve} of $(\sigma,h)$ is $\Trop_+(\sigma,h)=\{w:\ F_h^+(w)=F_h^-(w)\}$, the set of $w$ such that $\cA(w)$ contains points of both signs; we write $\Gamma_\epsilon=\Trop_+(\sigma_\epsilon,h)$. An edge $[a,a']$ of $\cS_h$ is \emph{bichromatic} if $\sigma(a)\neq\sigma(a')$.
\end{definition}

For generic $h$, $\Trop_+(\sigma,h)$ is the limit of the logarithmic images of $\{\sum_a\sigma(a)t^{h(a)}x^a=0\}$ (Proposition~\ref{prop:limit}; see \cite{SW05,Ale13,JSY22}).

\begin{lemma}[structure of $\Trop_+$]\label{lem:structure}
Let $h$ be generic. (i) $\Trop_+(\sigma,h)$ is the union of the closed edges of $\Trop(h)$ dual to the bichromatic edges of $\cS_h$. (ii) Every triangle of $\cS_h$ has no or exactly two bichromatic edges, so every vertex of $\Trop(h)$ on $\Trop_+(\sigma,h)$ has exactly two adjacent edges in it. (iii) $\Trop_+(\sigma,h)$ is a properly embedded piecewise-linear $1$-submanifold of $\R^2$, whose components are closed polygons and properly embedded polygonal lines whose ends are rays dual to bichromatic edges of $\cS_h$ on $\partial\Delta$.
\end{lemma}

\begin{proof}
(i) On the relative interior of an edge dual to $e=[a,a']$, $\cA(w)=\{a,a'\}$, so it lies in $\Trop_+$ if and only if $e$ is bichromatic; at a vertex dual to a triangle $\tau$, $\cA(w)$ is the vertex set of $\tau$, so the vertex lies in $\Trop_+$ if and only if $\tau$ is not monochromatic, and it is then an endpoint of the edges dual to the bichromatic edges of $\tau$; in the open regions $|\cA(w)|=1$. (ii) Three vertices coloured with two colours are monochromatic or split $2+1$, and then exactly the two edges at the odd vertex are bichromatic. (iii) Hence $\Trop_+$ is a closed, locally finite union of segments and rays with all vertices of degree two, and its rays are dual to edges on $\partial\Delta$.
\end{proof}

Each edge of $\Gamma_\epsilon$ separates a region where $F_h^+>F_h^-$ from one where $F_h^->F_h^+$, the \emph{positive} and \emph{negative regions}.

\subsection{Convergence and isotopy in one quadrant}

The following two statements are the tropical form of Viro's patchworking theorem; we prove the first, which makes the role of the signs explicit.

\begin{proposition}[$\Gamma_\epsilon$ is the limit]\label{prop:limit}
Let $h$ be generic and $K\subset\R^2$ compact. For every $\delta>0$ there is $t_0>0$ such that for $0<t<t_0$: (i) $\Log_t(C_{t,\epsilon})\cap K$ lies in the $\delta$-neighbourhood of $\Gamma_\epsilon$, and $\Gamma_\epsilon\cap K$ in the $\delta$-neighbourhood of $\Log_t(C_{t,\epsilon})$; (ii) on $K$ minus the $\delta$-neighbourhood of $\Gamma_\epsilon$, $\varphi_t^\epsilon$ is positive on the positive and negative on the negative regions.
\end{proposition}

\begin{proof}
Let $m=\min_a|c_a|$, $M=\max_a|c_a|$, and $\varphi_t^\epsilon=P-N$, where $P$ (resp.\ $N$) is the sum of the terms of~\eqref{eq:logform} with $\sigma_\epsilon(a)=+1$ (resp.\ $-1$). Then $m\,t^{-F_h^+(w)}\le P(w)\le|\cA|M\,t^{-F_h^+(w)}$, and similarly for $N$ with $F_h^-$. Hence if $F_h^+(w)-F_h^-(w)\ge\eta>0$, then $P(w)/N(w)\ge\frac{m}{|\cA|M}t^{-\eta}>1$ for $t<(m/|\cA|M)^{1/\eta}$, and symmetrically (if a sign class is empty, $\Gamma_\epsilon=\emptyset$ and $\varphi_t^\epsilon$ has constant sign). As $|F_h^+-F_h^-|$ is continuous and positive on $K$ minus the open $\delta$-neighbourhood of $\Gamma_\epsilon$, it has a positive lower bound there; this proves (ii) and the first half of (i). For the second half, every $w_0\in\Gamma_\epsilon$ lies in a closed edge dual to a bichromatic edge (Lemma~\ref{lem:structure}), so arbitrarily close to $w_0$, on the two sides of the edge, there are points $w_\pm$ with $\pm(F_h^+-F_h^-)(w_\pm)>0$. Cover $\Gamma_\epsilon\cap K$ by finitely many balls $B(w_0^{(j)},\delta/2)$, $w_0^{(j)}\in\Gamma_\epsilon\cap K$, choose such $w^{(j)}_\pm$ in them, and let $\eta'$ be the minimum of $|F_h^+-F_h^-|$ at these points. For $t<(m/|\cA|M)^{1/\eta'}$, $\varphi_t^\epsilon$ has a zero on each segment $[w^{(j)}_-,w^{(j)}_+]\subset B(w_0^{(j)},\delta/2)$; so each point of $\Gamma_\epsilon\cap K$ is within $\delta$ of a zero.
\end{proof}

\begin{theorem}[Viro's patchworking, tropical form]\label{thm:viro-quadrant}
Let $h$ be generic. There is $t_0\in(0,1)$ such that for $0<t<t_0$ the curve $C_{t,\epsilon}$ is smooth for every $\epsilon$, and the pair $(\R^2,\Log_t(C_{t,\epsilon}))$ is isotopic to $(\R^2,\Gamma_\epsilon)$.
\end{theorem}

\begin{proof}[References]
Viro's Main Patchwork Theorem \cite{Vir84}, \cite[Thms.~1.7.A and 4.3.A]{Vir06}, see also \cite{Ris93}, \cite[Ch.~2]{IMS07}, applies to a regular subdivision of $\Delta$ with nondegenerate truncations $f^\tau$. For a generic lifting the cells are triangles and segments with affinely independent vertices, and such truncations are nondegenerate: at a singular point $x\in(\C^*)^2$ of $f^\tau$ the non-zero numbers $c_ax^a$ would give a linear relation $\sum_{a\in\tau}c_ax^a(1,a)=0$ among the linearly independent vectors $(1,a)$; their zero sets in the open quadrants are empty or smooth properly embedded copies of $\R$. The theorem identifies the piece of $\overline{C_t}$ over each quadrant with the patchwork of the charts of the truncations, which under $\Log_t$ is isotopic to $\Gamma_\epsilon$ \cite{Vir01}, \cite[Ch.~2]{IMS07}. The proofs assume an integer lifting but use only the asymptotics $t\to0$, and apply to real $h$.
\end{proof}

The threshold $t_0$ is not explicit; \S\ref{sec:events} avoids it, as the topology is constant up to the first event.

\subsection{Gluing the quadrants}\label{ssec:gluing}

We glue the four curves $\Gamma_\epsilon$ along the coordinate axes, compactifying each copy of $\R^2$ in the negative directions, which correspond to the axes; this is the real tropical plane of \cite{Mik05,Vir01} (see also \cite{RS23}).

\begin{definition}\label{def:RT2}
Let $\T=[-\infty,\infty)$. The \emph{real tropical plane} is $\RT^2=\bigl(\bigsqcup_{\epsilon}\T^2\times\{\epsilon\}\bigr)/\sim$, where $(w,\epsilon)\sim(w,\epsilon')$ if $w_i=-\infty$ for every $i$ with $\epsilon_i\neq\epsilon'_i$. The map $\mathcal E(w,\epsilon)=(\epsilon_1e^{w_1},\epsilon_2e^{w_2})$ ($e^{-\infty}=0$) is a homeomorphism $\RT^2\to\R^2$; it sends $\R^2\times\{\epsilon\}$ onto $\cO_\epsilon$, where $\mathcal E^{-1}$ is $\Log_t$ up to the factor $\log(1/t)$, the common boundary $\{w_2=-\infty\}$ of two adjacent copies onto a half-axis, and $(-\infty,-\infty)$ to the origin. With $\overline\Gamma_\epsilon$ the closure of $\Gamma_\epsilon\times\{\epsilon\}$ in $\T^2\times\{\epsilon\}$, the \emph{tropical patchwork} is $T(f,h)=\mathcal E\bigl(\bigcup_{\epsilon}\overline\Gamma_\epsilon\bigr)\subset\R^2$.
\end{definition}

\begin{lemma}[ends and gluing]\label{lem:ends}
Assume (H0) and let $h$ be generic.
(i) Let $\gamma\subset\Gamma_\epsilon$ be a ray dual to an edge $e=[a,a']$ of $\cS_h$ on $\partial\Delta$. If $e\subset E_1$, then $\gamma=\{(w_1^e,s):s\le s_0\}$ with $w_1^e=(h(a')-h(a))/(a'_1-a_1)$, and its closure in $\T^2$ adds the point $(w_1^e,-\infty)$; symmetrically for $e\subset E_2$. If $e$ lies on an outer edge, $\gamma$ has no limit point in $\T^2\setminus\R^2$. Hence $\overline\Gamma_\epsilon\setminus\Gamma_\epsilon$ is finite, and $(-\infty,-\infty)\notin\overline\Gamma_\epsilon$.
(ii) For $e\subset E_1$, $e$ is bichromatic for $\sigma_\epsilon$ if and only if it is for $\sigma_{\epsilon'}$, $\epsilon'=(\epsilon_1,-\epsilon_2)$; the rays dual to $e$ in $\Gamma_\epsilon$ and $\Gamma_{\epsilon'}$ coincide, and their endpoints at $w_2=-\infty$ are identified in $\RT^2$. Symmetrically for $E_2$.
(iii) $T(f,h)$ is a closed, properly embedded topological $1$-submanifold of $\R^2$ not containing the origin. It meets $\{x_2=0\}$ exactly in the points $(\pm e^{w_1^e},0)$, $e$ running over the edges of $\cS_h$ in $E_1$ that are bichromatic for $\sigma_{(\pm1,1)}$, and near each of them it is a segment crossing the axis orthogonally. Symmetrically for $\{x_1=0\}$.
\end{lemma}

\begin{proof}
(i) The ray dual to an edge $e$ on an edge $E$ of $\Delta$ has the direction of the outer normal $n_E$. For $E_1$, $n_E=(0,-1)$, and on the ray $\langle a,w\rangle-h(a)=\langle a',w\rangle-h(a')$ with $a_2=a'_2=0$, i.e.\ $w_1=w_1^e$. If $E$ is an outer edge, then $0\notin E$; as $\Delta\subset\R^2_{\ge0}$ and $0\in\Delta$, a vector $n$ with $n_1,n_2\le0$ attains $\max_\Delta\langle n,\cdot\rangle=0$ on a face containing $0$, so $n_E$ has a positive coordinate and the ray leaves every compact subset of $\T^2$. (ii) For $a\in E_1$, $\sigma_\epsilon(a)=\sgn(c_a)\epsilon_1^{a_1}$ does not depend on $\epsilon_2$, and the ray depends only on $e$ and $h$. (iii) By (i) and (ii), each point added at $w_2=-\infty$ is the common endpoint of exactly two rays, in adjacent quadrants, which $\mathcal E$ maps to a segment of the line $x_1=\epsilon_1e^{w_1^e}$; the $w_1^e$ are the slopes of the lower hull of the lifting on $E_1$, hence pairwise distinct. Away from the axes $T(f,h)$ is homeomorphic to $\bigsqcup_\epsilon\Gamma_\epsilon$ (Lemma~\ref{lem:structure}), and by (i) it is closed and misses the origin.
\end{proof}

Thus the crossing points with $\{x_2=0\}$ are the patchwork of $f_t(x_1,0)$: a root of sign $\pm$ appears at $w_1^e$ if and only if $e$ is bichromatic for the signs $\sgn(c_a)(\pm1)^{a_1}$ (Descartes' rule in tropical form). Walking along the bichromatic edges of $\cS_h$ in each quadrant and crossing to the adjacent quadrant at bichromatic edges in $E_1\cup E_2$ gives the components of $T(f,h)$ and, at the rays dual to edges on outer edges, its ends.

\begin{theorem}[patchworking in $\R^2$]\label{thm:viro-plane}
Assume (H0) and let $h$ be generic. There is $t_0\in(0,1)$ such that for $0<t<t_0$ the curve $C_t$ satisfies (A1) and there is a homeomorphism of $\R^2$ mapping $C_t$ onto $T(f,h)$ and each open quadrant and each open half-axis onto itself.
\end{theorem}

\begin{proof}
Viro's theorem \cite{Vir84,Vir06,Ris93} is a statement on $X_\Delta$: for small $t$ the pair $(X_\Delta,\overline{C_t})$ is homeomorphic to $(X_\Delta,P)$, where $P$ is the patchwork of the charts of the truncations to the triangles of $\cS_h$, by a homeomorphism that maps each torus orbit into itself; in particular (A1) holds. By (H0), $\R^2$ is the affine chart at the vertex $0$, whose orbits are the open quadrants, the open half-axes and the origin, so it remains to map $P\cap\R^2$ onto $T(f,h)$ by a homeomorphism of $\R^2$ preserving them. The identification of Theorem~\ref{thm:viro-quadrant} over $\cO_\epsilon$ is made triangle by triangle of $\cS_h$: the chart of a bichromatic triangle corresponds to the vertex of $\Gamma_\epsilon$ dual to it with the halves of its two edges, and the arc of $P$ next to a half-axis to the ray of $\Gamma_\epsilon$ dual to a bichromatic edge of $\cS_h$ in $E_1$, resp.\ $E_2$. This edge, its chart and the crossing point of $T(f,h)$ are the same for the two quadrants adjacent along the half-axis (Lemma~\ref{lem:ends}(ii),(iii)). Near the half-axis, $T(f,h)$ consists of segments orthogonal to it, a standard crossing model, and $P$ of arcs crossing it transversally at the corresponding points; in a thin collar we identify $P$ with this model by a homeomorphism supported in the collar that preserves the half-axis and the two quadrants, and $T(f,h)$ by the identity. Then the identifications of the two adjacent quadrants can be taken to agree near the half-axis, so they glue to a homeomorphism of $\R^2$ preserving the orbits.
\end{proof}

\begin{example}\label{ex:cubic}
Let $f=1+x_1^3+x_2^3-\lambda x_1x_2$, $\lambda>0$, with $h=0$ on the vertices of $\Delta$ and $h(1,1)=-1$. Then $f_t=1+x_1^3+x_2^3-\lambda t^{-1}x_1x_2$, $\cS_h$ is the star triangulation from $(1,1)$, and $\Trop(h)$ has the vertices $(0,-1)$, $(-1,0)$, $(1,1)$. For $\epsilon=(+,+)$ the signs on $\bigl((0,0),(3,0),(0,3),(1,1)\bigr)$ are $(+,+,+,-)$, so $\Gamma_{++}$ is the bounded triangle of $\Trop(h)$, an oval; in each other quadrant exactly one vertex of $\Delta$ has the odd sign, and after gluing the three curves form one unbounded component (Figure~\ref{fig:cubic}).
\end{example}

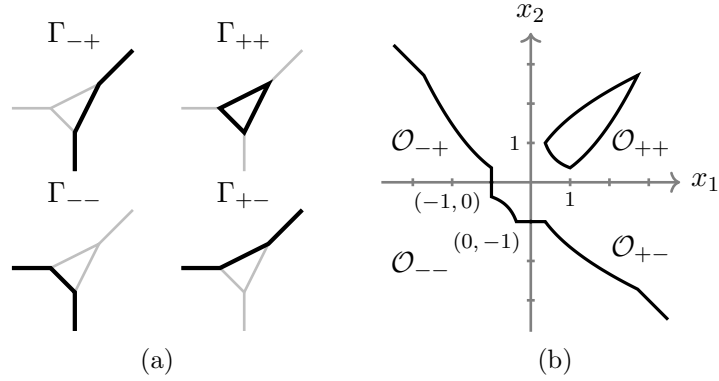
\begin{figure}[tb]
\centering
\begin{tabular}{@{}c@{\hspace{1.0cm}}c@{}}
\begin{tikzpicture}[scale=0.32]
\newcommand{\tropbase}{%
  \draw[gray!50,line width=1pt] (0,-1)--(0,-2.6); \draw[gray!50,line width=1pt] (-1,0)--(-2.6,0); \draw[gray!50,line width=1pt] (1,1)--(2.4,2.4);
  \draw[gray!50,line width=1pt] (0,-1)--(-1,0)--(1,1)--cycle;}
\begin{scope}[shift={(7,6.6)}]
  \tropbase \draw[line width=1.6pt] (0,-1)--(-1,0)--(1,1)--cycle; \node at (0,3) {$\Gamma_{++}$};
\end{scope}
\begin{scope}[shift={(0,6.6)}]
  \tropbase \draw[line width=1.6pt] (0,-2.6)--(0,-1)--(1,1)--(2.4,2.4); \node at (0,3) {$\Gamma_{-+}$};
\end{scope}
\begin{scope}[shift={(0,0)}]
  \tropbase \draw[line width=1.6pt] (0,-2.6)--(0,-1)--(-1,0)--(-2.6,0); \node at (0,3) {$\Gamma_{--}$};
\end{scope}
\begin{scope}[shift={(7,0)}]
  \tropbase \draw[line width=1.6pt] (-2.6,0)--(-1,0)--(1,1)--(2.4,2.4); \node at (0,3) {$\Gamma_{+-}$};
\end{scope}
\end{tikzpicture} &
\begin{tikzpicture}[scale=0.52]
\draw[->,gray,line width=1pt] (-3.8,0)--(3.8,0) node[right,text=black] {$x_1$};
\draw[->,gray,line width=1pt] (0,-3.8)--(0,3.8) node[above,text=black] {$x_2$};
\foreach \v in {-3,-2,-1,1,2,3} {\draw[gray,line width=1pt] (\v,0.1)--(\v,-0.1); \draw[gray,line width=1pt] (0.1,\v)--(-0.1,\v);}
\node[below] at (1,0) {\scriptsize $1$}; \node[left] at (0,1) {\scriptsize $1$};
\draw[very thick] (1.000,0.368) -- (0.975,0.377) -- (0.951,0.387) -- (0.928,0.397) -- (0.905,0.407) -- (0.882,0.417) -- (0.861,0.427) -- (0.839,0.438) -- (0.819,0.449) -- (0.799,0.461) -- (0.779,0.472) -- (0.760,0.484) -- (0.741,0.497) -- (0.723,0.509) -- (0.705,0.522) -- (0.687,0.535) -- (0.670,0.549) -- (0.654,0.563) -- (0.638,0.577) -- (0.622,0.592) -- (0.607,0.607) -- (0.592,0.622) -- (0.577,0.638) -- (0.563,0.654) -- (0.549,0.670) -- (0.535,0.687) -- (0.522,0.705) -- (0.509,0.723) -- (0.497,0.741) -- (0.484,0.760) -- (0.472,0.779) -- (0.461,0.799) -- (0.449,0.819) -- (0.438,0.839) -- (0.427,0.861) -- (0.417,0.882) -- (0.407,0.905) -- (0.397,0.928) -- (0.387,0.951) -- (0.377,0.975) -- (0.368,1.000) -- (0.368,1.000) -- (0.387,1.025) -- (0.407,1.051) -- (0.427,1.078) -- (0.449,1.105) -- (0.472,1.133) -- (0.497,1.162) -- (0.522,1.191) -- (0.549,1.221) -- (0.577,1.252) -- (0.607,1.284) -- (0.638,1.317) -- (0.670,1.350) -- (0.705,1.384) -- (0.741,1.419) -- (0.779,1.455) -- (0.819,1.492) -- (0.861,1.530) -- (0.905,1.568) -- (0.951,1.608) -- (1.000,1.649) -- (1.051,1.690) -- (1.105,1.733) -- (1.162,1.777) -- (1.221,1.822) -- (1.284,1.868) -- (1.350,1.916) -- (1.419,1.964) -- (1.492,2.014) -- (1.568,2.065) -- (1.649,2.117) -- (1.733,2.171) -- (1.822,2.226) -- (1.916,2.282) -- (2.014,2.340) -- (2.117,2.399) -- (2.226,2.460) -- (2.340,2.522) -- (2.460,2.586) -- (2.586,2.651) -- (2.718,2.718) -- (2.718,2.718) -- (2.651,2.586) -- (2.586,2.460) -- (2.522,2.340) -- (2.460,2.226) -- (2.399,2.117) -- (2.340,2.014) -- (2.282,1.916) -- (2.226,1.822) -- (2.171,1.733) -- (2.117,1.649) -- (2.065,1.568) -- (2.014,1.492) -- (1.964,1.419) -- (1.916,1.350) -- (1.868,1.284) -- (1.822,1.221) -- (1.777,1.162) -- (1.733,1.105) -- (1.690,1.051) -- (1.649,1.000) -- (1.608,0.951) -- (1.568,0.905) -- (1.530,0.861) -- (1.492,0.819) -- (1.455,0.779) -- (1.419,0.741) -- (1.384,0.705) -- (1.350,0.670) -- (1.317,0.638) -- (1.284,0.607) -- (1.252,0.577) -- (1.221,0.549) -- (1.191,0.522) -- (1.162,0.497) -- (1.133,0.472) -- (1.105,0.449) -- (1.078,0.427) -- (1.051,0.407) -- (1.025,0.387) -- (1.000,0.368) -- cycle;
\draw[very thick] (-3.490,3.490) -- (-2.718,2.718) -- (-2.651,2.586) -- (-2.586,2.460) -- (-2.522,2.340) -- (-2.460,2.226) -- (-2.399,2.117) -- (-2.340,2.014) -- (-2.282,1.916) -- (-2.226,1.822) -- (-2.171,1.733) -- (-2.117,1.649) -- (-2.065,1.568) -- (-2.014,1.492) -- (-1.964,1.419) -- (-1.916,1.350) -- (-1.868,1.284) -- (-1.822,1.221) -- (-1.777,1.162) -- (-1.733,1.105) -- (-1.690,1.051) -- (-1.649,1.000) -- (-1.608,0.951) -- (-1.568,0.905) -- (-1.530,0.861) -- (-1.492,0.819) -- (-1.455,0.779) -- (-1.419,0.741) -- (-1.384,0.705) -- (-1.350,0.670) -- (-1.317,0.638) -- (-1.284,0.607) -- (-1.252,0.577) -- (-1.221,0.549) -- (-1.191,0.522) -- (-1.162,0.497) -- (-1.133,0.472) -- (-1.105,0.449) -- (-1.078,0.427) -- (-1.051,0.407) -- (-1.025,0.387) -- (-1.000,0.368) -- (-1.000,0.000) -- (-1.000,-0.368) -- (-0.975,-0.377) -- (-0.951,-0.387) -- (-0.928,-0.397) -- (-0.905,-0.407) -- (-0.882,-0.417) -- (-0.861,-0.427) -- (-0.839,-0.438) -- (-0.819,-0.449) -- (-0.799,-0.461) -- (-0.779,-0.472) -- (-0.760,-0.484) -- (-0.741,-0.497) -- (-0.723,-0.509) -- (-0.705,-0.522) -- (-0.687,-0.535) -- (-0.670,-0.549) -- (-0.654,-0.563) -- (-0.638,-0.577) -- (-0.622,-0.592) -- (-0.607,-0.607) -- (-0.592,-0.622) -- (-0.577,-0.638) -- (-0.563,-0.654) -- (-0.549,-0.670) -- (-0.535,-0.687) -- (-0.522,-0.705) -- (-0.509,-0.723) -- (-0.497,-0.741) -- (-0.484,-0.760) -- (-0.472,-0.779) -- (-0.461,-0.799) -- (-0.449,-0.819) -- (-0.438,-0.839) -- (-0.427,-0.861) -- (-0.417,-0.882) -- (-0.407,-0.905) -- (-0.397,-0.928) -- (-0.387,-0.951) -- (-0.377,-0.975) -- (-0.368,-1.000) -- (0.000,-1.000) -- (0.368,-1.000) -- (0.387,-1.025) -- (0.407,-1.051) -- (0.427,-1.078) -- (0.449,-1.105) -- (0.472,-1.133) -- (0.497,-1.162) -- (0.522,-1.191) -- (0.549,-1.221) -- (0.577,-1.252) -- (0.607,-1.284) -- (0.638,-1.317) -- (0.670,-1.350) -- (0.705,-1.384) -- (0.741,-1.419) -- (0.779,-1.455) -- (0.819,-1.492) -- (0.861,-1.530) -- (0.905,-1.568) -- (0.951,-1.608) -- (1.000,-1.649) -- (1.051,-1.690) -- (1.105,-1.733) -- (1.162,-1.777) -- (1.221,-1.822) -- (1.284,-1.868) -- (1.350,-1.916) -- (1.419,-1.964) -- (1.492,-2.014) -- (1.568,-2.065) -- (1.649,-2.117) -- (1.733,-2.171) -- (1.822,-2.226) -- (1.916,-2.282) -- (2.014,-2.340) -- (2.117,-2.399) -- (2.226,-2.460) -- (2.340,-2.522) -- (2.460,-2.586) -- (2.586,-2.651) -- (2.718,-2.718) -- (3.490,-3.490);
\fill (-1,0) circle (1.4pt); \fill (0,-1) circle (1.4pt);
\node[below left] at (-1,0) {\scriptsize $(-1,0)$};
\node[below left] at (0,-1) {\scriptsize $(0,-1)$};
\node at (2.8,1.0) {$\mathcal O_{++}$}; \node at (-2.8,1.0) {$\mathcal O_{-+}$};
\node at (-2.8,-2.1) {$\mathcal O_{--}$}; \node at (2.8,-1.7) {$\mathcal O_{+-}$};
\end{tikzpicture}\\
\small (a) & \small (b)
\end{tabular}
\caption{Example~\ref{ex:cubic}. (a) $\Trop(h)$ (grey) and $\Gamma_\epsilon$ (black) in the four logarithmic charts, placed like the quadrants; rays towards $w_i=-\infty$ are glued with the ray of the same edge in the adjacent chart. (b) $T(f,h)\subset\R^2$: an oval in $\cO_{++}$ and an unbounded component crossing the axes at $(-1,0)$ and $(0,-1)$.}
\label{fig:cubic}
\end{figure}

\section{Events and local surgery}\label{sec:events}

We now let $t$ run from the patchwork regime to $t=1$. The topology of $C_t$ changes only at singular points and at multiple roots of the edge polynomials, both handled by Morse theory; the attachment of the local branches to the rest of the curve is the subject of \S\ref{sec:global}.

\subsection{Events}

\begin{definition}\label{def:event}
A parameter $t\in(0,1]$ is an \emph{event} if $\overline{C_t}$ violates (A1). By Lemma~\ref{lem:fixed} and Corollary~\ref{cor:ends} this happens at the following \emph{event points}: (A) \emph{interior event points}, the real singular points $p\in\R^2$ of $C_t$, i.e.\ $f_t(p)=0$ and $\nabla f_t(p)=0$, including those on the axes; (B) \emph{boundary event points}, the multiple roots $y_0\in\R^*$ of $\psi^e_t$, for an outer edge $e$.
\end{definition}

\begin{lemma}[finiteness of events]\label{lem:eventpoly}
Every event is a zero of $t\mapsto E_\cA(c(t))$. If $E_\cA(c)\neq0$, there are finitely many events in $(0,1]$, and none in some interval $(0,t_0)$.
\end{lemma}

\begin{proof}
A real singular point of $C_t$ in the torus is a singular point of $f_t=f_t^\Delta$; one on an axis, $(p_1,0)$ with $p_1\neq0$, is a multiple root of the truncation $f_t(x_1,0)$ to $E_1$, similarly for $E_2$, and the origin is not on $C_t$; a boundary event point is a multiple root of the truncation to $e$ in the coordinate $y$. So $E_\cA(c(t))=0$. Finally $E_\cA(c(t))=\sum_\beta d_\beta c^\beta t^{\langle\beta,h\rangle}$ is a sum of real powers of $t$, non-zero at $t=1$, so it has finitely many positive zeros by Descartes' rule for real exponents \cite[Part~V]{PS76}, and its term of lowest exponent dominates as $t\to0$.
\end{proof}

Not every zero of $E_\cA(c(t))$ is an event, because of complex singular points, non-real singular points of truncations and tangencies with the axes; an outer edge with only two points of $\cA$ produces no boundary event. We assume:

\begin{assumption}[Morse events]\label{ass:morse}
Every event point of an event $t^*$ is of one of the following kinds.
(A) An interior event point $p$ at which the Hessian $H_p$ of $f_{t^*}$ is non-degenerate and $\alpha:=\partial_tf_t(p)|_{t=t^*}\neq0$.
(B) A double root $y_0$ of $\psi^e_{t^*}$ with $\partial_zg^e_{t^*}(y_0,0)\neq0$ and $\alpha:=\partial_t\psi^e_t(y_0)|_{t=t^*}\neq0$.
\end{assumption}

Such points are isolated, and $\alpha$ is the speed at which the critical value of $f_t$ near $p$, resp.\ of $\psi^e_t$ near $y_0$, crosses zero (envelope theorem). The support can force several event points at one parameter (Example~\ref{ex:simultaneous}), and on the following edges, and only there, non-Morse event points for every lifting: the axis edge $E_1$ is \emph{detached} if it contains at least three points of $\cA$ and no $a\in\cA$ has $a_2\in\{1,2\}$, symmetrically for $E_2$, and an outer edge $e$ is \emph{detached} if it contains at least three points of $\cA$ and no point of $\cA$ has lattice distance one from $e$.

\begin{proposition}[Morse events for almost every lifting]\label{prop:genericmorse}
Let $f$ satisfy (H0) and $E_\cA(c)\neq0$.
(i) For almost every lifting $h\in\R^\cA$, every event point satisfies Assumption~\ref{ass:morse}, except the singular points on the axis of a detached axis edge and the boundary event points on a detached outer edge; and $C_t$ has no singular point on $\{x_2=0\}$ if some $a\in\cA$ has $a_2=1$, symmetrically for $\{x_1=0\}$.
(ii) For every lifting, the event points excepted in (i) violate Assumption~\ref{ass:morse}.
(iii) For a support $\cA$ with (H0) and signs $\varsigma\in\{\pm1\}^\cA$, the conclusion of (i) holds for the dequantization path of almost every $c$ with $\sgn c=\varsigma$.
\end{proposition}

\begin{proof}
There are no event points at the corners of $\partial X$ (Lemma~\ref{lem:fixed}) or at the origin, and an edge with two points of $\cA$ has a binomial truncation, without multiple roots in $\R^*$. (i) On $\cO_\epsilon$ write $f_t=\sum_aT_a$ with $T_a=\sigma_\epsilon(a)|c_a|t^{h(a)}e^{\langle a,u\rangle}$, $u=(\log|x_1|,\log|x_2|)$. The interior event points in $\cO_\epsilon$ are the zeros $(u,t)$, $0<t<1$, of $\Phi=(\sum_aT_a,\sum_aa\,T_a)$ ($t=1$ is not an event). As $\partial\Phi/\partial h(a)=\log t\,T_a(1,a)$ and the vectors $(1,a)$ span $\R^3$ by (H0), $\Phi$ is a submersion on $\R^2\times(0,1)\times\R^\cA$, and by Sard's theorem almost every $h$ is a regular value of the projection of $\Phi^{-1}(0)$ to $\R^\cA$. Then the derivative of $\Phi$ in $(u,t)$ is invertible at every zero; since $\partial_u\sum_aT_a=0$ there, its determinant is $\partial_tf_t\cdot\det\operatorname{Hess}_uf_t$, and $\operatorname{Hess}_uf_t=DH_pD$ with $D=\operatorname{diag}(p_1,p_2)$: this is (A). The same argument for $\phi_t=f_t(x_1,0)$ in $(\log|x_1|,t)$, and for $\psi^e_t$ in $(\log|y|,t)$, shows that for almost every $h$ their real multiple roots $q$ are double roots with $\partial_t\phi_t(q)\neq0$, resp.\ $\partial_t\psi^e_t(q)\neq0$; they are finitely many, since $E_\cA(c(t))$ vanishes there (proof of Lemma~\ref{lem:eventpoly}), and depend only on $h$ on $E_1$, resp.\ $e$. Let $m\ge1$ be the smallest $a_2>0$, resp.\ lattice distance from $e$, of a point of $\cA$, and $b$ the coefficient of $x_2^m$ in $f_t$, resp.\ of $z^m$ in $g^e_t$, at such a root: a sum of terms $c_at^{h(a)}q^{(\cdot)}$, one for each $a$ at height, resp.\ distance, $m$, so for fixed $h$ on $E_1$, resp.\ $e$, it vanishes only for $h$ in a null set, and by Fubini's theorem $b\neq0$ for almost every $h$. On the axis, $m=1$ gives $\partial_2f_t(q,0)=b\neq0$, so $(q,0)$ is not singular, and $m=2$ gives the Hessian $\operatorname{diag}(\phi_t''(q),2b)$: (A). On an outer edge, $m=1$ gives $\partial_zg^e_t(q,0)=b\neq0$: (B). (ii) On the axis of a detached axis edge $\partial_2f_t$ and $\partial_2^2f_t$ vanish, so singular points there are degenerate; on a detached outer edge $\partial_zg^e_t$ vanishes on $D_e^\circ$. (iii) $E_\cA(c)\neq0$ for almost every $c$, and on the dequantization path $T_a=\sigma_\epsilon(a)e^{s\log|c_a|+\langle a,u\rangle}$, so the proof of (i) applies with $\log|c|$ and $s$ in place of $h$ and $\log t$.
\end{proof}

For example, $1+x_1^2-\sqrt2\,x_1x_2+x_2^2$ satisfies (H0) and $E_\cA(c)\neq0$, and its outer edge $e=[(2,0),(0,2)]$ is detached: for every lifting with $2h(1,1)-h(2,0)-h(0,2)<0$ the edge polynomial $t^{h(2,0)}-\sqrt2\,t^{h(1,1)}y+t^{h(0,2)}y^2$, $y=x_2/x_1$, has a double root at exactly one $t^*\in(0,1)$, where $\partial_zg^e_{t^*}=0$. At such points Algorithm~\ref{alg:main} returns \textsc{non-Morse}, and at a non-Morse point off the detached edges it takes a new lifting near $-\log|c|$; a lifting far from $-\log|c|$ can avoid such points, but it loses the window of \S\ref{sec:lifting} (Remark~\ref{rem:detached}).

\begin{example}[simultaneous events]\label{ex:simultaneous}
For $f=x_1^4-\frac32x_1^2+1+2x_2^2+\frac1{10}x_1x_2^2+x_2^3$, both points $(\pm1,0)$ are singular points of $\{f^{(s^*)}=0\}$, $s^*=\log2/\log\frac32$, with positive definite Hessians and $\partial_sf^{(s)}(\pm1,0)=-2\log\frac32$: two ovals die at the same parameter. For every lifting $f_t(x_1,0)$ is even and $\partial_2f_t$ vanishes on the axis, as no exponent has $a_2=1$; so the real double roots of $f_t(x_1,0)$ come in pairs $\pm p$ of singular points. Here $-\log|c|$ is generic and $E_\cA(c)\neq0$.
\end{example}

\subsection{Morse functions on the swept surface}

Let $S=\{(x,t)\in X_\Delta\times(0,1]:\ x\in\overline{C_t}\}$ and $\tau\colon S\to(0,1]$ the projection, so that $\overline{C_t}=\tau^{-1}(t)$. By Lemma~\ref{lem:fixed}, $S$ is the union of the \emph{interior stratum} $S^\circ=S\cap(\R^2\times(0,1])$ and the \emph{boundary strata} $B_e=S\cap(D_e^\circ\times(0,1])=\{(y,t):\ \psi^e_t(y)=0\}$, $e$ outer.

\begin{lemma}\label{lem:morse}
(i) At an event point $p$ of type (A) of $t^*$, $S^\circ$ is smooth near $(p,t^*)$ and is the graph of a function $t=T(x)$ with $\nabla T(p)=0$ and $\operatorname{Hess}T(p)=-H_p/\alpha$; so $(p,t^*)$ is a non-degenerate critical point of $\tau|_{S^\circ}$, of index $0$ (resp.\ $2$) if $H_p$ is definite of sign $\eta$ with $\eta\alpha<0$ (resp.\ $\eta\alpha>0$), and $1$ if $H_p$ is indefinite.
(ii) At an event point $p=(y_0,0)\in D_e^\circ$ of type (B) of $t^*$, with $\beta=(\psi^e_{t^*})''(y_0)$, $S$ is a smooth surface on which $\tau$ has no critical point, while $B_e$ is near $(y_0,t^*)$ the graph of a function $t=T(y)$ with $T'(y_0)=0$ and $T''(y_0)=-\beta/\alpha$; so $(y_0,t^*)$ is a non-degenerate critical point of $\tau|_{B_e}$, of index $0$ if $\alpha\beta<0$ and $1$ if $\alpha\beta>0$.
(iii) $\tau|_{S^\circ}$ and $\tau|_{B_e}$ have no other critical points.
\end{lemma}

\begin{proof}
(i) At $(p,t^*)$ the gradient of $(x,t)\mapsto f_t(x)$ is $(0,\alpha)\neq0$; the implicit function theorem gives $T$, and differentiating $f_{T(x)}(x)=0$ twice at $p$ gives $H_p+\alpha\operatorname{Hess}T(p)=0$. (ii) In $U_e$ the gradient of $(y,z,t)\mapsto g^e_t(y,z)$ at $(p,t^*)$ is $(0,\partial_zg,\alpha)$, not a multiple of $(0,0,1)$, so $S$ is smooth there and $\tau|_S$ is a submersion; $B_e=\{\psi^e_t(y)=0\}$ is the univariate case of (i). (iii) A critical point of $\tau|_{S^\circ}$ is a point where $\nabla_xf_t=0$, and one of $\tau|_{B_e}$ a point where $(\psi^e_t)'=0$.
\end{proof}

Thus boundary events are critical points of $\tau$ on the strata $B_e$ in the sense of stratified Morse theory \cite{GM88}; there only the intersection of $\overline{C_t}$ with $\partial X$ changes.

\begin{proposition}[no event, no change]\label{prop:isotopy}
If $[t_1,t_2]\subset(0,1]$ contains no event, there is an isotopy of $X_\Delta$ that preserves $\R^2$ and each $D_e$ and carries $\overline{C_{t_1}}$ onto $\overline{C_{t_2}}$.
\end{proposition}

\begin{proof}
Stratify $X_\Delta\times[t_1,t_2]$ by the products of the fixed points, of the $D_e^\circ$ and of $\R^2$ with $[t_1,t_2]$, refined by $S^\circ$, the $B_e$ and their complements. As there are no events, $S$ is a smooth surface transverse to each $D_e^\circ\times[t_1,t_2]$ and disjoint from the fixed points, so this is a Whitney stratification. By Lemma~\ref{lem:morse}(iii) the projection to $[t_1,t_2]$ is a submersion on every stratum, and it is proper on the whole space, as $X_\Delta$ is compact. Thom's first isotopy lemma \cite[Prop.~11.1]{Mat12}, \cite{GM88} gives the isotopy.
\end{proof}

\subsection{Local surgery}

\begin{proposition}[interior events]\label{prop:interior}
Let $p$ be an event point of type (A) of $t^*$, and $\delta>0$ small.
(i) If $H_p$ is definite of sign $\eta$, then, as $t$ increases through $t^*$, a small oval around $p$ is born if $\eta\alpha<0$ and dies if $\eta\alpha>0$.
(ii) If $H_p$ is indefinite, then in a neighbourhood $N$ of $p$ the two arcs $a_1,a_2$ of $C_{t^*-\delta}\cap N$ are replaced by the two arcs of $C_{t^*+\delta}\cap N$ joining the four endpoints in the other pattern. Let $\gamma_i$ be the component of $C_{t^*-\delta}$ containing $a_i$. If $\gamma_1\neq\gamma_2$, they become one component, an oval if both are ovals and otherwise a line with the ends of the line, except that two lines $u_1\to v_1$, $u_2\to v_2$ become the lines $u_1\to v_2$ and $u_2\to v_1$. If $\gamma_1=\gamma_2$, it splits into two components, and the one without the ends of a line is an oval.
\end{proposition}

\begin{proof}
By the Morse lemma with parameters there are coordinates $u$ near $p$, depending smoothly on $t$, with $f_t=v(t)+\eta_1u_1^2+\eta_2u_2^2$, where $v(t)$ is the critical value near $p$, $v(t^*)=0$ and $v'(t^*)=\alpha$. Hence $C_t\cap N=\{\eta_1u_1^2+\eta_2u_2^2=-v(t)\}$ with $-v(t)=-\alpha(t-t^*)+o(t-t^*)$, which gives (i), and in case (ii) the two patterns of arcs. As the old and the new curve are level sets oriented by the same rule, $a_1$ and $a_2$ are traversed in opposite directions across $N$, and each new arc runs from the exit point of one old arc to the entry point of the other. If $\gamma_1\neq\gamma_2$, the two components are therefore joined, and lines $u_i\to v_i$ become $u_1\to v_2$ and $u_2\to v_1$. If $\gamma_1=\gamma_2=\cdots a_1P_1a_2P_2\cdots$, the new arcs close $P_1$ and $P_2$ separately, so $\gamma_1$ splits into two components, and the one not containing the ends of a line is an oval.
\end{proof}

\begin{proposition}[boundary events]\label{prop:boundary}
Let $p=(y_0,0)\in D_e^\circ$ be an event point of type (B) of $t^*$ and $\beta=(\psi^e_{t^*})''(y_0)$. As $t$ increases through $t^*$, two points $p_1,p_2$ of $\overline{C_t}\cap D_e^\circ$ near $p$ are created if $\alpha\beta<0$ and annihilated if $\alpha\beta>0$. While they exist, the arc of $\overline{C_t}$ between $p_1$ and $p_2$ near $p$ lies in one quadrant $Q_m$ adjacent to $p$, and the two outer pieces lie in the other one, $Q_o$.
(i) \emph{Annihilation.} The arc between $p_1$ and $p_2$ is a whole line $\lambda$ of $C_t$ with ends $(p_1,Q_m)$, $(p_2,Q_m)$; it disappears. The ends $(p_1,Q_o)$ and $(p_2,Q_o)$ are the terminal end of a line $L$ and the initial end of a line $L'$, and they are joined: if $L=L'$, $L$ closes up into an oval; otherwise $L$ and $L'$ merge into one line from the initial end of $L$ to the terminal end of $L'$.
(ii) \emph{Creation} is the inverse. An arc of a component $\gamma$ of $C_{t^*-\delta}$ near $p$, lying in $Q_o$, is cut at $p_1,p_2$, and a new line $\lambda$ in $Q_m$ appears. If $\gamma$ is an oval it becomes a line; if $\gamma$ is a line $u\to v$ it becomes two lines $u\to(p_a,Q_o)$ and $(p_b,Q_o)\to v$, where $p_a$ is the point met first along $\gamma$.
\end{proposition}

\begin{proof}
As $\partial_zg^e_{t^*}(p)\neq0$, near $p$ the curve $\overline{C_t}$ is a graph $z=\phi_t(y)$, smooth in $(y,t)$, with $\phi_t(y)=0$ if and only if $\psi^e_t(y)=0$. By Lemma~\ref{lem:morse}(ii), $\psi^e_t$ has two simple roots near $y_0$ for $t$ on one side of $t^*$, given by the sign of $\alpha\beta$, and none on the other. Between the two roots $\phi_t$ has one sign and outside them the other, and the two signs of $z$ are two different quadrants. Near $p$ the middle arc has both ends on $D_e^\circ$, so it is a whole component of $C_t$. When $p_1,p_2$ merge, the two outer pieces are joined by an oriented arc through $p$, so one of $(p_1,Q_o),(p_2,Q_o)$ is a terminal and the other an initial end. The rest is bookkeeping.
\end{proof}

\begin{remark}[what is local and what is not]\label{rem:local}
If $t^*$ has several event points, $C_{t^*+\delta}$ is obtained from $C_{t^*-\delta}$ by the replacements of Propositions~\ref{prop:interior} and~\ref{prop:boundary} at all of them, applied one after another, each to the curve produced by the previous ones, which need not be a level of $f_t$; $\gamma_i$, resp.\ $L,L',\gamma$, are components of that curve, and the order is irrelevant (Theorem~\ref{thm:update}). Births and deaths need only the local picture, and an annihilation only the current list of ends; a saddle and a creation need in addition the components $\gamma_1,\gamma_2$, resp.\ $\gamma$, that contain the local arcs, the only global information the method needs (\S\ref{sec:global}). Tangencies of $C_t$ with the axes are not events, since they do not change $\Sigma_t$.
\end{remark}

\begin{example}[Example~\ref{ex:cubic} continued]\label{ex:cubic2}
Here $f_t=1+x_1^3+x_2^3-s\,x_1x_2$ with $s=\lambda/t$. There are no boundary events, as the outer edge contains only two points of $\cA$, and $\nabla f_t=0$, $f_t=0$ has the single real solution $s=3$, $p=(1,1)$, where $H_p$ is positive definite and $\alpha=\lambda t^{*-2}>0$. So for $\lambda<3$ the oval dies at $t^*=\lambda/3$: $(n_o,n_u)=(1,1)$ for $\lambda>3$ and $(0,1)$ for $\lambda<3$.
\end{example}

\begin{example}[a boundary event]\label{ex:conic}
Let $f=x_1^2-b\,x_1x_2+x_2^2+x_1-1$, $0<b<2$, an ellipse, with $h(1,1)=-1$, $h(1,0)=1$ and $h=0$ elsewhere: $f_t=x_1^2-s\,x_1x_2+x_2^2+t\,x_1-1$ with $s=b/t$, and $\cS_h$ has the triangles $[(0,0),(2,0),(1,1)]$ and $[(0,0),(1,1),(0,2)]$. For $e=[(2,0),(0,2)]$ take $z=1/x_1$, $y=x_2/x_1$; then $\psi^e_t(y)=1-s\,y+y^2$ and $g^e_t=\psi^e_t(y)+t\,z-z^2$. For small $t$ the curve is a hyperbola with the lines $A$ from $e_1=(p_1,\cO_{++})$ to $e_4=(p_2,\cO_{--})$ and $B$ from $e_3=(p_1,\cO_{--})$ to $e_2=(p_2,\cO_{++})$, where $p_1<p_2$ are the roots of $\psi^e_t$ (Figure~\ref{fig:conic}(a)). At $t_1$ with $4t_1^2+t_1^4=b^2$ the critical value $-1+t^2/(s^2-4)$ of $f_t$ vanishes: a saddle in $\cO_{++}$ re-pairs the lines (Figure~\ref{fig:conic}(b)), and afterwards one line has both ends in $\cO_{++}$ (Figure~\ref{fig:conic}(c)). At $t^*=b/2$, $\psi^e_{t^*}=(y-1)^2$, with $\partial_zg^e=t^*$, $\alpha=b\,t^{*-2}>0$ and $\beta=2$: an annihilation (Figure~\ref{fig:conic}(d)) with $Q_m=\cO_{++}$, whose outer ends, in $Q_o=\cO_{--}$, belong to one line, which closes up: $(n_o,n_u)=(1,0)$ (Figure~\ref{fig:conic}(e)).
\end{example}

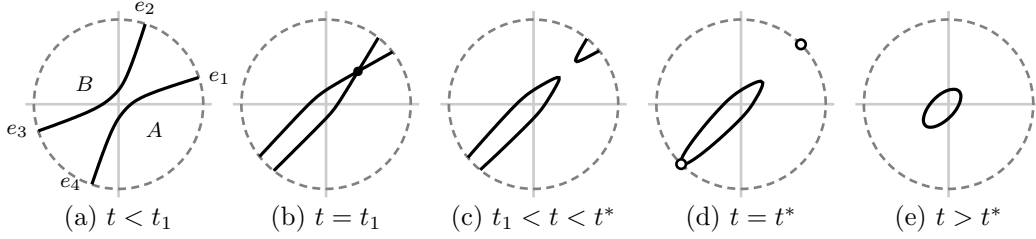
\begin{figure}[ht]
\centering
\begin{tikzpicture}[scale=1.12]
\begin{scope}[shift={(0.00,0)}]
\draw[gray!40,line width=1pt] (-1.1,0)--(1.1,0); \draw[gray!40,line width=1pt] (0,-1.1)--(0,1.1);
\draw[densely dashed,gray,line width=1pt] (0,0) circle (1);
\draw[very thick] (-0.946,-0.315) -- (-0.926,-0.308) -- (-0.907,-0.301) -- (-0.888,-0.294) -- (-0.868,-0.287) -- (-0.848,-0.280) -- (-0.830,-0.274) -- (-0.810,-0.266) -- (-0.790,-0.259) -- (-0.772,-0.253) -- (-0.753,-0.246) -- (-0.735,-0.239) -- (-0.715,-0.231) -- (-0.695,-0.224) -- (-0.677,-0.217) -- (-0.657,-0.210) -- (-0.639,-0.203) -- (-0.619,-0.196) -- (-0.599,-0.188) -- (-0.581,-0.181) -- (-0.561,-0.173) -- (-0.543,-0.166) -- (-0.524,-0.159) -- (-0.506,-0.152) -- (-0.486,-0.143) -- (-0.468,-0.136) -- (-0.448,-0.128) -- (-0.430,-0.121) -- (-0.410,-0.112) -- (-0.392,-0.105) -- (-0.372,-0.096) -- (-0.354,-0.088) -- (-0.337,-0.080) -- (-0.317,-0.070) -- (-0.299,-0.062) -- (-0.281,-0.053) -- (-0.263,-0.043) -- (-0.246,-0.034) -- (-0.228,-0.024) -- (-0.210,-0.013) -- (-0.192,-0.002) -- (-0.174,0.009) -- (-0.159,0.020) -- (-0.142,0.032) -- (-0.126,0.045) -- (-0.110,0.057) -- (-0.094,0.071) -- (-0.079,0.084) -- (-0.063,0.099) -- (-0.049,0.112) -- (-0.034,0.127) -- (-0.022,0.141) -- (-0.008,0.157) -- (0.006,0.174) -- (0.017,0.189) -- (0.028,0.206) -- (0.039,0.223) -- (0.050,0.241) -- (0.060,0.259) -- (0.069,0.277) -- (0.078,0.294) -- (0.088,0.314) -- (0.096,0.332) -- (0.104,0.350) -- (0.112,0.370) -- (0.120,0.388) -- (0.128,0.407) -- (0.135,0.426) -- (0.143,0.446) -- (0.150,0.463) -- (0.157,0.483) -- (0.165,0.503) -- (0.171,0.521) -- (0.178,0.541) -- (0.186,0.561) -- (0.192,0.579) -- (0.199,0.599) -- (0.206,0.618) -- (0.212,0.637) -- (0.219,0.656) -- (0.226,0.676) -- (0.232,0.695) -- (0.239,0.715) -- (0.245,0.735) -- (0.251,0.752) -- (0.258,0.772) -- (0.265,0.792) -- (0.271,0.810) -- (0.277,0.830) -- (0.284,0.850) -- (0.290,0.869) -- (0.296,0.888) -- (0.303,0.908) -- (0.309,0.928) -- (0.315,0.946);
\draw[very thick] (0.946,0.315) -- (0.926,0.309) -- (0.906,0.303) -- (0.886,0.297) -- (0.868,0.291) -- (0.848,0.285) -- (0.829,0.279) -- (0.810,0.273) -- (0.790,0.267) -- (0.770,0.260) -- (0.752,0.254) -- (0.732,0.248) -- (0.712,0.242) -- (0.692,0.235) -- (0.675,0.229) -- (0.655,0.223) -- (0.635,0.217) -- (0.617,0.210) -- (0.597,0.204) -- (0.577,0.197) -- (0.559,0.191) -- (0.539,0.184) -- (0.519,0.177) -- (0.501,0.171) -- (0.481,0.164) -- (0.463,0.158) -- (0.443,0.150) -- (0.425,0.143) -- (0.406,0.136) -- (0.386,0.128) -- (0.368,0.121) -- (0.348,0.113) -- (0.331,0.106) -- (0.312,0.098) -- (0.292,0.088) -- (0.274,0.080) -- (0.257,0.071) -- (0.239,0.061) -- (0.221,0.051) -- (0.204,0.041) -- (0.188,0.031) -- (0.170,0.018) -- (0.154,0.006) -- (0.139,-0.007) -- (0.123,-0.021) -- (0.109,-0.034) -- (0.094,-0.049) -- (0.081,-0.063) -- (0.068,-0.078) -- (0.054,-0.094) -- (0.041,-0.110) -- (0.029,-0.126) -- (0.017,-0.142) -- (0.006,-0.158) -- (-0.006,-0.175) -- (-0.017,-0.193) -- (-0.027,-0.210) -- (-0.037,-0.228) -- (-0.046,-0.246) -- (-0.057,-0.266) -- (-0.065,-0.283) -- (-0.074,-0.301) -- (-0.082,-0.319) -- (-0.091,-0.339) -- (-0.099,-0.357) -- (-0.106,-0.374) -- (-0.114,-0.394) -- (-0.122,-0.412) -- (-0.130,-0.432) -- (-0.137,-0.450) -- (-0.145,-0.470) -- (-0.152,-0.488) -- (-0.159,-0.506) -- (-0.166,-0.526) -- (-0.174,-0.546) -- (-0.180,-0.563) -- (-0.188,-0.583) -- (-0.194,-0.601) -- (-0.202,-0.621) -- (-0.209,-0.641) -- (-0.215,-0.659) -- (-0.222,-0.679) -- (-0.229,-0.697) -- (-0.236,-0.717) -- (-0.243,-0.737) -- (-0.249,-0.755) -- (-0.256,-0.775) -- (-0.262,-0.792) -- (-0.269,-0.812) -- (-0.276,-0.832) -- (-0.282,-0.850) -- (-0.289,-0.870) -- (-0.296,-0.890) -- (-0.302,-0.908) -- (-0.309,-0.928) -- (-0.315,-0.946);
\node[font=\scriptsize,right] at (0.946,0.315) {$e_1$}; \node[font=\scriptsize,above] at (0.315,0.946) {$e_2$};
\node[font=\scriptsize,left] at (-0.946,-0.315) {$e_3$}; \node[font=\scriptsize,left] at (-0.315,-0.946) {$e_4$};
\node[font=\scriptsize] at (0.42,-0.30) {$A$}; \node[font=\scriptsize] at (-0.40,0.24) {$B$};
\node at (0,-1.35) {\small (a) $t<t_1$};
\end{scope}
\begin{scope}[shift={(2.45,0)}]
\draw[gray!40,line width=1pt] (-1.1,0)--(1.1,0); \draw[gray!40,line width=1pt] (0,-1.1)--(0,1.1);
\draw[densely dashed,gray,line width=1pt] (0,0) circle (1);
\draw[very thick] (-0.617,-0.786) -- (-0.599,-0.767) -- (-0.580,-0.748) -- (-0.560,-0.728) -- (-0.541,-0.709) -- (-0.522,-0.691) -- (-0.508,-0.677) -- (-0.491,-0.660) -- (-0.471,-0.640) -- (-0.448,-0.617) -- (-0.435,-0.604) -- (-0.406,-0.576) -- (-0.390,-0.560) -- (-0.373,-0.542) -- (-0.354,-0.523) -- (-0.333,-0.503) -- (-0.311,-0.481) -- (-0.287,-0.457) -- (-0.287,-0.457) -- (-0.260,-0.431) -- (-0.232,-0.402) -- (-0.232,-0.402) -- (-0.201,-0.371) -- (-0.201,-0.371) -- (-0.167,-0.337) -- (-0.130,-0.300) -- (-0.130,-0.300) -- (-0.130,-0.300) -- (-0.090,-0.259) -- (-0.090,-0.259) -- (-0.046,-0.214) -- (-0.046,-0.214) -- (0.002,-0.165) -- (0.002,-0.165) -- (0.002,-0.165) -- (0.052,-0.110) -- (0.052,-0.110) -- (0.052,-0.110) -- (0.100,-0.051) -- (0.100,-0.051) -- (0.100,-0.051) -- (0.139,0.005) -- (0.139,0.005) -- (0.139,0.005) -- (0.171,0.053) -- (0.171,0.053) -- (0.197,0.094) -- (0.197,0.094) -- (0.218,0.129) -- (0.237,0.158) -- (0.237,0.158) -- (0.252,0.183) -- (0.266,0.205) -- (0.278,0.224) -- (0.288,0.241) -- (0.305,0.270) -- (0.319,0.293) -- (0.331,0.312) -- (0.345,0.335) -- (0.356,0.353) -- (0.368,0.373) -- (0.381,0.394) -- (0.396,0.418) -- (0.408,0.439) -- (0.422,0.463) -- (0.435,0.484) -- (0.449,0.507) -- (0.461,0.526) -- (0.473,0.546) -- (0.490,0.574) -- (0.502,0.595) -- (0.515,0.615) -- (0.527,0.636) -- (0.542,0.661) -- (0.553,0.678) -- (0.565,0.698) -- (0.580,0.724) -- (0.592,0.743) -- (0.605,0.765) -- (0.618,0.785);
\draw[very thick] (-0.785,-0.617) -- (-0.767,-0.598) -- (-0.751,-0.580) -- (-0.732,-0.560) -- (-0.714,-0.541) -- (-0.697,-0.523) -- (-0.676,-0.500) -- (-0.659,-0.482) -- (-0.640,-0.460) -- (-0.629,-0.449) -- (-0.604,-0.422) -- (-0.590,-0.407) -- (-0.575,-0.391) -- (-0.559,-0.374) -- (-0.541,-0.355) -- (-0.522,-0.335) -- (-0.502,-0.313) -- (-0.479,-0.289) -- (-0.455,-0.263) -- (-0.429,-0.235) -- (-0.429,-0.235) -- (-0.400,-0.205) -- (-0.400,-0.205) -- (-0.369,-0.171) -- (-0.335,-0.135) -- (-0.335,-0.135) -- (-0.297,-0.096) -- (-0.297,-0.096) -- (-0.256,-0.054) -- (-0.256,-0.054) -- (-0.256,-0.054) -- (-0.211,-0.008) -- (-0.211,-0.008) -- (-0.161,0.040) -- (-0.161,0.040) -- (-0.161,0.040) -- (-0.106,0.089) -- (-0.106,0.089) -- (-0.106,0.089) -- (-0.050,0.132) -- (-0.050,0.132) -- (-0.050,0.132) -- (0.002,0.168) -- (0.002,0.168) -- (0.047,0.197) -- (0.047,0.197) -- (0.086,0.220) -- (0.086,0.220) -- (0.119,0.240) -- (0.147,0.256) -- (0.147,0.256) -- (0.171,0.271) -- (0.192,0.283) -- (0.211,0.294) -- (0.242,0.312) -- (0.255,0.319) -- (0.278,0.332) -- (0.296,0.342) -- (0.324,0.358) -- (0.340,0.367) -- (0.365,0.382) -- (0.386,0.393) -- (0.409,0.406) -- (0.432,0.419) -- (0.454,0.432) -- (0.475,0.443) -- (0.498,0.456) -- (0.524,0.471) -- (0.544,0.482) -- (0.565,0.494) -- (0.587,0.506) -- (0.608,0.518) -- (0.635,0.534) -- (0.655,0.544) -- (0.678,0.558) -- (0.699,0.569) -- (0.720,0.581) -- (0.743,0.594) -- (0.763,0.605) -- (0.785,0.618);
\fill (0.377,0.388) circle (1.6pt);
\node at (0,-1.35) {\small (b) $t=t_1$};
\end{scope}
\begin{scope}[shift={(4.90,0)}]
\draw[gray!40,line width=1pt] (-1.1,0)--(1.1,0); \draw[gray!40,line width=1pt] (0,-1.1)--(0,1.1);
\draw[densely dashed,gray,line width=1pt] (0,0) circle (1);
\draw[very thick] (-0.775,-0.629) -- (-0.761,-0.614) -- (-0.748,-0.598) -- (-0.735,-0.583) -- (-0.722,-0.568) -- (-0.708,-0.552) -- (-0.695,-0.537) -- (-0.681,-0.521) -- (-0.668,-0.506) -- (-0.655,-0.491) -- (-0.642,-0.477) -- (-0.628,-0.461) -- (-0.615,-0.446) -- (-0.601,-0.430) -- (-0.588,-0.415) -- (-0.575,-0.401) -- (-0.561,-0.386) -- (-0.548,-0.370) -- (-0.535,-0.356) -- (-0.522,-0.341) -- (-0.508,-0.326) -- (-0.494,-0.311) -- (-0.481,-0.296) -- (-0.467,-0.281) -- (-0.453,-0.266) -- (-0.439,-0.250) -- (-0.427,-0.237) -- (-0.413,-0.221) -- (-0.399,-0.206) -- (-0.386,-0.192) -- (-0.372,-0.177) -- (-0.357,-0.161) -- (-0.345,-0.148) -- (-0.330,-0.132) -- (-0.317,-0.118) -- (-0.303,-0.103) -- (-0.289,-0.088) -- (-0.276,-0.074) -- (-0.261,-0.059) -- (-0.248,-0.045) -- (-0.233,-0.030) -- (-0.219,-0.016) -- (-0.204,-0.001) -- (-0.191,0.012) -- (-0.175,0.028) -- (-0.161,0.041) -- (-0.146,0.054) -- (-0.132,0.067) -- (-0.117,0.080) -- (-0.101,0.093) -- (-0.086,0.106) -- (-0.069,0.119) -- (-0.052,0.131) -- (-0.037,0.142) -- (-0.019,0.154) -- (-0.003,0.165) -- (0.014,0.176) -- (0.032,0.187) -- (0.048,0.197) -- (0.066,0.207) -- (0.083,0.218) -- (0.101,0.228) -- (0.119,0.238) -- (0.136,0.248) -- (0.154,0.258) -- (0.172,0.268) -- (0.190,0.277) -- (0.207,0.286) -- (0.225,0.294) -- (0.243,0.303) -- (0.261,0.310) -- (0.281,0.316) -- (0.301,0.314) -- (0.302,0.294) -- (0.297,0.277) -- (0.289,0.259) -- (0.279,0.239) -- (0.270,0.222) -- (0.260,0.203) -- (0.250,0.186) -- (0.241,0.170) -- (0.231,0.152) -- (0.220,0.134) -- (0.210,0.118) -- (0.199,0.099) -- (0.189,0.083) -- (0.178,0.066) -- (0.168,0.049) -- (0.157,0.032) -- (0.146,0.015) -- (0.134,-0.002) -- (0.123,-0.018) -- (0.112,-0.034) -- (0.099,-0.052) -- (0.087,-0.068) -- (0.074,-0.083) -- (0.061,-0.099) -- (0.048,-0.114) -- (0.034,-0.130) -- (0.022,-0.143) -- (0.008,-0.158) -- (-0.007,-0.174) -- (-0.020,-0.188) -- (-0.034,-0.202) -- (-0.048,-0.216) -- (-0.063,-0.232) -- (-0.077,-0.245) -- (-0.092,-0.261) -- (-0.106,-0.274) -- (-0.119,-0.288) -- (-0.134,-0.303) -- (-0.148,-0.316) -- (-0.162,-0.330) -- (-0.177,-0.345) -- (-0.191,-0.359) -- (-0.206,-0.373) -- (-0.221,-0.388) -- (-0.234,-0.401) -- (-0.250,-0.416) -- (-0.263,-0.429) -- (-0.278,-0.443) -- (-0.292,-0.457) -- (-0.308,-0.472) -- (-0.321,-0.485) -- (-0.336,-0.499) -- (-0.350,-0.513) -- (-0.366,-0.528) -- (-0.380,-0.541) -- (-0.394,-0.555) -- (-0.410,-0.570) -- (-0.423,-0.582) -- (-0.438,-0.597) -- (-0.452,-0.610) -- (-0.468,-0.624) -- (-0.483,-0.639) -- (-0.497,-0.652) -- (-0.512,-0.666) -- (-0.526,-0.679) -- (-0.541,-0.693) -- (-0.557,-0.707) -- (-0.572,-0.721) -- (-0.586,-0.734) -- (-0.600,-0.748) -- (-0.615,-0.761) -- (-0.629,-0.775);
\draw[very thick] (0.774,0.630) -- (0.755,0.619) -- (0.737,0.610) -- (0.719,0.600) -- (0.699,0.589) -- (0.681,0.580) -- (0.663,0.570) -- (0.643,0.560) -- (0.626,0.550) -- (0.608,0.541) -- (0.588,0.531) -- (0.568,0.522) -- (0.551,0.515) -- (0.530,0.507) -- (0.510,0.503) -- (0.496,0.515) -- (0.500,0.537) -- (0.508,0.555) -- (0.517,0.574) -- (0.526,0.592) -- (0.537,0.612) -- (0.546,0.630) -- (0.556,0.648) -- (0.568,0.668) -- (0.578,0.686) -- (0.588,0.703) -- (0.599,0.721) -- (0.609,0.739) -- (0.620,0.757) -- (0.630,0.774);
\node at (0,-1.35) {\small (c) $t_1<t<t^*$};
\end{scope}
\begin{scope}[shift={(7.35,0)}]
\draw[gray!40,line width=1pt] (-1.1,0)--(1.1,0); \draw[gray!40,line width=1pt] (0,-1.1)--(0,1.1);
\draw[densely dashed,gray,line width=1pt] (0,0) circle (1);
\draw[very thick] (-0.707,-0.707) -- (-0.712,-0.676) -- (-0.702,-0.643) -- (-0.688,-0.612) -- (-0.674,-0.587) -- (-0.656,-0.556) -- (-0.637,-0.526) -- (-0.622,-0.504) -- (-0.604,-0.478) -- (-0.583,-0.448) -- (-0.559,-0.415) -- (-0.545,-0.397) -- (-0.531,-0.378) -- (-0.498,-0.336) -- (-0.479,-0.313) -- (-0.460,-0.288) -- (-0.439,-0.263) -- (-0.416,-0.235) -- (-0.392,-0.207) -- (-0.366,-0.176) -- (-0.366,-0.176) -- (-0.338,-0.145) -- (-0.308,-0.111) -- (-0.277,-0.077) -- (-0.277,-0.077) -- (-0.243,-0.041) -- (-0.208,-0.004) -- (-0.208,-0.004) -- (-0.170,0.034) -- (-0.170,0.034) -- (-0.130,0.070) -- (-0.089,0.104) -- (-0.089,0.104) -- (-0.049,0.134) -- (-0.049,0.134) -- (-0.013,0.158) -- (0.020,0.179) -- (0.048,0.195) -- (0.073,0.209) -- (0.095,0.221) -- (0.115,0.231) -- (0.147,0.246) -- (0.173,0.257) -- (0.202,0.267) -- (0.233,0.272) -- (0.256,0.254) -- (0.252,0.221) -- (0.239,0.187) -- (0.222,0.150) -- (0.206,0.118) -- (0.195,0.100) -- (0.183,0.079) -- (0.153,0.029) -- (0.134,-0.002) -- (0.111,-0.035) -- (0.111,-0.035) -- (0.084,-0.072) -- (0.053,-0.109) -- (0.053,-0.109) -- (0.021,-0.145) -- (-0.013,-0.180) -- (-0.013,-0.180) -- (-0.045,-0.213) -- (-0.077,-0.244) -- (-0.077,-0.244) -- (-0.108,-0.274) -- (-0.137,-0.302) -- (-0.166,-0.328) -- (-0.193,-0.354) -- (-0.193,-0.354) -- (-0.219,-0.377) -- (-0.244,-0.400) -- (-0.268,-0.421) -- (-0.291,-0.441) -- (-0.312,-0.460) -- (-0.353,-0.494) -- (-0.371,-0.510) -- (-0.389,-0.524) -- (-0.422,-0.551) -- (-0.437,-0.563) -- (-0.465,-0.585) -- (-0.490,-0.604) -- (-0.513,-0.620) -- (-0.542,-0.641) -- (-0.567,-0.658) -- (-0.595,-0.675) -- (-0.621,-0.690) -- (-0.648,-0.702) -- (-0.678,-0.712) -- (-0.707,-0.707);
\draw[line width=1pt,fill=white] (-0.707,-0.707) circle (1.6pt); \draw[line width=1pt,fill=white] (0.707,0.707) circle (1.6pt);
\node at (0,-1.35) {\small (d) $t=t^*$};
\end{scope}
\begin{scope}[shift={(9.80,0)}]
\draw[gray!40,line width=1pt] (-1.1,0)--(1.1,0); \draw[gray!40,line width=1pt] (0,-1.1)--(0,1.1);
\draw[densely dashed,gray,line width=1pt] (0,0) circle (1);
\draw[very thick] (-0.237,-0.270) -- (-0.257,-0.264) -- (-0.273,-0.252) -- (-0.286,-0.236) -- (-0.293,-0.217) -- (-0.297,-0.197) -- (-0.297,-0.177) -- (-0.294,-0.157) -- (-0.290,-0.137) -- (-0.284,-0.117) -- (-0.276,-0.097) -- (-0.268,-0.079) -- (-0.259,-0.061) -- (-0.249,-0.043) -- (-0.238,-0.026) -- (-0.226,-0.008) -- (-0.214,0.007) -- (-0.202,0.023) -- (-0.188,0.039) -- (-0.174,0.054) -- (-0.160,0.068) -- (-0.143,0.083) -- (-0.128,0.096) -- (-0.113,0.108) -- (-0.097,0.119) -- (-0.079,0.131) -- (-0.061,0.141) -- (-0.043,0.150) -- (-0.023,0.159) -- (-0.006,0.165) -- (0.014,0.171) -- (0.034,0.174) -- (0.054,0.176) -- (0.074,0.176) -- (0.094,0.172) -- (0.112,0.163) -- (0.128,0.150) -- (0.138,0.132) -- (0.144,0.112) -- (0.146,0.092) -- (0.144,0.072) -- (0.141,0.052) -- (0.136,0.032) -- (0.129,0.012) -- (0.121,-0.007) -- (0.112,-0.026) -- (0.103,-0.042) -- (0.092,-0.061) -- (0.081,-0.077) -- (0.068,-0.094) -- (0.056,-0.110) -- (0.042,-0.126) -- (0.028,-0.140) -- (0.014,-0.153) -- (-0.001,-0.168) -- (-0.017,-0.181) -- (-0.032,-0.193) -- (-0.049,-0.206) -- (-0.065,-0.217) -- (-0.083,-0.228) -- (-0.101,-0.238) -- (-0.119,-0.247) -- (-0.139,-0.256) -- (-0.157,-0.262) -- (-0.177,-0.268) -- (-0.197,-0.271) -- (-0.217,-0.272) -- (-0.237,-0.270) -- cycle;
\node at (0,-1.35) {\small (e) $t>t^*$};
\end{scope}

\end{tikzpicture}
\caption{Example~\ref{ex:conic} with $b=1$, in the disc model $x\mapsto x/(5+|x|)$ of $X_\Delta=\RP^2$; the dashed circle is $\partial X$, with antipodal points identified. (a) Small $t$: the lines $A$ and $B$ and the ends $e_1,\dots,e_4$. (b) The saddle at $t_1\approx0.486$ (dot). (c) One line has both ends in $\cO_{++}$, near $y=x_2/x_1=1$. (d) At $t^*=0.5$ this line has gone to infinity: the parabola $\overline{C_{t^*}}$ is tangent to $\partial X$ at $y=1$ (circles). (e) After the annihilation: the ellipse.}
\label{fig:conic}
\end{figure}

\section{Global information}\label{sec:global}

We obtain the components needed at saddles and creations (Remark~\ref{rem:local}) by following a gradient-like field of $\tau$ downwards to a fixed bottom, the curve $C_{t_0}$ at a level $t_0$ before the first event together with the small curves born later; unlike in \cite[\S\S4.2--4.3]{SH97}, all separatrices are followed down to this common bottom. Throughout this section Assumption~\ref{ass:morse} holds, $E_\cA(c)\neq0$, and $t_0\in(0,t_1)$, where $t_1$ is the first event.

\subsection{The swept surface cut open at infinity}

Let $S$, $\tau$ and $B_e$ be as in \S\ref{sec:events}. $S$ is transverse to each $D_e^\circ\times(0,1]$: at a boundary event its normal vector $(0,\partial_zg^e_t,\alpha)$ in the coordinates $(y,z,t)$ of $U_e$ is not parallel to $(0,1,0)$, and elsewhere $\overline{C_t}$ is transverse to $D_e^\circ$. So $B_e$ is a smooth curve on $S$ separating the sides $z>0$ and $z<0$.

\begin{definition}\label{def:Shat}
$\hat S$ is the compact surface obtained from $S\cap\tau^{-1}[t_0,1]$ by cutting it open along the curves $B_e$. Its boundary consists of the \emph{bottom} $\hat C_{t_0}$, the \emph{top} $\hat C_1$ and the \emph{side} $\hat B$, which contains two copies of each $B_e$, one for each side of $D_e^\circ$. We put $\hat C_t=\tau^{-1}(t)\subset\hat S$.
\end{definition}

For a non-event $t$, $\hat C_t$ is $\overline{C_t}$ cut at its points on $\partial X$: ovals are circles and lines are arcs whose endpoints are their ends $(p,Q)$, so $\hat B\cap\hat C_t$ is the set of ends of $C_t$. We orient $\hat C_t$ as $C_t$; this is continuous on $\hat S$ off the critical points, since $z^{m_e}$ has constant sign on each side of $D_e^\circ$.

\begin{lemma}[critical points]\label{lem:crit}
The critical points of $\tau$ in $\hat S\setminus\hat B$ are the interior event points, of index $0$ (birth of an oval), $1$ (saddle) and $2$ (death of an oval) (Lemma~\ref{lem:morse}(i)). The critical points of $\tau|_{\hat B}$ are the two points over each boundary event point $(y_0,t^*)$, one on each copy of $B_e$; near each of them $(y,t)$ are coordinates on $\hat S$, in which $\hat B$ is the curve $t=T(y)$ of Lemma~\ref{lem:morse}(ii). (i) At a creation, $T$ has a minimum at $y_0$; on the side $Q_m$, $\hat S=\{t\ge T(y)\}$ and the point is a local minimum of $\tau$, the birth of $\lambda$; on the side $Q_o$, $\hat S=\{t\le T(y)\}$ and the point is a \emph{boundary saddle}: near it a level just below $t^*$ is one arc, and a level just above consists of two arcs ending on $\hat B$. (ii) At an annihilation, $T$ has a maximum at $y_0$; on the side $Q_m$, $\hat S=\{t\le T(y)\}$ and the point is a local maximum of $\tau$, the death of $\lambda$; on the side $Q_o$, $\hat S=\{t\ge T(y)\}$ and the point is a boundary saddle, joining two arcs ending on $\hat B$.
\end{lemma}

\begin{proof}
Near $p=(y_0,0)$ the curve $\overline{C_t}$ is the graph $z=\phi_t(y)$ (proof of Proposition~\ref{prop:boundary}), so $(y,t)$ are coordinates on $S$, in which $B_e=\{\psi^e_t(y)=0\}=\{t=T(y)\}$. The middle arc $\lambda$ lies over the interval between the two roots of $\psi^e_t$ near $y_0$, which is $\{y:\ t>T(y)\}$ at a creation and $\{y:\ t<T(y)\}$ at an annihilation, and the outer pieces on the other side.
\end{proof}

\subsection{A gradient-like vector field}\label{ssec:field}

All descending trajectories are taken for one vector field $W$, built from a normal lift. Let $g_t$ be a Riemannian metric on $\R^2$, smooth in $(x,t)$, which near each $D_e^\circ$ is, in the chart $U_e$, a positive definite quadratic form in $d\log|y|,d\log|z|$ with coefficients smooth in $(y,z,t)$ up to $z=0$. Such metrics exist, since $S$ stays away from the torus-fixed points (Lemma~\ref{lem:fixed}): interpolate between a flat metric in $\log|x|$ near $\partial X$ and the Euclidean metric near the axes; the metric of \S\ref{ssec:sepcomp} is of this kind if no $\nu_e$ has a positive coordinate. Away from the critical points let $V=\bigl(-\partial_tf_t\,\nabla_gf_t/|\nabla_gf_t|_g^2,\ 1\bigr)$ on $S^\circ$: points of $C_t$ move orthogonally to $C_t$, and $d\tau(V)=1$.

\begin{proposition}[an adapted field]\label{prop:W}
There is a smooth vector field $W$ on $\hat S$, tangent to $\hat B$, such that:
\begin{enumerate}[label=(\roman*),leftmargin=*,itemsep=1pt]
\item $W$ vanishes exactly at the critical points of Lemma~\ref{lem:crit}, and $d\tau(W)>0$ elsewhere;
\item near a boundary critical point, in the coordinates $u=y-y_0$, $v=t-T(y)$, $W=(T'(y),v)$ if $\hat S=\{v\ge0\}$ and $W=(T'(y),-v)$ if $\hat S=\{v\le0\}$;
\item away from the critical points, outside small neighbourhoods of the boundary critical points and outside thin bands $\tau^{-1}[a,b]$ of levels without events, one just below each event and disjoint from these neighbourhoods, $W$ is a positive multiple of $V$; on the band below $t^*$ it is a positive multiple of $V+\theta\chi'(t)Y$, where $\chi$ increases from $0$ near $a$ to $1$ near $b$, $Y$ is tangent to the levels and vanishes near $\hat B$, and $\theta=0$ unless, for $\theta=0$, a descending separatrix of an event point of $t^*$ converges to a lower saddle;
\item no trajectory of $W$ in $\hat S\setminus\hat B$ joins two saddles (interior or boundary); where $\theta>0$ is needed, all small $\theta>0$ outside a finite set will do.
\end{enumerate}
\end{proposition}

\begin{proof}
With $\phi=|\nabla_gf_t|_g^2/(|\nabla_gf_t|_g^2+(\partial_tf_t)^2)$, $\phi V$ is the gradient of $\tau$ for the metric induced on $S^\circ$ by $g_t+dt^2$; it is smooth and vanishes exactly at the interior critical points. Near $D_e^\circ$ let $\omega=(\log|y|,\log|z|)$ and $G$ the matrix of $g_t$ in $\omega$. As $f_t=z^{m_e}g^e_t$, on $S$ the field $V$ is the normal lift of $g^e_t$, whose $\omega$-component is $-\partial_tg^e_t\,G^{-1}\gamma/(\gamma^TG^{-1}\gamma)$ with $\gamma=(y\partial_yg^e_t,z\partial_zg^e_t)$. At $z=0$, $\gamma=(y(\psi^e_t)',0)$, so the $z$-component $z\,\dot\omega_2$ of $V$ vanishes and its $y$-component is $-\partial_tg^e_t/\partial_yg^e_t$: away from the boundary events $V$ extends smoothly to $\hat B$, tangent to it. Near the boundary critical points take the models (ii), where $\tau=v+T(y_0+u)$ and $d\tau(W)=T'(y)^2+|v|$. Gluing $\phi V$, $V$ and the models by a partition of unity preserves $d\tau(W)>0$ and the tangency to $\hat B$.

Connections are removed inductively over the events in increasing order. Suppose that no trajectory joins two saddles below the event $t^*$, and let $[a,b]$ be a band of levels between $t^*$ and the previous event, disjoint from the model neighbourhoods. The flow of $W/d\tau(W)$ identifies $\tau^{-1}[a,b]$ with $\mathcal F\times[a,b]$, $\mathcal F=\hat C_b$, so that $W/d\tau(W)=\partial_t$. Let $q_1,\dots,q_n\in\mathcal F$ be the points where the descending separatrices of the event points of $t^*$ cross $\hat C_b$, and $r_1,\dots,r_m$ those where the ascending separatrices in $\hat S\setminus\hat B$ of the saddles below $a$ cross $\hat C_a$; a descending separatrix of $t^*$ converges to a lower saddle if and only if it passes through some $r_j$. On the band replace $\partial_t$ by $\partial_t+\theta\chi'(t)Y$, with $Y$ a vector field on $\mathcal F$ that vanishes near the endpoints of its arcs but not at the $q_i$. The holonomy from $\hat C_b$ to $\hat C_a$ is the flow $\varphi_\theta$ of $-Y$ at time $\theta$. For $\theta$ in some interval $[0,\theta_0]$ every $\varphi_\theta(q_i)$ moves monotonically in a coordinate of $\mathcal F$, so each equation $\varphi_\theta(q_i)=r_j$ has at most one solution $\theta\in(0,\theta_0]$, and every other $\theta$ removes the connections at $t^*$. The band lies above all earlier events, so the earlier descending separatrices do not change, and two points of one level are never joined, since $\tau$ increases strictly along non-constant trajectories. On the band the new field $d\tau(W)(\partial_t+\theta\chi'(t)Y)$ is smooth.
\end{proof}

In the models (ii) trajectories leave a minimum and enter a maximum; exactly one trajectory in $\hat S\setminus\hat B$ converges from below to the boundary saddle of a creation, its \emph{descending separatrix}, and the two branches of $\hat B$ above it, carrying the new ends, leave it; the two branches of $\hat B$ below the boundary saddle of an annihilation enter it, and exactly one trajectory in $\hat S\setminus\hat B$, its \emph{ascending separatrix}, leaves it upwards. An interior saddle has two descending and two ascending separatrices. The proof of (iv) is the argument of \cite{Sma61}, see \cite{Mil65}, made explicit; unlike in \cite{BNR16}, $\tau$ has no critical point at a boundary event. But (iv) is not automatic:

\begin{example}[a saddle connection]\label{ex:connection}
Let $f_t=P(x_1)+(\tfrac32-x_1)x_2^2+\tfrac1{10}x_2^3-t$ with $P(x_1)=\tfrac15x_1^3-\tfrac9{10}x_1^2+\tfrac65x_1$ ($h=1$ on the constant term and $0$ elsewhere). The points $p=(1,0)$ and $q=(2,0)$ are saddles of $t=\frac12$ and $t=\frac25$, with Hessians $\operatorname{diag}(-\frac35,1)$, $\operatorname{diag}(\frac35,-1)$ and $\partial_tf_t=-1$. As $\partial_2f_t$ vanishes on the axis $x_2=0$, the axis is invariant under $V$ for every metric diagonal in $(x_1,x_2)$, and on it $S$ is the graph $t=P(x_1)$, which decreases on $[1,2]$; so a descending separatrix of $p$ runs along the axis into $q$, also for every generic lifting close to this one, as no exponent has $a_2=1$; Lemma~\ref{lem:axisconn} gives the labels. The side from which it reaches $q$ is local: for the heights $h(a)-Ka_1$ (the rescaling $r_1=t^Kx_1$), $p$ becomes $(2^{-K},0)$, to the right of $q=(2(\frac25)^K,0)$ if $K\ge4$, but the separatrix still reaches $q$ from the left (\S\ref{app:regression}).
\end{example}

\subsection{Landing points}\label{ssec:landing}

Fix $\delta>0$ smaller than half the distance between any two of the numbers $t_0$, $1$ and the events; a level $t$ is \emph{regular} if its distance to every event exceeds $\delta$. Let $m_1,\dots,m_k$ be the local minima of $\tau$ on $\hat S$, i.e.\ the births of ovals and the points of Lemma~\ref{lem:crit}(i) on the sides $Q_m$, and $B_{m_j}$ the component of $\tau^{-1}(\tau(m_j)+\delta)$ near $m_j$: a small oval, resp.\ the line $\lambda$ just after its creation. Removing from $\hat S$ the discs, resp.\ half-discs, that the $B_{m_j}$ bound below, we obtain a surface $\hat S'$ with the \emph{bottom} $B=\hat C_{t_0}\sqcup B_{m_1}\sqcup\dots\sqcup B_{m_k}$, a finite union of oriented circles and arcs; $B_{<t}$ is the union of $\hat C_{t_0}$ and the $B_{m_j}$ with $\tau(m_j)<t$.

\begin{lemma}[descending trajectories]\label{lem:descend}
Let $x\in\hat S'$ with $W(x)\neq0$, and follow $-W$ from $x$. Either the trajectory reaches $B$, at a point $\Lambda(x)$, the \emph{landing point} of $x$; or it converges to a saddle, which happens if and only if $x$ lies on an ascending separatrix of that saddle, or on a branch of $\hat B$ above the boundary saddle of a creation. The descending separatrices in $\hat S\setminus\hat B$ of all saddles, i.e.\ the two of an interior saddle and the one of the boundary saddle of a creation, land on $B$, in finite time from any level below the saddle.
\end{lemma}

\begin{proof}
As $W$ is tangent to $\hat B$, the trajectory can leave the compact surface $\hat S'$ only through $B$. Otherwise its limit set is a non-empty connected set of zeros of $W$, as $\tau$ is monotone along it; as the zeros are finitely many, it converges to one of them, neither a maximum, since $\tau$ decreases, nor a minimum, since these were removed: a saddle, with $x$ on an ascending separatrix. A descending separatrix in $\hat S\setminus\hat B$ stays there and does not converge to another saddle (Proposition~\ref{prop:W}(iv)), so it reaches the transverse section $B$ in finite time.
\end{proof}

\begin{definition}\label{def:arcs}
Let $t$ be a regular level. A point of $\hat C_t$ is \emph{special} if it does not land, i.e.\ if it lies on an ascending separatrix of a saddle below $t$ (or is an end on $Q_o$ created at a creation below $t$). The \emph{arcs} of $\hat C_t$ are the components of $\hat C_t$ minus its special points and its ends. Let $L_t\subset B$ be the set of landing points of the descending separatrices in $\hat S\setminus\hat B$ of the saddles below $t$. The \emph{bottom intervals} at level $t$ are the components of $B_{<t}$ minus $L_t$ and the endpoints of its arcs.
\end{definition}

An ascending separatrix meets every level at most once, so $\hat C_t$ has finitely many special points, and arcs and bottom intervals are open intervals or circles.

\begin{proposition}[arcs and bottom intervals]\label{prop:landing}
Let $t$ be a regular level. Then $\Lambda$ maps each arc of $\hat C_t$ homeomorphically, preserving the orientation, onto a bottom interval at level $t$, different arcs onto different intervals. A bottom interval is not the image of an arc if and only if the trajectories of $W$ from it converge to a maximum of $\tau$ below $t$, and is then called \emph{dead}.
\end{proposition}

\begin{proof}
Let $J$ be a bottom interval. Following $W$, a point of $J$ either reaches $\hat C_t$ or converges to a zero of $W$ below $t$ (as in the proof of Lemma~\ref{lem:descend}), which is not a saddle, since $J$ avoids $\hat B$ and $L_t$, hence a maximum. Both alternatives are open conditions, so one of them holds on all of $J$. In the first case the flow carries $J$ homeomorphically onto a connected subset $\alpha$ of $\hat C_t$ without special points and ends. If $x\notin\alpha$ is a limit of points of $\alpha$ and is neither special nor an end, then $\Lambda(x)$ lies in the closure of $J$ by continuity of the flow at $x$, and not in $L_t$ or $\hat B$, hence in $J$, a contradiction. So $\alpha$ is an arc and $\Lambda|_\alpha$ is the inverse of the flow. Every point of an arc lands on a bottom interval (Lemma~\ref{lem:descend}), and trajectories are disjoint; this gives the bijection. The flow preserves the orientation of the levels, which is continuous.
\end{proof}

The bottom is fixed once and for all; the events are recorded by the points $L_t$ at which it is cut. For small $t_0$, $\hat C_{t_0}$ can be read on $T(f,h)$. For a quadrant $\epsilon$ let $D_\epsilon=F^+_h-F^-_h$ for the signs $\sigma_\epsilon$, so that $\Gamma_\epsilon=D_\epsilon^{-1}(0)$. For generic $h$ there is a polyhedral error bound: a constant $c\in(0,1]$ with
\begin{equation}\label{eq:errorbound}
|D_\epsilon(w)|\ge c\,\min\{1,\operatorname{dist}(w,\Gamma_\epsilon)\}\qquad\text{for all }w\in\R^2\text{ and all }\epsilon.
\end{equation}
Indeed, if a sign class is empty, $\Gamma_\epsilon=\emptyset$ and $D_\epsilon=\pm\infty$. Otherwise the domains of linearity of $F^+_h$ and $F^-_h$ have a common refinement into finitely many polyhedra $\sigma$ covering $\R^2$, on each of which $D_\epsilon$ is a non-constant affine function $\ell_\sigma=\langle a-b,w\rangle-h(a)+h(b)$, $a\in\cA^+$, $b\in\cA^-$. If $\ell_\sigma$ has no zero on $\sigma$, say $\ell_\sigma>0$, it attains a positive minimum on the polyhedron $\sigma$; otherwise Hoffman's error bound \cite{Hof52} gives $C_\sigma$ with $\operatorname{dist}(w,\sigma\cap\ell_\sigma^{-1}(0))\le C_\sigma|\ell_\sigma(w)|$ for $w\in\sigma$, and $\sigma\cap\ell_\sigma^{-1}(0)\subset\Gamma_\epsilon$. Let $c$ be the minimum of $1$, these minima and the $1/C_\sigma$.

\begin{lemma}[the bottom]\label{lem:bottom}
Let $h$ be generic, $c$ as in \eqref{eq:errorbound}, $K_0=\log(|\cA|M/m)$, where $M$ and $m$ are the largest and the smallest $|c_a|$, and for a component $\Gamma'$ of $T(f,h)$ let $\Gamma'_\epsilon=\{w\in\Gamma_\epsilon:\mathcal E(w,\epsilon)\in\Gamma'\}$. Let $\eta\in(0,1)$ be smaller than half the distance between $\Gamma'_\epsilon$ and $\Gamma''_\epsilon$ for any components $\Gamma'\neq\Gamma''$ and any $\epsilon$, and than half of $|w^e_1-w^{e'}_1|$ for any two crossing points $(\epsilon_1e^{w^e_1},0)\neq(\epsilon_1e^{w^{e'}_1},0)$ of $T(f,h)$ on one half-axis, and symmetrically on the other axis (Lemma~\ref{lem:ends}). If $t_0$ is small enough, then:
\begin{enumerate}[label=(\roman*),leftmargin=*,itemsep=1pt]
\item every $b\in C_{t_0}\cap\cO_\epsilon$ satisfies $\operatorname{dist}(\Log_{t_0}(b),\Gamma_\epsilon)\le K_0/(c\log(1/t_0))<\eta$;
\item every component $Z$ of $C_{t_0}$ lies in the neighbourhood $U(\Gamma_Z)$, defined in the proof, of exactly one component $\Gamma_Z$ of $T(f,h)$, and $Z\mapsto\Gamma_Z$ is a bijection preserving the types and the ends; in particular the component through $b\in C_{t_0}\cap\cO_\epsilon$ corresponds to the $\Gamma'$ for which $\Gamma'_\epsilon$ contains the point of $\Gamma_\epsilon$ nearest to $\Log_{t_0}(b)$;
\item no saddle lies below $t_0$, so every component of $\hat C_{t_0}$ is one arc and its own bottom interval.
\end{enumerate}
\end{lemma}

\begin{proof}
Let $s=\log(1/t_0)$. (i) For $b\in C_{t_0}\cap\cO_\epsilon$ and $w=\Log_{t_0}(b)$, in the notation of the proof of Proposition~\ref{prop:limit} $P(w)=N(w)$, and the bounds given there, with $t=e^{-s}$, yield $|D_\epsilon(w)|\le K_0/s$. By \eqref{eq:errorbound}, $\min\{1,\operatorname{dist}(w,\Gamma_\epsilon)\}\le K_0/(cs)<\eta<1$ for small $t_0$.

(ii) Let $U(\Gamma')$ consist of the points $x$ of the open quadrants with $\operatorname{dist}(\Log_{t_0}(x),\Gamma'_{\sgn x})<\eta$, the points $(x_1,0)$ with $\sgn x_1=\epsilon_1$ and $\bigl|\log|x_1|/s-w^e_1\bigr|<\eta$ for a crossing point $(\epsilon_1e^{w^e_1},0)$ of $\Gamma'$, and symmetrically on $\{x_1=0\}$. These open sets are pairwise disjoint by the choice of $\eta$ and Lemma~\ref{lem:ends}, and they cover $C_{t_0}$: off the axes by (i), on the axes by the same argument for $f_{t_0}(x_1,0)$ and $f_{t_0}(0,x_2)$, whose positive tropical curves are the crossing points. So every component $Z$ lies in exactly one $U(\Gamma_Z)$. Every $U(\Gamma')$ meets $C_{t_0}$ (Proposition~\ref{prop:limit}(i)), and $C_{t_0}$ and $T(f,h)$ have the same number of components (Theorem~\ref{thm:viro-plane}); so $Z\mapsto\Gamma_Z$ is bijective. If $\Gamma'$ is an oval, $U(\Gamma')$ is bounded, so $\Gamma_Z=\Gamma'$ forces $Z$ to be an oval; as the numbers of ovals agree, ovals correspond to ovals and lines to lines. By one-dimensional patchworking for the edge polynomials, distinct ends $(p,Q)$ of $C_{t_0}$ correspond to distinct rays of the $\Gamma_Q$ dual to bichromatic edges of $\cS_h$ on outer edges, and far out the branch at $(p,Q)$ lies within $\eta$ of its ray; so the two ends of a line $Z$ correspond to the two such rays of $\Gamma_Z$. The last assertion follows from (i) and the choice of $\eta$. (iii) A saddle below $t_0$ would be an event below $t_1$.
\end{proof}

\begin{lemma}[positions]\label{lem:positions}
Let $h$ be generic, $s=\log(1/t_0)$, and on $\cO_\epsilon$ use the coordinates $u=(\log|x_1|,\allowbreak\log|x_2|)$, so that $\Log_{t_0}=u/s$. There are $R,\kappa>0$ such that for $t_0$ small enough and every component $Z$ of $C_{t_0}$:
\begin{enumerate}[label=(\alph*),leftmargin=*,itemsep=1pt]
\item if $v$ is a vertex of $(\Gamma_Z)_\epsilon$, dual to the triangle with vertices $a,a',a''$, $a$ the vertex of the odd sign, then $u(Z)\cap B(sv,2R)$ is one arc, along which $\langle a'-a'',u\rangle$ is strictly monotone;
\item if $e$ is an edge of $(\Gamma_Z)_\epsilon$, then the part of $u(Z)$ in the $\kappa$-neighbourhood of $se$ and outside the balls $B(sv,R)$ around the endpoints $v$ of $e$ is one arc, the graph of a function over $se$;
\item these arcs cover $u(Z\cap\cO_\epsilon)$ and follow each other along $Z$ in the order of their vertices and edges along $\Gamma_Z$, where $Z$ is oriented as in \S\ref{sec:events} and $(\Gamma_Z)_\epsilon$ with the positive region on its left if $\epsilon_1\epsilon_2=1$ and on its right if $\epsilon_1\epsilon_2=-1$ (the chart $u$ has orientation sign $\epsilon_1\epsilon_2$); across an axis, the arcs of two rays glued at a crossing point (Lemma~\ref{lem:ends}) join there.
\end{enumerate}
\end{lemma}

\begin{proof}
On $\cO_\epsilon$, $f_{t_0}=\sum_a\sigma_\epsilon(a)e^{\ell_a(u)}$ with $\ell_a(u)=\langle a,u\rangle-sh(a)+\log|c_a|$, and by Lemma~\ref{lem:bottom}, $u(Z\cap\cO_\epsilon)$ lies in the $\kappa$-neighbourhood of $s(\Gamma_Z)_\epsilon$ for $\kappa=K_0/c+1$. Constants are chosen in this order: $c$, $K_0$, $\kappa$ and $G,\gamma>0$, $C$ below depend only on $\cA$, $h$ and the $c_a$, uniformly over the quadrants, vertices and edges; then $R$ is taken large and finally $s$ large compared with $R$.

(a) Let $\tau$ be the triangle. As $h$ is generic, the points of $\cA\setminus\tau$ lie strictly above the plane through the lifted triangle, so there is $G>0$ with $\ell_b(u)\le\max_{\tau}\ell(u)-sG+C|u-sv|+C$ for $b\notin\tau$. Hence on $B(sv,2R)$, for $s$ large compared with $R$, $\sigma_\epsilon(a)e^{-\ell_a}f_{t_0}=1-e^X-e^Y+\rho$ in the affine coordinates $X=\ell_{a'}-\ell_a$, $Y=\ell_{a''}-\ell_a$, where $\rho$ and its first derivatives are $O(e^{-sG/2})$. The curve $e^X+e^Y=1$ bounds a convex region and is asymptotic to the legs $\{Y=0,\,X<0\}$ and $\{X=0,\,Y<0\}$, which are, up to translations of size $O(1)$, the dilated edges of $\Gamma_\epsilon$ at $v$ dual to $[a,a'']$ and $[a,a']$; along it, $X-Y=\langle a'-a'',u\rangle+\mathrm{const}$ has derivative $e^X+e^Y=1$ in the direction of the tangent vector $(e^Y,-e^X)$, of length at most $1$. As the gradient of $1-e^X-e^Y$ has length at least $1/2$ near the curve, the zero set in the ball is a $C^1$-small perturbation of this curve; for $R$ large this gives (a).

(b) Let $e$ be dual to $[a,a']$. For $b\neq a,a'$ the function $F_h(w)-\langle b,w\rangle+h(b)$ is affine along $e$ and positive on its relative interior, hence at least $\gamma\min\{1,\operatorname{dist}(w,\partial e)\}$ on $e$. So in the region of (b), for $s\ge R$, every term $e^{\ell_b}$, $b\neq a,a'$, is at most $e^{-\gamma R+C}\max(e^{\ell_a},e^{\ell_{a'}})$, and these two are comparable; thus $\sigma_\epsilon(a)e^{-\ell_a}f_{t_0}=1-e^X+\rho$ with $X=\ell_{a'}-\ell_a$ and $\rho=O(e^{-\gamma R})$ in the $C^1$ norm, and the zero set is a graph $X=g$ over the direction of $e$ with $g,g'=O(e^{-\gamma R})$.

(c) For $R$ large, the balls $B(sv,2R)$ and the regions of (b) cover the $\kappa$-neighbourhood of $s(\Gamma_Z)_\epsilon$ and overlap in annuli $R<|u-sv|<2R$ where both descriptions apply; the orientations agree by Proposition~\ref{prop:limit}(ii), since $\{f_{t_0}>0\}$ lies on the left of $Z$ in the coordinates $x$, hence on the left, resp.\ right, in the coordinates $u$. Near an axis, the arc of a ray ending on it is a graph over the axis direction, and $C_{t_0}$ crosses the axis transversally at the corresponding root of $f_{t_0}(x_1,0)$, resp.\ $f_{t_0}(0,x_2)$.
\end{proof}

By Lemmas~\ref{lem:bottom} and~\ref{lem:positions}, landing points on $\hat C_{t_0}$ can be compared without tracing $C_{t_0}$:

\begin{definition}[labels]\label{def:labels}
A landing point $b\in\hat C_{t_0}\cap\cO_\epsilon$ on the component $Z$ is labelled by a \emph{cell}, the vertex $v$ of $(\Gamma_Z)_\epsilon$ if $|u(b)-sv|<3R/2$ and otherwise the edge $e$ whose region in Lemma~\ref{lem:positions}(b) contains $u(b)$, and a \emph{coordinate}, $\pm\langle a'-a'',u(b)\rangle$, resp.\ the projection of $u(b)$ to $se$, with the sign for which it increases along $Z$. A landing point $b=(b_1,0)$ is labelled by the crossing point $(\sgn(b_1)e^{w^e_1},0)$ of $T(f,h)$ with $\bigl|\log|b_1|/s-w^e_1\bigr|<\eta$, symmetrically on the other axis, and a landing point on a slice $B_m$ by its angle around $m$, resp.\ its coordinate $y$. Labels on one component $\Gamma_Z$ are ordered by cell along $\Gamma_Z$, an axis label lying between the cells of the two rays glued there, and then by coordinate.
\end{definition}

This order is the order of the landing points along $Z$: outside $B(sv,R)$ the branches of $e^X+e^Y=1$ in the proof of Lemma~\ref{lem:positions}(a) have tangents exponentially close in $R$ to the directions $d$ of the legs and lie within $O(1)$ of the rays $sv+\R_{\ge0}d$, so $|u-sv|$ increases strictly along them, also after a $C^1$-small perturbation; hence $u(Z)$ crosses the circle of radius $3R/2$ around $sv$ once, transversally, on each edge at $v$, and the regions of the cells cut $u(Z\cap\cO_\epsilon)$ into consecutive arcs, ordered as in Lemma~\ref{lem:positions}(c). The cell determines the component (Lemma~\ref{lem:bottom}(ii)); an axis label exists for small $t_0$, is unique by the choice of $\eta$, and lies between the cells of the two rays glued at the crossing point (Lemma~\ref{lem:positions}(c)). Thus the bottom intervals of $\hat C_{t_0}$ become intervals of $T(f,h)$ between labels.

\subsection{The state and its update}

\begin{definition}[state]\label{def:state}
The \emph{state} at a regular level $t$ lists the components of $\hat C_t$, each as the sequence of the bottom intervals $\Lambda(\alpha)$ of its arcs $\alpha$, in the order of the orientation, cyclic for an oval and linear for a line, together with the initial and the terminal end of every line.
\end{definition}

The state determines $\Sigma_t$ (Proposition~\ref{prop:landing}); at $t_0$ it consists of the components of $\hat C_{t_0}$, each a sequence with one term (Lemma~\ref{lem:bottom}(iii)). The \emph{successor} of a term is the next term, cyclically for an oval, or the terminal end after the last term of a line. For a bottom interval $J$ and $b\in J$, \emph{cutting $J$ at $b$} replaces $J$ in its sequence by the components of $J\setminus\{b\}$: if $J$ is an interval, by the intervals $J^-_b$ and $J^+_b$ that end, resp.\ start, at $b$, in this order, with $J^+_b$ the successor of $J^-_b$; if $J$ is a circle, by the single interval $J\setminus\{b\}$, which ends and starts at $b$, is denoted by $J^-_b$ and by $J^+_b$, is its own successor and occurs once.

\begin{theorem}[update rules]\label{thm:update}
Let $t^*$ be an event, and $t_-<t^*<t_+$ regular levels such that $t^*$ is the only event in $[t_-,t_+]$. The state at $t_+$ is obtained from the state at $t_-$ by applying, at each event point of $t^*$, the following rule for its type, one after another in any order, each to the current state; the landing points $b,b_1,b_2$ do not depend on the state, and all other components are unchanged. A creation is treated by (b) and (d), an annihilation by (e).
\begin{enumerate}[label=(\alph*),leftmargin=*,itemsep=2pt]
\item \emph{Saddle at $p$.} Let $b_1,b_2$ be the landing points of the two descending separatrices of $(p,t^*)$. Cut at $b_1$ and at $b_2$, and reconnect: $J^+_{b_2}$ becomes the successor of $J^-_{b_1}$, and $J^+_{b_1}$ that of $J^-_{b_2}$, where $J^\mp_{b_i}$ are the intervals that end, resp.\ start, at $b_i$ after both cuts.
\item \emph{Minimum $m$.} Add $B_m$ as a new oval, resp., on the side $Q_m$ of a creation, as a new line with ends $(p_1,Q_m),(p_2,Q_m)$, oriented as in \S\ref{sec:events}.
\item \emph{Death at a maximum $M$.} Delete the oval that contains the bottom interval of $\Lambda(x)$, where $x$ is any non-special point near $M$ with $\tau(x)<t^*$.
\item \emph{Creation at $p$, side $Q_o$.} Let $b$ be the landing point of the descending separatrix of the boundary saddle. Cut at $b$ and separate there: the successor of $J^-_b$ becomes the new terminal end $(p_a,Q_o)$, and $J^+_b$ follows the new initial end $(p_b,Q_o)$, with $p_a,p_b$ as in Proposition~\ref{prop:boundary}; an oval thus becomes a line, and a line two lines.
\item \emph{Annihilation at $p$.} Delete the line whose ends are $(p_1,Q_m)$ and $(p_2,Q_m)$. The ends $(p_1,Q_o),(p_2,Q_o)$ are the terminal end of a line $L$ and the initial end of a line $L'$; join them, i.e.\ let the first term of $L'$ be the successor of the last term of $L$. If $L\neq L'$ they become one line; if $L=L'$ its sequence becomes cyclic, an oval.
\end{enumerate}
In (a) and (d), the components $\gamma_1,\gamma_2$, resp.\ $\gamma$, of Propositions~\ref{prop:interior} and~\ref{prop:boundary}, for the curve produced by the previous replacements (Remark~\ref{rem:local}), are those that contain the bottom intervals of $b_1,b_2$, resp.\ $b$, in the current state.
\end{theorem}

\begin{proof}
Between events the flow of $W/d\tau(W)$ carries $\hat C_t$ onto $\hat C_{t'}$, preserving special points, ends, orientations and landing points, so the state is constant. At $t^*$ choose disjoint \emph{Morse blocks} $N_1,\dots,N_k$ around the critical points of level $t^*$ other than the local minima (which are not in $\hat S'$): compact neighbourhoods in which the models of \S\ref{sec:events} and Proposition~\ref{prop:W}(ii) hold, bounded by arcs of trajectories and by parts of the levels $t^*\mp c$, $c>0$ small, through which $W$ enters, resp.\ exits. The blocks of the saddles contain no special point of the level $t^*-c$, which is possible since only finitely many ascending separatrices of lower saddles cross that level, none converging to a saddle of level $t^*$ (Proposition~\ref{prop:W}(iv)). Let $E_i$ be the union of the trajectory segments of $W$ in $\hat S'\cap\tau^{-1}[t_-,t_+]$ that meet $N_i$; for small blocks the $E_i$ are disjoint, since $\tau$ increases along non-constant trajectories. A segment that meets no $N_i$ contains no zero of $W$, so it runs from $\hat C_{t_-}$, or from a slice $B_m$ of a minimum of level $t^*$, to $\hat C_{t_+}$, with a constant landing point. Hence outside the $E_i$ the arcs of $\hat C_{t_+}$ are the flow images of those of $\hat C_{t_-}$ and of the new slices, with the same bottom intervals, while $\hat C_{t_\mp}\cap E_i$ flow to the entrance, resp.\ exit, of $N_i$, which we compare in the models below. (A point on a descending separatrix does not flow up to $t_+$; it lies in an $E_i$.)

The landing points of the descending separatrices of different event points are different points of $B_{<t^*}$, not in $L_{t_-}$, since they lie on different trajectories, which do not meet the saddles below $t^*$ (Proposition~\ref{prop:W}(iv)); cutting the state at $t_-$ at all of them changes no component. We show that the rule of an event point changes only the successors of the intervals that end at its own landing points, resp.\ at its ends, and adds the flow image of a new slice $B_m$ (rule (b)) or deletes a component in its $E_i$ (the oval of (c), the line $\lambda$ of (e)). These changes concern disjoint data, so they commute, and together they give the successors at $t_+$, whose cycles and maximal chains are the components of $\hat C_{t_+}$; hence the rules can be applied in any order.

(a) Near $p$, $W$ is the gradient of $\tau$ for a smooth metric (proof of Proposition~\ref{prop:W}), so its two descending and two ascending separatrices are smooth arcs from $p$, alternating around $p$ (stable manifold theorem), which cut $N_p$ into four sectors that trajectories do not leave while in $N_p$. The level $t^*-c$ meets $N_p$ in two arcs $a_1,a_2$ without special points, and $a_i$ crosses the descending branch landing at $b_i$ once, its two halves landing on $J^-_{b_i}$ and $J^+_{b_i}$. The level $t^*+c$ meets $N_p$ in two arcs, each crossing one ascending branch once, with its halves in the two sectors next to it, which are next to different descending branches; backwards along $W$ each half flows into the half of $a_1$ or $a_2$ in its sector. So each new arc consists of a half of $a_1$, a special point and a half of $a_2$. Since $a_1$ and $a_2$ run through $N_p$ in opposite directions (proof of Proposition~\ref{prop:interior}), one new arc consists of the arcs landing on $J^-_{b_1}$ and $J^+_{b_2}$, and the other of those landing on $J^-_{b_2}$ and $J^+_{b_1}$, in this order, which is rule (a).

(d) In the model of Proposition~\ref{prop:W}(ii), $\hat S=\{v\le0\}$, $T''(y_0)>0$ and $W=(T'(y),-v)$. The level $t^*-c$ is one arc in the block, without special points, crossed by the descending separatrix $\{u=0\}$. The level $t^*+c$ consists of the pieces with $u>0$ and $u<0$, which end on $\hat B$ at the new ends; backwards, $|u|$ decreases, so each piece flows to the half of the old arc on the same side of $\{u=0\}$. Hence the two pieces land on $J^-_b$ and $J^+_b$, and the first ends, the second starts, at a new end; the new ends flow along $\hat B$ into the saddle.

(e) On the side $Q_o$, $\hat S=\{v\ge0\}$, $T''(y_0)<0$ and $W=(T'(y),v)$: the two pieces of the level $t^*-c$ ending on $\hat B$ form one arc at level $t^*+c$, crossed by the ascending separatrix $\{u=0\}$ in a new special point; backwards its two halves flow to the two pieces. On the side $Q_m$ the line $\lambda$, identified by its ends, flows into the maximum; its bottom intervals become dead.

(b), (c) A new oval or line near a minimum $m$ is one arc, landing on $B_m$. At a death the dying oval of $\hat C_{t_-}$ lies in the set $E_i$ of the block of $M$, as all its points converge to $M$; so its bottom intervals become dead, and any point near $M$ identifies it through its landing point.
\end{proof}

\begin{corollary}\label{cor:procedure}
Starting from the state at $t_0$ (Lemma~\ref{lem:bottom}) and applying Theorem~\ref{thm:update} at the events in increasing order, one obtains the state, hence $\Sigma_t$, at every regular level $t\in[t_0,1]$, in particular $\Sigma_1$. The trajectories followed are the two descending separatrices of every saddle, the descending separatrix of every creation and one trajectory at every death, at most $2N$ for $N$ event points; the ends are followed by continuation of the roots of the edge polynomials.
\end{corollary}

\subsection{Examples and remarks}

\begin{example}[Example~\ref{ex:conic} continued]\label{ex:conic2}
Let $b=1$, so $t_1\approx0.486$, and $t_0=0.05$. The bottom consists of the lines $A$ and $B$ of Figure~\ref{fig:conic}(a). For the Euclidean metric (admissible here, as the separatrices stay in a compact set) the descending separatrices of the saddle land at $b_A\in A$ and $b_B\in B$ near the rays of $\Gamma_{++}$ dual to $[(2,0),(1,1)]$ and $[(1,1),(0,2)]$, which belong to $A$ and $B$. So $\gamma_1=A\neq\gamma_2=B$, and rule (a) gives the lines $\lambda=e_1A^-B^+e_2$ and $\mu=e_3B^-A^+e_4$, where $A^-\subset A$ runs from $e_1$ to $b_A$ and $B^+\subset B$ from $b_B$ to $e_2$. At the annihilation at $t^*=1/2$, with $Q_m=\cO_{++}$, rule (e) deletes $\lambda$, whose ends are $e_1,e_2$, and joins the terminal end $e_4$ and the initial end $e_3$ of $\mu$ into the oval $(B^-A^+)$: $(n_o,n_u)=(1,0)$.
\end{example}

\begin{example}[events at infinity]\label{ex:infinity}
Let $f=-2+4x_2+3x_1^2+2x_2^2+2x_1x_2^2-4x_1^2x_2+x_1^2x_2^2$, with the generic lifting $h=1$ on $(0,0)$ and $(2,1)$, $h=2$ on $(0,1)$, $h=-1$ on $(1,2)$ and $h=0$ elsewhere, and $D_{e_R}=\{x_1=\infty\}$, $D_{e_T}=\{x_2=\infty\}$ for the right and top edges of $\Delta=[0,2]^2$. For $t_0=\frac15$ and the metric $\sum_id\xi_i^2$, $\xi_i=\operatorname{asinh}x_i$, which is of the kind required in \S\ref{ssec:field}, the bottom consists of two lines, $A$ from $e_1$ to $e_2$ and $B$ from $e_3$ to $e_4$, with ends on $D_{e_T}$ (\S\ref{app:worked}). With $(u\,v)$ the bottom interval from $u$ to $v$, the events and states are:
\begin{center}\small
\begin{tabular}{@{}llll@{}}
\toprule
$t$ & event & landing points & lines $L$ and ovals $O$ after the event\\
\midrule
$0.401$ & saddle $q_1$ & $c_1\in A$, $c_2\in B$ & $L_1=(e_3\,c_2)(c_1\,e_2)$, $L_2=(e_1\,c_1)(c_2\,e_4)$\\
$0.564$ & saddle $q_2$ & $d_1\in(e_1\,c_1)$, $d_2\in(c_2\,e_4)$ & $L_1$, $L_2'=(e_1\,d_1)(d_2\,e_4)$, $O_1=(d_1\,c_1)(c_2\,d_2)$\\
$1/\sqrt2$ & annihilation on $D_{e_T}$ & & $O_1$, $O_2=(e_3\,c_2)(c_1\,e_2)$\\
$\sqrt3/2$ & creation on $D_{e_R}$ & $b\in(c_1\,e_2)$ & $L_3=(b\,e_2)(e_3\,c_2)(c_1\,b)$, $\lambda$, $O_1$\\
\bottomrule
\end{tabular}
\end{center}
Hence $(n_o,n_u)=(1,2)$ for $t=1$, from five separatrices: the oval $O_1$ and two lines ending at $x_2=1$ and $x_2=3$ on $D_{e_R}$. Section~\ref{app:worked} gives the positions and pictures of all levels.
\end{example}

\begin{remark}[branches are not enough]\label{rem:branches}
The tropical branch on which a separatrix lands does not suffice: if an oval of $\hat C_{t_0}$ splits at a saddle $q_1$ into two ovals, which meet again at a saddle $q_2$, the two arcs at $q_2$ lie on different components because the landing points of $q_2$ lie in the two intervals cut by those of $q_1$, which is decided by their order. For the same reason landing points are computed at a level $t_0>0$, not as limits: as $t_0\to0$ the points $\Log_{t_0}\Lambda(x)$ converge to $\Gamma_\epsilon$, but different trajectories may have the same limit: in Example~\ref{ex:conic2}, numerically, the images of $b_A$ and of all tested points of $A\cap\cO_{++}$ tend to the vertex $(0,-1)$.
\end{remark}

Saddle connections that persist for all liftings near a given one occur along invariant parts of the axes (Example~\ref{ex:connection}); their labels are determined by the order of the landing points:

\begin{lemma}[connections along an invariant part of an axis]\label{lem:axisconn}
Let $\{i,j\}=\{1,2\}$, let $g_t$ be diagonal in $(x_1,x_2)$ near the axis $\{x_j=0\}$, and let $Z$ be a component of $S\cap\{x_j=0\}$ on which $\partial_jf_t$ vanishes, an \emph{invariant part} of the axis (this holds for every component if no exponent of $\cA$ has $a_j=1$, and can hold by an identity of the coefficients otherwise, \S\ref{ssec:eventcomp}). Suppose that a descending separatrix $\gamma$ of $V$ of a saddle $(p,t^*)$ runs along $Z$ into a lower saddle $(q,t_q)$, and that no other descending separatrix of $V$ converges to a saddle. Let $\iota=\sgn(x_i-q_i)$ for the points $x$ of $\gamma$ just above $t_q$ (the \emph{side} from which $\gamma$ reaches $q$), let $\gamma_q^+$ be the descending separatrix of $q$ that leaves $q$ into $\{x_j>0\}$, $b_q$ its landing point, and $T$ the oriented tangent of a level just below $t_q$ at its point on $\gamma_q^+$ (for levels close to $t_q$ the signs of $\iota$ and $T_i$ do not change). Let $W$ satisfy (i)--(iv) of Proposition~\ref{prop:W}, with model neighbourhoods that miss the descending separatrices of $V$ except near the starts of those of the creations, and with $\theta=0$ except on the band below $t^*$, where $Y$ is supported in a small arc of its upper level around the point of $\gamma$ and $-Y$ points into $\{x_j>0\}$. Then for all small $\theta>0$: (i) the other descending separatrices of $W$ follow those of $V$ away from the starts of the creations, with the same landing points; (ii) $\gamma$ is replaced by a separatrix that lands next to $b_q$, with no landing point in between, after $b_q$ if $\iota T_i>0$ and before $b_q$ otherwise. The same holds for several such connections at once, if no other descending separatrix of $V$ converges to a saddle; two connections into one $q$ come from opposite sides and land on opposite sides of $b_q$.
\end{lemma}

\begin{proof}
At a point of $Z$ the gradient of $f_t$ for $g_t$ has no $x_j$-component, so $V$ is tangent to $Z$. (i) The other separatrices that cross the band cross its upper level away from $\gamma$, so a small enough arc misses them, and those of the saddles below $t^*$ do not reach the band; so none of them changes, and by the choice of the neighbourhoods they follow $V$ away from the starts of the creations. (ii) At $q$ the Hessian is diagonal: in the plane $\{x_j=0\}$, with coordinates $(x_i,t)$, $Z$ is near $q$ the smooth curve $\{\phi=0\}$, $\phi=f_t(x_i,0)$, since $\partial_t\phi(q)=\alpha\neq0$; so $\partial_jf_t(x_i,0)=\lambda\phi$ with $\lambda$ smooth, and $\partial_i\partial_jf_t=\lambda\,\partial_i\phi=0$ at $q$. As $\gamma$ converges to $q$ along the axis, the axis carries the ascending separatrices of $q$, and the descending ones leave it. For $\theta>0$ the descending holonomy across the band is the flow of $-Y$ at time $\theta$, so the new separatrix $\gamma_\theta$ leaves the axis in the band, by $O(\theta)$, into $\{x_j>0\}$, where it stays, since near $\gamma$ the axis meets $S$ only in the invariant set $Z$; below the band it follows $V$. By continuous dependence on the initial point it enters the Morse block of $q$ (proof of Theorem~\ref{thm:update}) in the sector between the ascending separatrix of $q$ containing $\gamma$, on the side $\iota e_i$, and $\gamma_q^+$, leaves it next to $\gamma_q^+$ on the side $\iota e_i$, and lands within $o(1)$ of $b_q$ as $\theta\to0$. As $b_q$ is the landing point of no other separatrix, for small $\theta$ no landing point lies in between, and since the flow preserves the order on the levels (Proposition~\ref{prop:landing}), the side is that of $\gamma_\theta$ on a level just below $t_q$, given by the sign of $\iota T_i$. For several connections take disjoint arcs, each so small that it misses the separatrices of $V$ other than its own; for small $\theta$ the twisted separatrices stay close to those of $V$ and miss the other arcs, and the argument applies to each; two connections into $q$ arrive along the two ascending separatrices of $q$ on the axis, hence from opposite sides.
\end{proof}

\begin{remark}[from $W$ to $V$]\label{rem:numerics}
Outside the bands and neighbourhoods of the boundary critical points, $W$ has the trajectories of $V$. If no descending separatrix of $V$ converges to an interior saddle or accumulates at the boundary saddle of an annihilation, take $\theta=0$ on every band and model neighbourhoods that miss the separatrices except near the starts of those of the creations; then $W$ has the labels of $V$. Lemma~\ref{lem:axisconn} treats connections along invariant parts of the axes, but not those forced by other symmetries (e.g.\ along the diagonal).
\end{remark}

\begin{remark}[isotopy type]\label{rem:isotopy}
The state also locates points: a non-special point $y\in\hat C_t$ that is not an end lies on the component whose sequence contains the bottom interval of $\Lambda(y)$ (Proposition~\ref{prop:landing}). This could give the isotopy type of $(\R^2,C)$, which by the Schoenflies theorem is determined by the cyclic order of the ends, their pairing and the tree of the regions of $\R^2\setminus C$, found by locating, for each oval, the first point of $C$ on a horizontal half-line from its rightmost point; we have not implemented this.
\end{remark}

\section{The choice of the lifting}\label{sec:lifting}

Write $\varsigma=\sgn c$, $\ell=\log|c|\in\R^{\cA}$ and $s=1+\log(1/t)\in[1,\infty)$. Along~\eqref{eq:deformation}, $\log|c(t)|=\ell-(s-1)h$ moves on a half-line, whose intersections with the real discriminants are the events. These stay near the walls of the secondary fan of $\cA$, and for $h=-\ell$ the half-line lies on a ray from the origin, which crosses no wall; hence its events lie in an explicit window (\S\ref{ssec:window}), and for coefficients $\varsigma e^{\sigma u}$ with Gaussian $u$ an event has probability $O(1/\sigma)$, unlike for other liftings (\S\ref{ssec:spread}).

\subsection{Reduced coordinates and dequantization}\label{ssec:reduced}

Let $\hat A$ be the $3\times\cA$ matrix with the columns $(1,a)$, of rank $3$ by (H0), $m=|\cA|-3$, and $B$ a \emph{Gale dual} of $\cA$: an integer $\cA\times m$ matrix whose columns form a basis of the lattice $\ker\hat A\cap\Z^{\cA}$. Its rows $b_a\in\Z^m$ satisfy $\sum_ab_a=0$ and $\sum_aa_ib_a=0$, and $\ker B^T$ is the row space of $\hat A$, the space of affine functions of $a$. The rescalings $x\mapsto(\rho_1x_1,\rho_2x_2)$ and $f\mapsto\rho_0f$, $\rho_i>0$, preserve the quadrants and the topology of $C$ and add an affine function to $\ell$. Hence the topology of $C$ depends only on $\varsigma$ and the \emph{reduced coordinates} $\zeta=B^T\ell\in\R^m$, which along the path of $h$ move on the half-line $\zeta-(s-1)B^Th$.

\begin{proposition}[dequantization]\label{prop:dequant}
For $h=-\ell$, $c_at^{h(a)}=\varsigma_a|c_a|^s$ with $s=1+\log(1/t)$, i.e.\ $f_t=f^{(s)}:=\sum_{a}\varsigma_a|c_a|^sx^a$, and the reduced coordinates move on the ray $s\mapsto s\zeta$. This is Viro's dequantization of $f$ \cite{Vir01}, a real form of Maslov dequantization \cite{Lit07}; its start is the patchwork of $\max_a(\langle a,w\rangle+\log|c_a|)$, and $-\ell$ is generic for $\ell$ outside finitely many hyperplanes.
\end{proposition}

\begin{proof}
$c_at^{-\log|c_a|}=\varsigma_a|c_a|e^{\log|c_a|\log(1/t)}=\varsigma_a|c_a|^{1+\log(1/t)}$, and $F_{-\ell}(w)=\max_a(\langle a,w\rangle+\ell_a)$; the last claim is the remark after Definition~\ref{def:generic}.
\end{proof}

We call $s\mapsto f^{(s)}$, $s\ge1$, the \emph{dequantization path} of $f$; as $t$ increases, $s$ decreases from $\infty$ (the patchwork) to $1$ (the polynomial $f$).

\begin{remark}[nesting]\label{prop:nesting}
For $R\ge1$, $(|c_a|^R)^{s'}=|c_a|^{Rs'}$, so the dequantization path of $f^{(R)}$ is the tail $\{f^{(s)}:s\ge R\}$ of the path of $f$, and the events of $f^{(R)}$ are those of $f$ with $s\ge R$; their number is non-increasing in $R$.
\end{remark}

\subsection{Events and the reduced discriminant}\label{ssec:hk}

By Descartes' rule there are at most $|\cB|-1$ events (proof of Lemma~\ref{lem:eventpoly}), whatever $h$ is. Among the factors of $E_\cA$, the $\cA$-discriminant $\Delta_\cA$ detects singular points in $(\C^*)^2$; up to a monomial it is a Laurent polynomial in the reduced coefficients $c^B=(\prod_ac_a^{b_{aj}})_{j\le m}$ \cite[Ch.~9]{GKZ94}, whose zero set is the closure of the image of the Horn--Kapranov map $\xi\mapsto(\prod_a\langle b_a,\xi\rangle^{b_{aj}})_{j\le m}$ \cite{Kap91}, \cite[Thm.~2.10]{DRRS07}.

\begin{lemma}[real Horn--Kapranov parametrization]\label{lem:hk}
Let $K_\epsilon$ be the open cone of the $\xi\in\R^m$ with $\sgn\langle b_a,\xi\rangle=\sigma_\epsilon(a)$ for all $a$, and $\Phi(\xi)=\sum_a\log|\langle b_a,\xi\rangle|\,b_a$ for $\xi\in K_\epsilon$. The singular points of $C$ in $\cO_\epsilon$ correspond bijectively to the rays $\R_{>0}\xi\subset K_\epsilon$ with $\Phi(\xi)=\zeta$; the singular point of such a ray is $(\epsilon_1e^{\omega_1},\epsilon_2e^{\omega_2})$, where $(\alpha,\omega)$ is the unique solution of $\alpha+\langle a,\omega\rangle=\log|\langle b_a,\xi\rangle|-\ell_a$, $a\in\cA$.
\end{lemma}

\begin{proof}
A point $x\in\cO_\epsilon$ is a singular point of $C$ if and only if the vector $(c_ax^a)_a$ of its terms satisfies $\sum_ac_ax^a=f(x)=0$ and $\sum_aa_ic_ax^a=x_i\partial_if(x)=0$, i.e.\ lies in $\ker\hat A$, i.e.\ equals $B\xi$ for some $\xi\in\R^m$. Then $\sgn\langle b_a,\xi\rangle=\sgn(c_ax^a)=\sigma_\epsilon(a)$, and with $\omega_i=\log|x_i|$, $\Phi(\xi)=\sum_a(\ell_a+\langle a,\omega\rangle)b_a=\zeta$, as $\sum_aa_ib_a=0$. Conversely, if $\xi\in K_\epsilon$ and $\Phi(\xi)=\zeta$, the vector $(\log|\langle b_a,\xi\rangle|-\ell_a)_a$ lies in $\ker B^T$, the row space of $\hat A$, which gives $(\alpha,\omega)$, unique since $\hat A$ has rank $3$; then $x=(\epsilon_1e^{\omega_1},\epsilon_2e^{\omega_2})$ satisfies $c_ax^a=\sigma_\epsilon(a)e^{\ell_a+\langle a,\omega\rangle}=e^{-\alpha}\langle b_a,\xi\rangle$, so $(c_ax^a)_a\in\ker\hat A$ and $x$ is singular. Replacing $\xi$ by $\rho\xi$, $\rho>0$, only changes $\alpha$.
\end{proof}

Thus the interior events in $\cO_\epsilon$ are the intersections of the half-line $\zeta-(s-1)B^Th$ with the \emph{signed contour} $\Phi(K_\epsilon)$ (cf.\ \cite[\S3]{RR17}, \cite[Prop.~2.10]{DRT26}), and there are none if $K_\epsilon=\emptyset$, i.e.\ by the theorem of the alternative if a non-zero affine function $\chi$ satisfies $\sigma_\epsilon(a)\chi(a)\ge0$ for all $a$; the boundary events are described in the same way, in dimension one, by the edge polynomials. The contours are explicit for a circuit $Z\subset\cA$. Its \emph{relation} is the primitive integer vector $\theta_Z\in\Z^Z$, unique up to sign, with $\sum_a\theta_a=0$ and $\sum_a\theta_aa=0$; put $\kappa_\theta=\sum_{a\in Z}\theta_a\log|\theta_a|$.

\begin{proposition}[circuits]\label{prop:circuits}
Let $Z\subset\cA$ be a circuit with relation $\theta$ and $f^Z=\sum_{a\in Z}c_ax^a$. If $|Z|=4$, $f^Z$ has a singular point in $\cO_\epsilon$ if and only if $\sigma_\epsilon|_Z=\pm\sgn\theta$ and $\langle\theta,\ell\rangle=\kappa_\theta$, and then exactly one. If $Z$ consists of three points of an edge with lattice coordinate $k$, the same holds for the multiple roots in $\R_{>0}$, resp.\ $\R_{<0}$, of $\sum_{a\in Z}c_ay^{k(a)}$, with $\sigma_\epsilon(a)$ replaced by $\sgn c_a$, resp.\ $(-1)^{k(a)}\sgn c_a$. Hence $(f^Z)^{(s)}$ has such a singular point if and only if the sign condition holds and $s\langle\theta,\ell\rangle=\kappa_\theta$: if $\langle\theta,\ell\rangle\neq0$, exactly at the \emph{singular parameter} $s^*=\kappa_\theta/\langle\theta,\ell\rangle$ of the circuit family; if $\langle\theta,\ell\rangle=0$, for no $s$ if $\kappa_\theta\neq0$ and for all $s$ if $\kappa_\theta=0$. If $Z=\cA$, resp.\ $Z=\cA\cap e$ for an outer edge $e$, the sign condition holds and $s^*$ exists, then $s^*\neq1$, as $E_\cA(c)\neq0$, and if $s^*>1$ it is an event of the dequantization path of $f$.
\end{proposition}

\begin{proof}
Let $|Z|=4$. Then $\ker\hat A_Z=\R\theta$, where $\hat A_Z$ has the columns $(1,a)$, $a\in Z$, so Lemma~\ref{lem:hk} applies to $f^Z$ with $m=1$ and $b_a=\theta_a$: $K_\epsilon$ is $\R_{>0}$, resp.\ $\R_{<0}$, if $\sigma_\epsilon|_Z=\sgn\theta$, resp.\ $-\sgn\theta$, and empty otherwise, $\Phi(\xi)=\sum_a\theta_a\log|\theta_a\xi|=\kappa_\theta$ as $\sum_a\theta_a=0$, and the reduced coordinate of $f^Z$ is $\langle\theta,\ell\rangle$. So $f^Z$ has a singular point in $\cO_\epsilon$ if and only if $\sigma_\epsilon|_Z=\pm\sgn\theta$ and $\langle\theta,\ell\rangle=\kappa_\theta$, and then exactly one, as $K_\epsilon$ is one ray. For an edge the same argument applies with $(1,k(a))$ in place of $(1,a)$. Finally, $(f^Z)^{(s)}$ has the log-coefficients $s\ell$; for $Z=\cA$, resp.\ $\cA\cap e$, its singular points in $(\R^*)^2$, resp.\ the real multiple roots of its edge polynomial, are event points of the dequantization path, and none occurs at $s=1$, as $E_\cA(c)\neq0$.
\end{proof}

In general the singular points of $f^Z$ are not those of $f$: for $1+x_1+x_2+x_1x_2+2x_1^2+3x_2^2$ the truncation $(1+x_1)(1+x_2)$ to a parallelogram ($\kappa_\theta=0$) is singular for every $s$, although $E_\cA(c)\neq0$ and $-\ell$ is generic. For three consecutive points of an outer edge, $\theta=(1,-2,1)$ and $\kappa_\theta=-\log4$: the dequantization path of an edge trinomial $c_0+c_1y+c_2y^2$ has a boundary event if and only if $c_0c_2>0$ and $0<2\ell_1-\ell_0-\ell_2<\log4$, at $s^*=\log4/(2\ell_1-\ell_0-\ell_2)$ (for the conic of Example~\ref{ex:conic}, if $1<b<2$). For the cubic of Example~\ref{ex:cubic}, $\theta=\pm(1,1,1,-3)$ on $\bigl((0,0),(3,0),(0,3),(1,1)\bigr)$, and there is an event if and only if $1<\lambda<3$, at $s^*=\log3/\log\lambda$, the death of the oval of the start; the lifting of Example~\ref{ex:cubic2} has one for every $\lambda<3$.

\subsection{The event window}\label{ssec:window}

For a general support the signed contours are not explicit, but their tropical limit is: the tropical $\cA$-discriminant is $\operatorname{trop}(\ker\hat A)+\operatorname{rowspace}(\hat A)$ \cite[Thm.~1.1]{DFS07}. A quantitative real version confines the events of the dequantization path to a window.

Let $F$ be $\Delta$ or one of its edges and $\cA_F=\cA\cap F$, which spans $F$ affinely. Let $\Pi$ be a hyperplane of the affine span of $F$ spanned by points of $\cA_F$ (a line through two points of $\cA$ if $F=\Delta$, a point of $\cA_F$ if $F$ is an edge), and $\chi_\Pi$ the primitive integral affine function on the span of $F$ that vanishes on $\Pi$, unique up to sign, so that $|\chi_\Pi(a)|$ is the lattice distance of $a$ from $\Pi$ (for $\Pi=\{a'\}$, $\chi_\Pi(a)=k(a)-k(a')$ for a lattice coordinate $k$ on $F$: $k=a_1$ on $E_1$, $k=a_2$ on $E_2$, $k=\langle a,\mu_e\rangle$ on an outer edge $e$). Put
\[
r_\Pi=\log\Bigl((|\cA_F\setminus\Pi|-1)\max_{a\in\cA_F\setminus\Pi}|\chi_\Pi(a)|\Bigr),\qquad r_F=\max_\Pi r_\Pi,
\]
the maximum over the $\Pi$ with $|\cA_F\setminus\Pi|\ge2$ ($r_F=-\infty$ if there is none). For a circuit $Z$ of $\cA_F$ let $g_Z(\ell)=\min_\gamma\max_{a\in Z}(\gamma(a)-\ell_a)$, the minimum over the affine functions $\gamma$ on the span of $F$ with $\gamma\ge\ell$ on $\cA_F$, a linear program in at most four variables, and let $g_F(\ell)$ be the minimum of $g_Z(\ell)$ over the circuits of $\cA_F$ ($g_F=\infty$ if there are none). Thus $g_Z(\ell)=0$ if and only if the points $(a,\ell_a)$, $a\in Z$, lie on a common face of the upper hull of $\{(a,\ell_a):a\in\cA_F\}$, and the \emph{gap} $g(\ell):=g_\Delta(\ell)$ is positive if and only if $-\ell$ is generic \cite[Ch.~7]{GKZ94}. For a circuit of an edge $F$, computing $g_Z$ in $F$ or in $\Delta$ gives the same value (add to a majorant a large multiple of the affine function vanishing on $F$), so $g_F\ge g$; moreover $g_Z$ depends only on $\zeta$, and $g_Z(s\ell)=s\,g_Z(\ell)$ for $s>0$. If $\cA$ has no circuit, then $\cA=\{0,(k,0),(0,l)\}$, $f$ and its truncations have no singular points, and there are no events; otherwise $g(\ell)<\infty$, and $r\ge0$ below, since a line through two points of $\cA$ misses two points of a circuit.

\begin{lemma}\label{lem:cocircuit}
Let $v\in\R^{\cA_F}$ have no zero entry and satisfy $\sum_av_a=0$ and $\sum_av_aa=0$, and put $w_a=\log|v_a|$.
(i) Let $\Pi$ be as above, and split $\cA_F\setminus\Pi$ into the sets $P$ and $N$ on which $\chi_\Pi(a)v_a$ is positive, resp.\ negative. Then both are non-empty, and $|\max_Pw-\max_Nw|\le r_\Pi\le r_F$.
(ii) The set $T=\{a\in\cA_F:w_a\ge\max w-r_F\}$ contains a circuit of $\cA_F$.
\end{lemma}

\begin{proof}
(i) As $\chi_\Pi$ is affine, $\sum_{a\notin\Pi}\chi_\Pi(a)v_a=0$, and $|\chi_\Pi(a)|\ge1$ for $a\notin\Pi$. Hence both sets are non-empty and $e^{\max_Pw}\le\sum_P|\chi_\Pi(a)v_a|=\sum_N|\chi_\Pi(a)v_a|\le|N|\max|\chi_\Pi|\,e^{\max_Nw}$, where $|N|\le|\cA_F\setminus\Pi|-1$; symmetrically for $N$. (ii) Let $a_1$ attain $\max w$. If a hyperplane $\Pi$ spanned by points of $\cA_F$ does not contain $a_1$, then by (i) the class of $\cA_F\setminus\Pi$ not containing $a_1$ has a point in $T$. If $T$ were affinely independent, we could extend it by points of $\cA_F$ to an affine basis of the span of $F$, and the basis points other than $a_1$ would span such a $\Pi$ with $T\setminus\Pi=\{a_1\}$, a contradiction.
\end{proof}

\begin{theorem}[event window]\label{thm:window}
Let $s>0$. If $f^{(s)}$ has a singular point in $(\R^*)^2$, then $s\,g(\ell)\le r_\Delta$. If the truncation of $f^{(s)}$ to an edge $F$ of $\Delta$ has a multiple root in $\R^*$, as at a boundary event or a singular point on an axis, then $s\,g_F(\ell)\le r_F$. In particular every event of the dequantization path satisfies $s\,g(\ell)\le r$, where $r$ is the maximum of the $r_F$ over $F=\Delta$ and the edges of $\Delta$; if $0<g(\ell)<\infty$, the events lie in the window $1\le s\le r/g(\ell)$.
\end{theorem}

\begin{proof}
At a singular point $x\in\cO_\epsilon$ of $f^{(s)}$, the vector $v$ of the terms $\varsigma_a|c_a|^sx^a$ satisfies the hypotheses of Lemma~\ref{lem:cocircuit} for $F=\Delta$ (\S\ref{ssec:hk}), with $w_a=s\ell_a+\langle a,\omega\rangle$, $\omega_i=\log|x_i|$. Let $Z\subset T$ be a circuit as in Lemma~\ref{lem:cocircuit}(ii). The affine function $\gamma(a)=(\max w-\langle a,\omega\rangle)/s$ satisfies $\gamma\ge\ell$ on $\cA$ and $s(\gamma(a)-\ell_a)=\max w-w_a\le r_\Delta$ for $a\in Z$, hence $s\,g(\ell)\le s\,g_Z(\ell)\le r_\Delta$. A multiple root $y_0\in\R^*$ of the truncation $\sum_{a\in\cA_F}\varsigma_a|c_a|^sy^{k(a)}$ gives $v_a=\varsigma_a|c_a|^sy_0^{k(a)}$ with $\sum_av_a=\sum_ak(a)v_a=0$, and the argument applies in dimension one. The rest follows from $g_F\ge g$.
\end{proof}

\begin{corollary}\label{cor:window}
(i) If $g(\ell)>r$, the dequantization path has no events, and the topology of $C$ is that of the patchwork of $-\ell$ (Theorem~\ref{thm:viro-plane} and Proposition~\ref{prop:isotopy}). (ii) If $0<g(\ell)<\infty$, then, since $g(R\ell)=Rg(\ell)$, $f^{(R)}$ has no events for $R>r/g(\ell)$. If $g(\ell)=0$, the algorithm uses the perturbed lifting of Remark~\ref{rem:freedom}(ii).
\end{corollary}

Part (i) is a criterion on the sizes of the coefficients, as in effective patchworking; its closest predecessor that we know is \cite[Prop.~3.1]{EdW23}, whose linear inequalities on $\log|c|$ certify that the real polyhedral homotopy avoids the discriminants, while the emptiness criterion of \cite[Thm.~3.3]{DRT26} depends only on the support and the signs. For an outer edge, a criterion of Kurtz type on the ratios $|c_k|^2/|c_{k-1}c_{k+1}|$ of the coefficients of $\psi^e$ excludes boundary events (\S\ref{app:kurtzcrit}).

\begin{remark}\label{rem:window}
Since $|\chi_\Pi(a)|\le\operatorname{diam}(\Delta)^2$, $r\le\log(|\cA|\operatorname{diam}(\Delta)^2)$. For $s=1$, Theorem~\ref{thm:window} says that the signed contours and their edge analogues lie in the neighbourhood $\{g\le r\}$ of the set $\{g=0\}$ of the $\ell$ with $-\ell$ on a wall of the secondary fan. The bound is sharp: for $1+x_1^3+x_2^3-\lambda x_1x_2$, $\lambda>1$, $r=\log3$ (every line through two points of $\cA$ leaves two points with $\max|\chi_\Pi|=3$) and $g(\ell)=\log\lambda$, so $r/g=s^*$ when there is an event (Proposition~\ref{prop:circuits}).
\end{remark}

\begin{remark}[free parameters]\label{rem:freedom}
(i) Adding $\delta=\|\ell'-\ell\|_\infty$ to a majorant of $\ell$ gives a majorant of $\ell'$, so $|g(\ell)-g(\ell')|\le2\delta$. (ii) If $-\ell$ is not generic, as in Example~\ref{ex:conic}, we use $h=-\ell+\varepsilon\eta$ with $\varepsilon=2^{-k}$, $k\ge1$ the smallest such that $\cS_h$ refines the subdivision induced by $-\ell$. Its path is $s(\ell-\varepsilon\eta)+\varepsilon\eta$, so by (i) and Theorem~\ref{thm:window} its events satisfy $s\le(r+2\varepsilon\|\eta\|_\infty)/g(\ell-\varepsilon\eta)$. A circuit $Z$ with $g_Z(\ell)=0$ has $\langle\theta_Z,\ell\rangle=0$, so on this path $\langle\theta_Z,\log|c|\rangle=-(s-1)\varepsilon\langle\theta_Z,\eta\rangle$, and Proposition~\ref{prop:circuits} tells whether $f^Z$ becomes singular; we take, among a few random $\eta$, one with the fewest circuits that produce an event (\S\ref{app:inputs}). (iii) For exact computations one may round $-M\ell$, $M\in\Z_{>0}$, to the nearest integers: by (i) with $\delta\le\frac12$, the segment from $-M\ell$ to the rounded lifting crosses no wall once $Mg(\ell)>1$, so the rounded lifting is generic, with the same triangulation. (iv) A translation achieving (H0) changes the support and the gap but need not make the gap positive, as circuits on an outer edge can obstruct this: for $1+x_1^2+x_1x_2+x_2^2$ the coefficients on the outer edge, hence $g=0$, are translation invariant; positive rescalings leave $g$ unchanged.
\end{remark}

\subsection{Spread of the coefficients}\label{ssec:spread}

Write $c=\varsigma e^{\sigma u}$, $u$ of unit scale, with \emph{spread} $\sigma$.

\begin{theorem}[probabilistic bound]\label{thm:prob}
Fix $\varsigma\in\{\pm1\}^{\cA}$ and $\sigma>0$, let $u$ be a standard Gaussian vector in $\R^{\cA}$, and for the dequantization path of $c=\varsigma e^{\sigma u}$ let $K(\sigma)$ be its number of events and $N(\sigma)$ its number of event points of the kinds (A) and (B) of Assumption~\ref{ass:morse}, a boundary event point counted once. Then
\[
\Pr[K(\sigma)>0]\le\frac{C_1}{\sigma},\qquad \E[K(\sigma)]\le\frac{(|\cB|-1)\,C_1}{\sigma},\qquad \E[N(\sigma)]\le\frac{B_\cA(|\cB|-1)\,C_1}{\sigma},
\]
with $|\cB|$ the number of terms of $E_\cA$, $C_1=\frac{r}{\sqrt{2\pi}}\sum_Z\|\theta_Z\|_1/\|\theta_Z\|_2$ over the circuits $Z$ of $\cA$ ($C_1=0$ if $\cA$ has no circuit; then $f$ has no events), and $B_\cA=(d-1)^2+\sum_e\lfloor L_e/2\rfloor$ over the outer edges, $d=\deg f$. Since $\|\theta_Z\|_1\le2\|\theta_Z\|_2$, $C_1=O(|\cA|^4\log(|\cA|\operatorname{diam}\Delta))$.
\end{theorem}

\begin{proof}
If $K(\sigma)>0$, Theorem~\ref{thm:window} gives $s\ge1$ and a circuit $Z$ of $\cA$ with $s\sigma g_Z(u)\le r$. For an affine $\gamma$ attaining $g_Z(u)$ we have $\langle\theta_Z,u\rangle=\langle\theta_Z,u-\gamma\rangle$ and $0\le\gamma-u\le g_Z(u)$ on $Z$, so $\sum_a\theta_a=0$ gives $|\langle\theta_Z,u\rangle|\le\frac12\|\theta_Z\|_1g_Z(u)\le r\|\theta_Z\|_1/(2\sigma)$. Since $\langle\theta_Z,u\rangle\sim N(0,\|\theta_Z\|_2^2)$ has density at most $(\sqrt{2\pi}\|\theta_Z\|_2)^{-1}$, a union bound gives the first inequality, and $K(\sigma)\le|\cB|-1$ (proof of Lemma~\ref{lem:eventpoly}; $E_\cA(c)\neq0$ almost surely) the second. At one parameter a point of kind (A) is a non-degenerate common zero of $\partial_1f$ and $\partial_2f$. Let $H=\gcd(\partial_1f,\partial_2f)$ and $\partial_if=HP_i$, so that $P_1$ and $P_2$ are coprime. At a zero of $H$ the rows $P_i\nabla H$ of the Hessian are dependent, so the points of kind (A) are common zeros of $P_1$ and $P_2$, at most $\deg P_1\deg P_2\le(d-1)^2$ by B\'ezout's theorem. An edge polynomial $\psi^e$, of degree $L_e$ after division by a monomial, has at most $\lfloor L_e/2\rfloor$ double roots. So $N(\sigma)\le B_\cA K(\sigma)$.
\end{proof}

The constant $C_1$ is pessimistic (\S\ref{ssec:performance}), but the rate $1/\sigma$ is attained: for the support and signs $(+,+,+,-)$ of Example~\ref{ex:cubic}, Proposition~\ref{prop:circuits} gives $\E[K(\sigma)]=\E[N(\sigma)]=\Pr[0<\sigma X<\log3]\sim\frac{\sqrt3\log3}{2\sqrt{2\pi}\,\sigma}$, with $X=u_{11}-\frac13(u_{00}+u_{30}+u_{03})\sim N(0,\frac43)$. Non-Morse event points have probability zero unless an edge is detached (Proposition~\ref{prop:genericmorse}(iii)), and then they can have positive probability: for $\cA=\{(0,0),(2,0),(1,1),(0,2)\}$ and $\varsigma_{20}=\varsigma_{02}$ the edge polynomial of $f^{(s)}$ has a real double root at $s^*=\log4/Y$, $Y=\sigma(2u_{11}-u_{20}-u_{02})\sim N(0,6\sigma^2)$, an event if $0<Y<\log4$. $K(\sigma)$ also counts non-Morse events with non-isolated event points, as for $f=1+x_1^4+x_2^4+\sqrt2(x_1^2x_2^2-x_1^2-x_2^2)$, with $E_\cA(c)\neq0$ and $f^{(2)}=(x_1^2+x_2^2-1)^2$.

Other liftings need not have this decay: their half-lines cross the walls of the normal fan of $\operatorname{Newt}(D)$, $D=\Delta_\cA\prod_{e\text{ outer}}\Delta_{\cA\cap e}$, at distance proportional to $\sigma$, and a change of sign of the dominant term of $D$ produces a zero of $D(c(t))$, an event unless all its singular points are non-real. By Proposition~\ref{prop:otherrays}, if $\cA$ contains a circuit, $v$ is independent of $u$ with a density and the paths of $h=-v$ are almost surely \emph{real-generic}, i.e.\ every zero of $D(c(t))$ in $(0,1]$ is an event, for all large $\sigma$, then $\liminf_{\sigma\to\infty}\E[K_{-v}(\varsigma e^{\sigma u})]$, with $K_{-v}(c)$ the number of events of the path of $h=-v$, is at least the expected number of sign changes of the dominant term of $D$ along the ray $u+\varrho v$, whereas $\E[K(\sigma)]\to0$; for the cubic of Example~\ref{ex:cubic} with signs $(+,+,+,-)$ and standard Gaussian $v$ this number is $\frac12$.

\section{Algorithm and numerical experiments}\label{sec:algorithm}\label{sec:experiments}

We state the algorithm and report on its implementation; details are in \S\S\ref{app:impl}--\ref{app:special}.

\subsection{Algorithm and conditional correctness}\label{ssec:algcorrect}

Algorithm~\ref{alg:main} assembles Sections~\ref{sec:tropical}--\ref{sec:lifting}; its numerical data are the events and the labels of the landing points.

\begin{theorem}[conditional correctness]\label{thm:algorithm}
Let $f$ satisfy (H0) and $E_\cA(c)\neq0$. Suppose that, for every lifting $h$ that Algorithm~\ref{alg:main} uses, the events are computed with all their event points and with their correct types, that $t_0$ is as in Lemmas~\ref{lem:bottom} and~\ref{lem:positions}, and that, for one field $W$ as in Proposition~\ref{prop:W}, the landing point of every descending separatrix in $\hat S\setminus\hat B$ of a saddle (Lemma~\ref{lem:descend}) is computed accurately enough to determine its label (Definition~\ref{def:labels}), including its order relative to the other landing points on the same component of the bottom and to the crossing points with the axes, or, for a connection along an invariant part of an axis, is given by Lemma~\ref{lem:axisconn}. If Algorithm~\ref{alg:main} terminates, it returns \textup{\textsc{non-Morse}} only at an event point on a detached edge, and otherwise the numbers $n_o$ and $n_u$ of ovals and lines of $C$, and the pairing of the ends of $C$ by its lines. If $\eta$ is drawn from a distribution with a density, then with probability one the algorithm takes a new lifting in line~\ref{alg:restart} at most once; it then terminates if the computations in lines~\ref{alg:trop}, \ref{alg:events} and~\ref{alg:labels} terminate.
\end{theorem}

\begin{proof}
Lines 1 and~\ref{alg:restart} give a generic $h$. Since $\varepsilon$ takes the countably many values $2^{-k}$, $\varepsilon\eta$ has a density if $\eta$ has one, so by Proposition~\ref{prop:genericmorse}(i) the path of the lifting drawn in line~\ref{alg:restart} satisfies Assumption~\ref{ass:morse} off the detached edges with probability one; this gives the last statement. If the algorithm terminates, it returns \textsc{non-Morse} only as stated, or it reaches the choice of $t_0$ with a generic $h$ whose path satisfies Assumption~\ref{ass:morse}; then the initial state is correct by Lemma~\ref{lem:bottom}, the labels determine the bottom intervals and their order by Lemma~\ref{lem:positions} and Definition~\ref{def:labels}, and the updates are correct by Theorem~\ref{thm:update}. Counting a dying oval in $n_\dagger$ instead of deleting it leaves $n_\dagger$ cyclic sequences of dead bottom intervals in the state, which the later rules do not touch, since no later landing point lies in a dead interval (its trajectories converge upwards to a maximum, Proposition~\ref{prop:landing}); so $n_o$ is the number of cyclic sequences minus $n_\dagger$. As $t=1$ is not an event, the final state is the state at $t=1$ (Corollary~\ref{cor:procedure}).
\end{proof}

\begin{algorithm}[bp]
\caption{Topology of a real plane curve}\label{alg:main}
\begin{algorithmic}[1]
\Require $f=\sum_{a\in\cA}c_ax^a$ satisfying (H0) and $E_\cA(c)\neq0$
\Ensure $(n_o,n_u)$ and the pairing of the ends of $C$ by its lines, or \textsc{non-Morse}
\State $h\gets-\log|c|$, perturbed as in Remark~\ref{rem:freedom}(ii) until it is generic
\State compute $\cS_h$, the curves $\Gamma_\epsilon$ and $T(f,h)$; label its components, cells and ends \Comment{\S\ref{sec:tropical}}\label{alg:trop}
\State compute the events $t_1<\dots<t_K$, all their event points and their types \Comment{\S\ref{ssec:implsummary}}\label{alg:events}
\If{some event points violate Assumption~\ref{ass:morse}} \Comment{Proposition~\ref{prop:genericmorse}}
  \State \textbf{if} one of them lies on a detached edge \textbf{then return} \textsc{non-Morse}
  \State $h\gets-\log|c|+\varepsilon\eta$ for a new random $\eta$ (Remark~\ref{rem:freedom}(ii)), until $h$ is generic; \textbf{go to} line~\ref{alg:trop}\label{alg:restart}
\EndIf
\State choose $t_0\in(0,t_1)$ ($t_1:=1$ if $K=0$); label the landing points of the descending separatrices\label{alg:labels}
\State state $\gets$ the components of $T(f,h)$, each a sequence with one term; $n_\dagger\gets0$ \Comment{Lemma~\ref{lem:bottom}}
\For{$i=1,\dots,K$}
  \State apply Theorem~\ref{thm:update} at the event points of $t_i$; at a death $n_\dagger\gets n_\dagger+1$, not rule (c)
\EndFor
\State \Return $n_o=\#$cyclic sequences $-\,n_\dagger$, $n_u=\#$linear sequences, and their ends
\end{algorithmic}
\end{algorithm}

\subsection{Implementation and scope}\label{ssec:implsummary}

We implemented Algorithm~\ref{alg:main} for $h=-\log|c|$ in Python, with kernels in C. The events are computed by one of two solvers. The \emph{real solver} searches the regions where three or four terms of $f^{(s)}$ dominate (Lemma~\ref{lem:cocircuit}) within the window of Theorem~\ref{thm:window}, by adaptive subdivision with interval bounds and the Krawczyk test in floating point; the \emph{complex solver} follows the complex critical points of $f$ from a total-degree start along a slightly non-real path. Both obtain the boundary events and the events on the axes from the real multiple roots of the edge polynomials and of the restrictions of $f_t$ to the axes. A descending separatrix is computed by integrating $V$, i.e.\ $\dot x=-\partial_tf_t\,\nabla_gf_t/|\nabla_gf_t|_g^2$ for decreasing $t$, with a projection onto the levels, from a start near the event point on the eigenvector of $H_p$ (for $g_{t^*}$) whose eigenvalue has the sign of $\alpha$ (Lemma~\ref{lem:morse}), resp.\ over $y_0$ just below a creation. By Remark~\ref{rem:numerics} the trajectories of $V$ have the labels of $W$ unless one converges to a saddle; near passes are tested heuristically and flagged if unresolved (\S\ref{ssec:sepcomp}), and connections along invariant parts of the axes are labelled by Lemma~\ref{lem:axisconn}. Labels are read on $T(f,h)$ at a heuristically chosen level $t_0$, confirmed at a lower level, and their order is checked separately. The answer is tested in exact arithmetic against necessary conditions, such as $n_u$ from the edge polynomials and the parities of the axis crossings of the lines.

None of the hypotheses of Theorem~\ref{thm:algorithm} is certified; this would require a certified complete list of events (certifying each computed event \cite{HS12} does not give it), a certified $t_0$, and validated enclosures of the separatrices \cite{Tuc11} that separate the labels and decide every near pass. A run returns $(n_o,n_u)$ and the pairing (in rare cases only the counts), possibly with \emph{flags}, or stops with an error. A flag is a heuristic diagnostic of a step that may have failed, such as an undecided box, a lost path, a separatrix that passes a saddle on an unresolved side, or an event point that fails a numerical test of Assumption~\ref{ass:morse} (Table~\ref{tab:states}). A flagged answer is unresolved and should be passed to an exact method; the absence of a flag is no proof of correctness. Unlike Algorithm~\ref{alg:main}, the implementation does not take a new lifting at an event point that fails a test: it flags the answer and drops the point or applies the rule of the type given by the computed signs, which is not justified in general (\S\ref{app:flags}). The boundary charts needed when an outer edge has an inner normal with a positive coordinate, such as the edge $[(2,0),(3,1)]$, are not implemented; for such supports the metric lacks the property required in \S\ref{ssec:field}, without a flag, so they are outside the scope of the implementation. The test families have triangular Newton polygons, the inputs of Table~\ref{tab:regression} triangles and quadrilaterals; only input 8 there has such an edge.

\subsection{Experimental design and performance}\label{ssec:performance}\label{ssec:setup}

We compare both solvers with the patchwork of $-\log|c|$ and with two implementations of the projection method \cite{Col75,GVEK96,EKW07}, written for this purpose and not tuned: an exact one (a resultant modulo primes and the Descartes method \cite{CA76}) and a numerical one (a total-degree homotopy in double precision, without error control). We did not compare with a mature certified implementation such as \cite{BEKS13}, so the times show how the costs grow, not that our method is faster than such implementations. The main inputs are random polynomials $\sum_a\varsigma_ae^{\sigma u_a}x^a$ (supports in $\conv\{0,de_1,de_2\}$ containing its vertices, uniform random signs, standard Gaussian $u_a$): a benchmark of $228$ polynomials with $\sigma=1$, $d=6,10,15,20$ and $6$ terms up to dense supports, polynomials with $8$ terms up to $d=80$, a spread family with $\sigma=1,4,16,64$, and $120$ held-out polynomials fixed before the final runs; further families are Harnack curves, near-singular cubics and polynomials with equal coefficients (\S\ref{app:inputs}). The reference for $(n_o,n_u)$ is the exact method on the binary coefficients; where it did not finish, only $n_u$ is known exactly. All runs used one frozen version of the code, with $20$ minutes of wall-clock time per run; two runs at a time shared $6.3$\,GB of memory, and a run killed for lack of memory counts as out of memory (\S\ref{app:runs}).

\begin{figure}[htbp]
\centering
\definecolor{mExact}{HTML}{0072B2}\definecolor{mNum}{HTML}{E69F00}\definecolor{mCplx}{HTML}{009E73}\definecolor{mReal}{HTML}{D55E00}
\pgfplotsset{every axis/.append style={width=0.5\textwidth,height=0.26\textwidth,font=\scriptsize,
  log basis y=10,grid=major,grid style={gray!20,line width=1pt},axis line style={line width=1pt},tick style={line width=1pt,black},title style={yshift=-3pt},
  legend style={font=\scriptsize,draw=none,fill=white,fill opacity=0.8,text opacity=1}},
  pExact/.style={mExact,mark=square,line width=1.2pt,mark options={solid}},
  pNum/.style={mNum,mark=o,line width=1.2pt,densely dashed,mark options={solid}},
  pCplx/.style={mCplx,mark=triangle,line width=1.2pt,densely dotted,mark options={solid}},
  pReal/.style={mReal,mark=diamond*,line width=1.2pt},
  pWrong/.style={only marks,mark=x,mark size=3.2pt,ultra thick,black},
  pFlag/.style={only marks,mark=o,mark size=4pt,line width=1pt,black}}
\begin{tikzpicture}[baseline=(current axis.north)]
\begin{loglogaxis}[xlabel={$d$},ylabel={seconds},title={(a) $|\cA|=8$},xmin=8,xmax=100,xtick={10,20,40,80},xticklabels={10,20,40,80},ymin=0.01,ymax=2000,
  legend to name=timesleg,legend columns=6,legend style={draw=none,fill=none,font=\small,
  /tikz/every even column/.append style={column sep=5pt}}]
\addplot[pExact] coordinates {(10,0.409) (15,2.44) (20,15.3) (30,144)};
\addplot[pNum] coordinates {(10,0.0193) (15,0.0384) (20,0.0887) (30,0.186) (40,0.324) (80,70.5)};
\addplot[pCplx] coordinates {(10,0.0715) (15,0.0972) (20,0.175) (30,0.578) (40,0.763) (80,30.9)};
\addplot[pReal] coordinates {(10,0.0273) (15,0.0337) (20,0.028) (30,0.0667) (40,0.115) (80,0.0755)};
\addplot[pFlag] coordinates {(30,0.578) (80,30.9)};
\addlegendimage{pWrong}
\legend{exact method,numerical method,complex solver,real solver,\hspace{2pt}flagged,\hspace{2pt}wrong}
\end{loglogaxis}
\end{tikzpicture}\hfill
\begin{tikzpicture}[baseline=(current axis.north)]
\begin{loglogaxis}[xlabel={$\sigma$},ylabel={seconds},title={(b) spread, twelve polynomials},log basis x=2,xmin=0.8,xmax=80,xtick={1,4,16,64},xticklabels={1,4,16,64},ymin=0.01,ymax=2000]
\addplot[pExact] coordinates {(1,0.489) (4,0.865) (16,4.76) (64,81.8)};
\addplot[pNum] coordinates {(1,0.0255) (4,0.0413) (16,0.165) (64,10.6)};
\addplot[pCplx] coordinates {(1,0.183) (4,0.155) (16,0.219) (64,2.48)};
\addplot[pReal] coordinates {(1,1.01) (4,0.238) (16,0.0593) (64,0.0507)};
\addplot[pFlag] coordinates {(16,0.219) (64,2.48)};
\addplot[pWrong] coordinates {(16,0.165) (64,10.6)};
\end{loglogaxis}
\end{tikzpicture}\\[4pt]
\pgfplotslegendfromname{timesleg}
\caption{CPU times in seconds (log scales); circles: medians over sets of runs with a flagged answer, crosses: with a wrong answer (against the exact reference) or, without one, with a wrong number $n_u$. (a) Medians over three polynomials with $|\cA|=8$ for each $d$; the exact method exceeded the time limit for $d=40$ and was not run for $d=80$. (b) Medians over the twelve polynomials of the spread family; for $\sigma=64$ the exact method ran out of memory on three. Figure~\ref{fig:timesSM} shows the dense benchmark and Harnack curves.}\label{fig:times}
\end{figure}
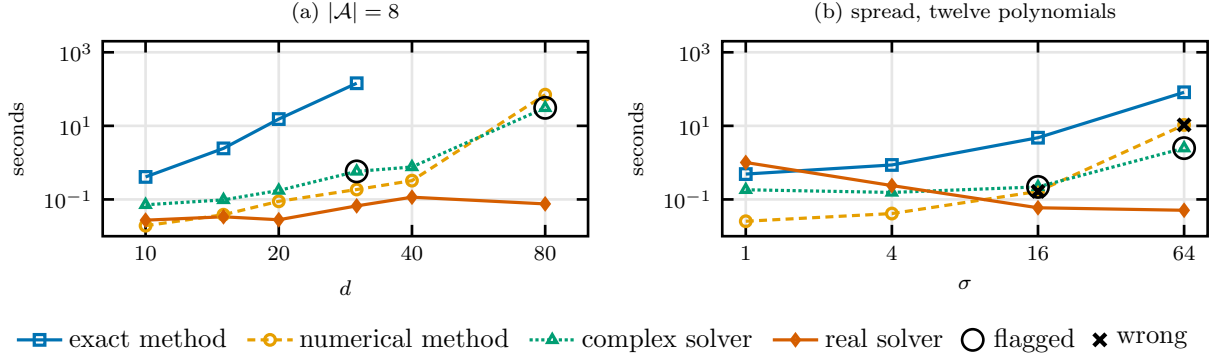

On the tested inputs with $8$ terms, the median time of the real solver remained below $0.12$\,s up to $d=80$, consistent with its search over triples of terms, whereas the exact method needed $144$\,s for $d=30$ and exceeded the time limit for $d=40$ (Figure~\ref{fig:times}(a)); with three inputs per degree, this is an observation, not a complexity bound. The advantage disappears as the number of terms grows: the real solver beats the exact method only up to $8$, $16$, $32$ and $64$ terms for $d=6,10,15,20$, and on dense polynomials it is the slowest method ($968$\,s for $d=20$, against $25$\,s and $0.28$\,s for the exact and the numerical method; Figure~\ref{fig:timesSM}(a)). A larger spread shrinks the window of Theorem~\ref{thm:window}: from $\sigma=1$ to $\sigma=64$ the median time of the real solver drops from $1.0$\,s to $0.051$\,s, whereas that of the exact method grows from $0.48$\,s to $59$\,s for $d=10$, and for $d=15$, $\sigma=64$ it runs out of memory (Figure~\ref{fig:times}(b)). So the method has an advantage for sparse supports of high degree, for large spreads and for curves close to their patchwork (\S\ref{app:perf}), and none for dense supports with small spread, for cubics whose oval dies along the path, or for polynomials with equal coefficients.

The number of events falls with the spread as well: by Remark~\ref{prop:nesting} the event lists of $226$ benchmark paths give the events of $f^{(\sigma)}$ for all $\sigma\ge1$, as far as the lists are complete; the fraction with a detected event drops from $65\%$ for $\sigma=1$ to $3.5\%$ for $\sigma=64$, and $\sigma$ times the mean number of detected event points stays between about $2$ and $4$, consistent with Theorem~\ref{thm:prob}, whose constant $C_1$ is pessimistic (\S\ref{app:reliability}).

\subsection{Validation and limitations}\label{ssec:correctness}

\begin{table}[htbp]
\centering\small
\caption{Correctness of the answers. (a) Counts $(n_o,n_u)$ on the inputs with an exact reference (six families of \S\ref{app:inputs}, by family in Table~\ref{tab:outcomesfull}; held-out family): correct answers of the projection methods and of the patchwork (not computed for equal coefficients), and outcomes of our solvers (c: correct without a flag, cf: correct with a flag, wf: wrong with a flag, w: wrong without a flag, --: no answer). (b) Pairing on the inputs with lines on which both solvers answered: answers that agree with an exact reference (first row) or with an uncertified tracing (second row; where it was inconclusive, a check by hand, resp.\ the pairing of input 12 carried over by an exact change of coordinates), and answers that pass the tests of \S\ref{app:checks} (third row). Only the first row is an exact reference.}\label{tab:outcomes}
\setlength{\tabcolsep}{2.6pt}
\begin{tabular}{@{}lrrrrrrrrrrrrrr@{}}
\multicolumn{15}{@{}l}{(a) counts}\\
\toprule
 & & \multicolumn{2}{c}{projection} & \multicolumn{5}{c}{complex solver} & \multicolumn{5}{c}{real solver} & \\
\cmidrule(lr){3-4}\cmidrule(lr){5-9}\cmidrule(lr){10-14}
inputs & number & exact & numerical & c & cf & wf & w & -- & c & cf & wf & w & -- & patchwork \\
\midrule
six families & 403 & 403 & 380 & 370 & 24 & 7 & 0 & 2 & 392 & 7 & 2 & 0 & 2 & 134 \\
held-out & 118 & 118 & 118 & 115 & 3 & 0 & 0 & 0 & 117 & 0 & 1 & 0 & 0 & 55 \\
\bottomrule
\end{tabular}

\bigskip
\begin{tabular}{@{}llrrr@{}}
\multicolumn{5}{@{}l}{(b) pairing}\\
\toprule
compared with & inputs & with lines & complex & real \\
\midrule
exact smoothing & line arrangements (\S\ref{app:lines}) & 72 & 70 & 70 \\
numerical tracing & Table~\ref{tab:regression}, axis stress test, rescaled connections & 186 & 186 & 186 \\
necessary conditions & six families of (a), Table~\ref{tab:noref} & 310 & 310 & 310 \\
\bottomrule
\end{tabular}
\end{table}

\emph{Counts.} On the $403$ inputs of the six families with an exact reference, every wrong answer of our solvers was flagged ($7$ of the complex and $2$ of the real solver), and $31$ of the $40$ flagged answers were correct; each solver stopped with an error on two benchmark polynomials with $d=20$ and $181$ terms, and the numerical method returned $21$ wrong answers without a warning. This is an observation, not a guarantee; \S\ref{app:reliability} lists the causes of the wrong answers. On the $118$ held-out inputs with an exact reference, the complex solver was correct on all inputs and the real solver on all but one, which was flagged; on the other two ($d=30$) the four other methods agreed, with the number of unbounded components checked exactly. On $13$ inputs of the six families the exact method did not finish, and on $8$ of them no independent method confirms the counts (Table~\ref{tab:noref}).

\emph{Pairing.} An exact reference for the pairing exists only for the $72$ perturbed line arrangements, where it follows from the arrangement for every perturbation of the given sign (\S\ref{app:lines}); both solvers paired the ends correctly on $70$, and their answers on the other two are flagged. A numerical tracing of $C$, attempted on every input of three further test families, was inconclusive on two inputs, and wherever it was conclusive it agreed with both solvers; as it is not certified (on the line arrangements it disagrees with the exact pairing on $43$ inputs), this is agreement, not an independent reference. On the other families only necessary conditions and the agreement of the two solvers are checked.

\emph{Special configurations and limitations.} Saddle connections, several event points at one parameter, event points on the axes, events within the rounding of $s=1$ and unusual pairings are rare in the random families; the test families of \S\ref{app:special} produce them, and there neither solver returned a wrong answer without a flag (Table~\ref{tab:coverage}). Most of these families entered the suite because an earlier version of the code failed on them, often without a flag; so they are not a random sample, and the flags of one version say little about errors that it does not anticipate. Changing the tolerances of the integration, the starts of the separatrices or the level $t_0$, one at a time in eight variants, for $218$ solver--input pairs ($1962$ runs) changed no answer, pairing or set of flags (\S\ref{app:reliability}). Besides the missing certificates and the scope of \S\ref{ssec:implsummary}, separatrices can be stiff and passes close to a saddle are only flagged, connections forced by symmetries other than an invariant axis are not treated, the start of the complex solver loses critical points for large spreads, and the isotopy type is not implemented.

\section*{Acknowledgements}
The author is grateful to Xiaohong Jia and J.~Maurice Rojas for many helpful discussions and valuable suggestions. Their insights helped clarify the ideas and improve the presentation of this work. AI tools were used to assist with literature searches, algorithm implementation, data processing, figure generation, and the drafting and revision of the manuscript. The main ideas of this work originated with the author, who checked and verified all content, including the mathematical arguments, computational results, figures, and references. The author assumes responsibility for all content.

\FloatBarrier

\clearpage
\setcounter{section}{0}\setcounter{table}{0}\setcounter{figure}{0}\setcounter{equation}{0}
\renewcommand{\thesection}{SM\arabic{section}}\renewcommand{\theHsection}{SM\arabic{section}}
\renewcommand{\thetable}{SM\arabic{table}}\renewcommand{\theHtable}{SM\arabic{table}}
\renewcommand{\thefigure}{SM\arabic{figure}}\renewcommand{\theHfigure}{SM\arabic{figure}}
\renewcommand{\theequation}{SM\arabic{equation}}\renewcommand{\theHequation}{SM\arabic{equation}}
\begin{center}
{\large\bfseries Supplementary materials}
\end{center}
\medskip
\noindent These supplementary materials describe the implementation (\S\ref{app:impl}), report in full the numerical evaluation summarized in Section~\ref{sec:experiments} (\S\ref{app:eval}), describe the tests of special configurations (\S\ref{app:special}), give the details of Example~\ref{ex:infinity} (\S\ref{app:worked}), and give complements to Section~\ref{sec:lifting} (\S\ref{app:kurtz}). Sections, statements and tables numbered with the prefix SM are those of the supplementary materials; all other numbers refer to the main text.

\section{Implementation}\label{app:impl}

Besides the events, Algorithm~\ref{alg:main} integrates at most $2N$ trajectories for $N$ event points, $N\le B_\cA K$ (proof of Theorem~\ref{thm:prob}), and its state has $O(n_0+N)$ entries, $n_0$ the number of components of $T(f,h)$; stiff separatrices, near connections and close events can need much more work than these counts suggest. Our implementation certifies none of the hypotheses of Theorem~\ref{thm:algorithm} (\S\ref{ssec:implsummary}); validated enclosures of the separatrices \cite{Tuc11} would also have to separate the labels and decide every near pass. This section describes what the implementation computes, which decisions it takes numerically and what it does when they fail.

\subsection{Inputs, return states and flags}\label{app:flags}

\emph{Inputs.} The implementation takes the support, the coefficients as binary doubles and $h=-\log|c|$, perturbed as in Remark~\ref{rem:freedom}(ii) if it is not generic, and assumes (H0). A lifting given with the input (inputs 1, 5c, 5d of \S\ref{app:regression}, Example~\ref{ex:infinity}, the rescaled connections) is used without a test of genericity. $E_\cA(c)\neq0$ is not tested beforehand, but a critical value that vanishes to the precision of the endpoint test is flagged (\S\ref{ssec:eventcomp}). The charts $U_e$ for normals $\nu_e$ with a positive coordinate are not implemented; for such a support (input 8) the metric of \S\ref{ssec:sepcomp} lacks the property required in \S\ref{ssec:field}, without a flag.

\emph{Return states and flags} (Table~\ref{tab:states}). A flag is a heuristic diagnostic of a step that may have failed. A flagged answer is unresolved and should be passed to an exact method such as \cite{BEKS13}, which is not done automatically; an answer without a flag is not certified either. The flag classes separate a property of the support (a point on a detached edge is not Morse for any lifting: \emph{structural}), a point that fails a numerical test, which may be non-Morse for this lifting, as happens for almost no lifting (Proposition~\ref{prop:genericmorse}(i)), or Morse but badly conditioned (\emph{unclassified}), and inputs that violate (A1) (\emph{input}). The tests of Assumption~\ref{ass:morse} compare the determinant of the Hessian in $\log|x|$ with $10^{-12}$ times the square of its largest entry, and $\alpha$, resp.\ $(\psi^e)''$, $\partial_zg^e$ and $\alpha$ at a boundary event point, with $10^{-12}$ of the absolute values of their terms; on an invariant part of an axis and next to the supporting lines $\alpha$ is not tested.

\begin{table}[htbp]
\centering\small
\caption{Return states of a run, and the classes of flags.}\label{tab:states}
\begin{tabular}{@{}l>{\raggedright\arraybackslash}p{11.4cm}@{}}
\toprule
return state & output \\
\midrule
complete & $(n_o,n_u)$ and the pairing of the ends; no flag \\
flagged & $(n_o,n_u)$ and the pairing, with at least one flag \\
counts only & $(n_o,n_u)$ without the pairing (the roots of the edge polynomials could not be continued to $s=1$); flagged \\
error & no answer: a step or a test of \S\ref{app:checks} failed \\
\midrule
flag class & raised for \\
\midrule
input & a critical value of $f$ that is zero to $40$ digits of the terms (computed with $50$): (A1) not decided \\
structural & event points on a detached edge, which violate Assumption~\ref{ass:morse} for every lifting (Proposition~\ref{prop:genericmorse}(ii)) \\
unclassified & event points that fail a numerical test of Assumption~\ref{ass:morse}; clusters; event points whose identity with, or order relative to, another one is not resolved (\S\ref{ssec:eventcomp}) \\
numerical & other steps that may have failed: unresolved boxes, full buffers, lost paths, missed start points and unrefined candidates of the complex solver, boundary events not followed in their chart, isolated collinear circuits near their discriminant, events at the endpoint not resolved or their search incomplete, unresolved passes near saddles, landings on births not confirmed, landing points not separated or reordered, pairing not computed \\
\bottomrule
\end{tabular}
\end{table}

\emph{Differences from Algorithm~\ref{alg:main}.} (i) At a point that fails a test the implementation neither takes a new lifting (line~\ref{alg:restart}) nor returns \textsc{non-Morse}: it flags the answer and drops the point or applies the rule of the type given by the computed signs. On a detached axis edge this reproduces Remark~\ref{rem:detached}; on a detached outer edge the sides $Q_m,Q_o$ of a boundary event are not defined, and an annihilation or a creation there stops the computation with an error; elsewhere the rule is not justified. (ii) If the computation or a test fails on a perturbed lifting, it is repeated with another direction $\eta$, up to three attempts; all attempts are recorded, and the return state is that of the last one. An unexpected error is recorded but not retried. (iii) Event points at numerically equal parameters are processed as one event (\S\ref{ssec:eventcomp}). (iv) The answer can lack the pairing.

\begin{remark}[detached edges]\label{rem:detached}
Let $q$ be a singular point on the axis $\{x_2=0\}$ of a detached $E_1$, resp.\ a boundary event point on a detached outer edge $e$; let $z=x_2$, resp.\ $z$ as in $U_e$, $\phi_t=f_t(x_1,0)$, resp.\ $\psi^e_t$, and $m\ge3$, resp.\ $m\ge2$, as in the proof of Proposition~\ref{prop:genericmorse}. Then $f_t$, resp.\ $g^e_t$, is $\phi_t+z^mG_t$ with $G_t$ smooth, and for almost every $h$ the root $q$ of $\phi_{t^*}$ is double with $\alpha=\partial_t\phi_t(q)$, $\beta=\phi_{t^*}''(q)$ and $G_{t^*}(q,0)$ non-zero. The substitution $Z=z|G_t|^{1/m}$, followed by $Z\mapsto Z|Z|^{m/2-1}$ for even and $Z\mapsto Z^m$ for odd $m$, is a homeomorphism near $(q,0)$, continuous in $t$, that preserves $\{z=0\}$ and its sides and maps the curves to $\{\phi_t\pm Z^2=0\}$, resp.\ $\{\phi_t\pm Z=0\}$. So for odd $m$ the topology changes as at a tangency with the axis ($\Sigma_t$ does not change), resp.\ as at a point of type (B) (Proposition~\ref{prop:boundary}); for even $m$ as at a Morse point with the signs of $\alpha$, $\beta$ and $G_{t^*}(q,0)$, which for $e$ is a singular point of $\overline{C_{t^*}}$ on $D_e$. Section~\ref{sec:global} does not cover these points; the implementation applies these rules on the axes, with a flag. If $\Delta=\conv\{0,de_1,de_2\}$, a generic translation $x\mapsto x+p$ removes the detached edges and keeps (H0) and $E_\cA(c)\neq0$, at the cost of sparsity.
\end{remark}

\subsection{Event search and the event list}\label{ssec:eventcomp}

By Theorem~\ref{thm:window}, if $g(\ell)>0$ every event satisfies $1<s\le r/g(\ell)$, a singular point in $(\R^*)^2$ even $s\le r_\Delta/g(\ell)$, and if $g(\ell)>r$ there are none (Corollary~\ref{cor:window}). The implementation searches for interior events up to $(r_\Delta+2d_0)/g$ and for boundary events and events on the axes up to $(r+2d_0)/g$, where $g$ is the gap of the lifting used, $d_0=0$ for $-\log|c|$, and $d_0$ is half the range of $\varepsilon\eta$ for the perturbed lifting; the complex solver follows its paths to the smaller of this bound and a second bound for interior events, but at least to $s=2$. Each bound $W$ is computed in double precision and attained by circuits (Remark~\ref{rem:window}), so it is enlarged by $10^{-9}\max(1,W)$, and Corollary~\ref{cor:window} is applied only if the enlarged bound is at most $1$. Neither solver uses directed rounding or validated tracking, so neither proves that its list of events is complete; certifying the points found would not prove it either (\S\ref{ssec:implsummary}).

\emph{Interior events} solve $f^{(s)}=\partial_1f^{(s)}=\partial_2f^{(s)}=0$ in $(\R^*)^2$; their type follows from the Hessian and $\alpha$ (Lemma~\ref{lem:morse}(i)). The \emph{real solver} uses the dominant terms. In $u=(\log|x_1|,\log|x_2|)$ on $\cO_\epsilon$ the system reads $\sum_aT_a=\sum_aa\,T_a=0$, $T_a=\sigma_\epsilon(a)|c_a|^se^{\langle a,u\rangle}$, so if $|T_{a_1}|\ge|T_{a_2}|\ge\dots$, then $|T_{a_1}|\le(|\cA|-1)|T_{a_2}|$ and $|T_{a_2}|\le(|\cA|-2)d\,|T_{a_3}|$; by Lemma~\ref{lem:cocircuit}(ii) the three largest terms, and a fourth one if their exponents are affinely independent, also lie within the factor $e^{r_\Delta}$ of each other. For affinely independent $\tau=\{a_1,a_2,a_3\}$ the \emph{tube} of $\tau$ is the set where $X=\log|T_{a_2}/T_{a_1}|$, $Y=\log|T_{a_3}/T_{a_1}|$ satisfy $|X|,|Y|\le\lambda_3+0.1$, $\lambda_3=\min(\log((|\cA|-1)(|\cA|-2)d),r_\Delta)$ with $r_\Delta$ enlarged like the window bounds, restricted to the intervals of $s$ on which, for $|X|,|Y|\le\lambda_3$, no other term can exceed the largest term of $\tau$ and some fourth term can lie within the factor $e^{\lambda_4}$ of it, $\lambda_4=\min(\log(\|M_\tau^{-1}\|_\infty(|\cA|-3)\max_a\|a\|_\infty),r_\Delta)$, $M_\tau$ with the columns $(1,a_i)$. These bound differences of logarithms of terms; in $u$ the tube is wide for thin triangles, up to $(2\lambda_3+0.2)\|D_\tau^{-1}\|$, $D_\tau$ with the rows $a_2-a_1$, $a_3-a_1$. A collinear triple uses coordinates along and across its line; on an edge of $\Delta$ its tube is cut where the terms off the line are below $e^{-30}$ of those on it, and zeros beyond the cut are found in a chart at the line (below). The tubes are searched in the window, in the quadrants with $K_\epsilon\neq\emptyset$ (\S\ref{ssec:hk}), by adaptive subdivision with interval bounds and the Krawczyk test \cite{Tuc11}, in floating point without directed rounding; boxes that remain undecided (Table~\ref{tab:params}) are unresolved and flagged, and so is a full buffer of roots.

The \emph{complex solver} computes the critical points of $f$ in $(\C^*)^2$ by a total-degree homotopy \cite{SoWa05}, each equation rescaled by a power of two if its largest coefficient lies outside $[2^{-16},2^{16}]$, and follows them in $\log x$ along $s+i\,10^{-8}\sqrt s\,w(s)$, $w$ rising smoothly from $0$ to $1$ on $[1,1.02]$. This path avoids, by about $10^{-8}$, the branch points of the critical locus over the real line, collisions of critical points that are not events but where the critical points move with unbounded speed (input 13). Paths that run into a divisor at a boundary event are continued through it in the chart of the edge; if another path is lost, all are followed again with $10^{-7}$ in place of $10^{-8}$. A sign change of the real part of a critical value along a nearly real path is refined by Newton's method on the real system. Flagged are a failed continuation, losses in both passes, a candidate that is not refined (also in the chart of a face or an edge, unless it lies at the parameter of a boundary event, which it then is), a full buffer of sign changes, and a start with fewer points than the mixed volume of $x_1\partial_1f$, $x_2\partial_2f$ (Bernstein's bound, as in \cite{HuSt95}); the last happens for large spreads.

\emph{Boundary events} on an outer edge $e$ with at least three points of $\cA$ are the parameters at which $\psi^e_t$ has a real multiple root $y_0\neq0$. Not every zero of $\Delta_{\cA\cap e}(c(t))$ is one (\S\ref{app:otherrays}; input 10), so both solvers compute the real double roots directly, by subdivision in $(\log|y|,s)$. The type follows from $\alpha\beta$ (Lemma~\ref{lem:morse}(ii)), and the sides $Q_m$, $Q_o$ from the signs of $\partial_zg^e(y_0,0)$ and of $\psi^e_t$ near $y_0$. The ends are labelled by $e$, the side $Q$ and the rank of their root among the real roots of $\psi^e_t$, which changes only at boundary events on $e$.

\emph{Events on the axes.} A point $(p,0)$ is singular if and only if $p$ is a multiple root of $f_t(x_1,0)$ and $\partial_2f_t(p,0)=0$. If no exponent has $a_2=1$, the support certifies that the axis is invariant under $V$ for a metric diagonal in $(x_1,x_2)$, and every real double root of $f_t(x_1,0)$ is singular; if the first term $b\,x_2^k$ of the expansion of $f_t$ in $x_2$ has $k=2$ (and $\partial_1^2f_t(p,0)\neq0$), it is a Morse point, and otherwise $E_1$ is detached or, for almost no lifting, a coefficient vanishes at $p$ (Proposition~\ref{prop:genericmorse}). A point with odd $k\ge3$ is dropped, since $\Sigma_t$ does not change there (Remark~\ref{rem:detached}), and one with even $k\ge4$ is kept with the type given by the signs of $\partial_1^2f_t$, $b$ and $\alpha$, which assumes that $k$ is the smallest exponent $a_2>0$; both are flagged. If some exponent has $a_2=1$, a multiple root, computed like the boundary events, is an event point if the terms of $\partial_2f_t(p,0)$ cancel to within $10^{-8}$ of their absolute values (input 8), and it is classified by its affine Hessian with a test that does not depend on a scaling of the coordinates (input 15); a point that fails it is dropped, with a flag. An axis can also be invariant by an identity of the coefficients (input 12: $\partial_2f^{(s)}(x_1,0)=-f^{(s)}(x_1,0)$). This is not recognized; instead the two roots of $f_t(x_1,0)$ that meet at the point are continued a little in $s$, and if the terms of $\partial_2f_t$ cancel at the points reached, the component through the point is treated as invariant, a heuristic that raises no flag. A critical zero exponentially close to the line $z=0$ of an axis or a divisor, with $f=y^kz^mg$, which occurs for large spreads, is found by Newton's method on $g=y\partial_yg=z\partial_zg=0$ in $(\log|y|,\log|z|)$, each equation scaled by its terms, started from the double roots of $g(y,0)$; the solutions at which the terms on the line are the largest are kept and classified in the same chart, and a degenerate one is dropped with a flag.

\emph{The event list.} Each event point is recorded with its type, its parameter (also as a hexadecimal double), its chart and position, and its event. Copies of a point found in several tubes, along several paths or in adjacent boxes are one point if they lie in one chart and agree to $10^{-8}$ in $s$ and $10^{-6}$ in position, relative. A critical zero in $\log|x|$ next to an event point on an axis with exponents $a_2=1$ is that point seen from the torus, and is dropped, if the two agree to $10^{-6}$ in $s$ and $10^{-5}$ in the affine position; if they agree only to $10^{-3}$, they are neither identified nor separated: the zero is kept, and the answer is flagged. Closeness alone is not taken as identity: the two saddles of input 7a lie at different places and $1.3\cdot10^{-10}$ apart in $s$. Two critical zeros in one quadrant within $10^{-9}$ in $s$ and $10^{-3}$ in $u$, or a zero in the chart of a supporting line and an event point on that line at that distance, form a \emph{cluster}, and the answer is flagged (input 7b); other pairs, such as two event points on one line, are not tested for this. Event points whose parameters agree to $10^{-9}$, relative, are processed as one event (Theorem~\ref{thm:update}), annihilations first and creations last, as the continuation of the edge roots needs. This is harmless for points whose rules commute. Their order matters if a separatrix of one of them could land on the slice of a birth or creation among them, or if a creation and an annihilation lie on one supporting line; such pairs that are not images of each other under a symmetry forced by the support (\S\ref{app:otherrays}) have no resolved order, and the answer is flagged.

\emph{Events near $s=1$.} An event within the rounding of $s=1$ cannot be put on either side in double precision (input 14: $s^*-1\approx1.3\cdot10^{-16}$). If $v(s)$ is the critical value of $f^{(s)}$ at the critical point that becomes singular at $s^*$, then $dv/ds=\alpha_s:=\partial_sf^{(s)}$ there, so $v(1)\approx\alpha_s(1-s^*)$, and the event lies on the path if and only if $v(1)\alpha_s<0$. This assumes that the critical branch exists on $[1,s^*]$ and that $\alpha_s$ keeps its sign there, which holds if $|s^*-1|$ is small but is not verified. For every event within $10^{-9}$ of $s=1$, and for the candidates of a missed event (the real critical points of the polynomials of the outer edges and the invariant axes, and the critical points of the start of the complex solver with a critical value below $10^{-8}$ of the terms), $v(1)$ is computed with $50$ significant digits, and $s^*-1=-v(1)/\alpha_s$. An event on the path gets the parameter $\max(s,s_+)$, where $s_+$ is the double after $1$, and keeps $s^*-1$ for its order; an event off the path is removed. If $|v(1)|$ is below $10^{-40}$ of the terms, (A1) is not decided (flag of class input); an unresolved side or an unrefined candidate is flagged too.

\subsection{Separatrices}\label{ssec:sepcomp}

Theorem~\ref{thm:update} needs, at each saddle and creation, the labels of the landing points of its descending separatrices and their order on the bottom. The landing points are those of the field $W$ of Proposition~\ref{prop:W}; we integrate the normal lift $V$, which has the same labels if no descending separatrix converges to a saddle (Remark~\ref{rem:numerics}), and a connection along an invariant part of an axis is labelled by Lemma~\ref{lem:axisconn}. We use the metric $\sum_id\xi_i^2$, $\xi_i=\operatorname{asinh}(x_i/\rho_i)$, $\rho_i=e^{s\bar r_i}$, where $\bar r_i$ is smaller than the $i$-th coordinates of all vertices of $\Trop(h)$. It is flat in $\log|x|$ far from the axes, hence of the kind required in \S\ref{ssec:field} if no $\nu_e$ has a positive coordinate, and diagonal in $(x_1,x_2)$ near the axes, which trajectories can cross. In the coordinates $\xi$ the field is
\[
\frac{d\xi}{ds}=-\frac{\partial_sf^{(s)}}{|\nabla_\xi f^{(s)}|^2}\,\nabla_\xi f^{(s)}-(\bar r_i\tanh\xi_i)_i,
\]
the last term coming from the dependence of $\xi$ on $s$. Saddles on a part of an axis that is not invariant are treated like interior saddles (input 8).

\emph{Starts and integration.} A descending separatrix of an interior saddle starts on the eigenvector of the Hessian in $\xi$ whose eigenvalue $\lambda$ has the sign of $\alpha$, at the distance $(2|\alpha_s|\delta/|\lambda|)^{1/2}$ of the level $s+\delta$ in the quadratic model, $\delta=10^{-8}\max(1,s)$. The separatrix of a creation over $y_0$ starts on the level $s^*+\delta_c$, on the side $Q_o$. Near the boundary saddle the level curves in $(\log|y|,\log|z|)$ are U-shaped, with curvature $|\beta|/(|\alpha|\,|s-s^*|)$ at the bottom, so $V$ contracts across the separatrix at the rate $|\beta|/(|\alpha|(s-s^*)^2)$, and an explicit method started at $\delta_c$ needs about $|\beta|/(|\alpha|\delta_c)$ steps; we take $\delta_c$ for about $4000$ steps, within the bounds of Table~\ref{tab:params}. Both starts rely on the contraction of the flow towards the separatrix to damp the $O(\delta)$ error of the start, which we do not bound; \S\ref{app:reliability} varies $\delta$ and $\delta_c$. The trajectories are integrated by the Dormand--Prince method, and a step is accepted only if a damped Newton projection onto $\{f^{(s)}=0\}$ brings the residual, relative to the terms, below $10^{-10}$ with a small move. A failed integration is restarted further from its event, at most three times; then the computation stops with an error.

\emph{Near passes.} A connection that is not forced shows as a separatrix that passes a lower saddle, or the boundary saddle of an annihilation, within the accuracy of the computation. If on a level just above a lower interior saddle $q$ a separatrix comes within $10^{-3}$ of $q$ in $\xi$, its component along the descending eigenvector of $q$ gives the side, and below $10^{-6}$, relative, the side is not resolved and the answer is flagged; near the boundary saddle of an annihilation the same test uses $\log|y|$ in the chart of the edge and the ascending separatrix of the model of Proposition~\ref{prop:W}(ii).

\emph{Separatrices on an invariant axis.} If a saddle on an invariant part of an axis has its descending eigenvector along the axis, both descending separatrices lie on the axis. They are followed as the two roots of $f_t(x_1,0)$ born at $p$ down to $t_0$, where they land at crossing points of the bottom with the axis, labelled by those of $T(f,h)$. If such a root merges with another one at a lower event point $q$ on the axis, the separatrix lands on the slice of a birth, or it runs into a saddle whose separatrices leave the axis and gets the label of Lemma~\ref{lem:axisconn}.

\subsection{Common bottom, landing order and state updates}\label{ssec:labels}

\emph{Slices and the bottom.} A separatrix is integrated down to the level $t_0$, or until it meets the slice $B_m$ of a birth or creation $m$, at the distance $\delta_m$ from it in $s$ (at most a twentieth of the distance $\gamma_m$ to the nearest other event); there its label is its angle around $m$, resp.\ its coordinate $y$ on $\lambda$, in the quadratic model of $m$. A landing on the slice of a birth counts only if the trajectory is still on the shrinking oval at $\delta_m/8$, or, if this check cannot be integrated, with a flag. On $\hat C_{t_0}$ the label is read on $T(f,h)$ as in Definition~\ref{def:labels}: the component, the nearest piece of $\Gamma_\epsilon$ to $u(b)/s_0$ and the coordinate along it, or the crossing point of $T(f,h)$ with an axis. Since $T(f,h)$ is computed once, all labels are comparable, and inserting a label into the ordered list of its component gives the intervals $J^\pm_b$ of Theorem~\ref{thm:update}. Lemmas~\ref{lem:bottom} and~\ref{lem:positions} justify these labels only for $t_0$ small enough, with constants that we do not compute, so $t_0$ is chosen heuristically: $s_0=1+\log(1/t_0)=\max(1.25s_1+1,30/\phi,10)$, where $s_1$ is the parameter of the first event and $\phi$ the smallest distance between non-adjacent pieces of a $\Gamma_\epsilon$ and length of a bounded piece. A label is accepted only if its piece is clearly the nearest (Table~\ref{tab:params}) and continuing the separatrix to $2s_0$ gives the same component and the same or an adjacent piece. Otherwise $s_0\mapsto4s_0$, and all separatrices that land off the axes are continued to the new level, so that their labels refer to one bottom; after five increases the computation stops with an error. A separatrix on an invariant axis lands at the nearest crossing point of $T(f,h)$ once this is clear in the same sense. Passing these tests does not prove that $t_0$ is small enough.

\emph{The order of the landing points} is checked separately, since keeping the component and the piece from $s_0$ to $2s_0$ does not confirm the order of several points. Landing points on one component whose coordinates agree to $10^{-6}$ are ordered at the largest common level $s=1.04^g$ at which they are separated, by the side of one from the other along the level curve, and a pair separated at no common level is flagged; then the order of all points of the component at $s_0$ is compared with that at $2s_0$, cyclically on an oval, leaving out points whose labels at $2s_0$ agree to $10^{-6}$, and a change is flagged. On the slice of a birth the separatrices approach the minimum along the slow direction of its Hessian, so two that arrive from the same side can have angles that agree to the accuracy of the integration, whereas the flow preserves their order on the oval. So each landing point on a slice is also recorded at the distances $4^j\delta_m$, $j\le12$, below $\gamma_m/4$, and the points of a slice are ordered at the highest of these levels at which all of them were recorded; neighbours closer than $10^{-6}$ there (of a turn, resp.\ relative to $\max(1,|y|)$) and changes of the order against that on the slice itself, for points more than $10^{-4}$ apart there, are flagged. The checks are heuristic: an inaccurate trajectory can be wrong at all levels.

\emph{Connections along an axis and state updates.} In Lemma~\ref{lem:axisconn}, $\iota$ is the sign of $x_i-q_i$ for the root at two levels just above $q$: the two signs must agree, and $|\log|x_i|-\log|q_i||$ must follow the square-root law of two merging roots; $T$ is the tangent of the level at the start of $\gamma_q^+$, whose component along the axis must be at least a tenth of its length. Otherwise the computation stops with an error; if the root runs into a flagged point with even $k\ge4$, the rule of its sign type is applied. The state of Theorem~\ref{thm:update} is a list of components, each a sequence of bottom intervals and, for a line, its two ends; at a death the oval is counted (Algorithm~\ref{alg:main}). A landing point is used only through its label and its order, so the integration has to be accurate only up to the separation of the labels, which we do not certify.

\subsection{Tests of the answer, and the parameters}\label{app:checks}

The answer is checked, in exact rational arithmetic on the binary coefficients, against conditions that every smooth curve satisfies: $n_o\ge0$; $n_u$ is the number of real points of $\overline C$ at infinity, i.e.\ of non-zero real roots of the polynomials of the outer edges (Descartes' method after a squarefree test modulo primes); every end lies on exactly one line (Corollary~\ref{cor:ends}); a line whose ends lie on opposite sides of an axis $\{x_j=0\}$ crosses it an odd number of times and every other component an even number, so, if the real roots of $f|_{x_j=0}$ are simple, the number of such lines is at most theirs and has its parity; and if $\Delta=\conv\{0,de_1,de_2\}$, gluing the ends at infinity gives exactly $d\bmod2$ one-sided circles in $\RP^2$, a circle being one-sided if and only if it meets the line at infinity an odd number of times. A test whose polynomial is not certified squarefree is skipped, and this is recorded. A failed test stops the computation with an error; passing the tests does not certify the answer.

\begin{table}[htbp]
\centering\scriptsize
\caption{Numerical parameters (relative unless stated; in brackets the normalization) and their effect. Flag: the answer is flagged; error: no answer.}\label{tab:params}
\setlength{\tabcolsep}{3pt}
\begin{tabular}{@{}>{\raggedright\arraybackslash}p{2.6cm}>{\raggedright\arraybackslash}p{8.6cm}>{\raggedright\arraybackslash}p{4.1cm}@{}}
\toprule
quantity & value & effect \\
\midrule
Morse tests & $10^{-12}$ [Hessian determinant: square of the largest entry; else terms]; next to a supporting line $10^{-8}$ [terms] & flag; point dropped or rule by signs \\
singular point on an axis with $a_2=1$ & $10^{-8}$ [terms of $\partial_2f_t$] & accepted as event point \\
invariance of an axis & roots continued to $10^{-3},10^{-2},10^{-1}$ in $s$; cancellation $10^{-8}$ & treated as invariant, no flag \\
window, endpoint & enlargement $10^{-9}$ [$\max(1,W)$]; events within $10^{-9}$ of $s=1$ and candidates below $10^{-8}$ [terms] refined with $50$ digits; $|v(1)|$ below $10^{-40}$ [terms] & Corollary~\ref{cor:window} only if $\le1$; side of the event; (A1) undecided: flag \\
real solver & tube margin $0.1$, cut $e^{-30}$; boxes to width $10^{-10}$ [$\max(1,s)$ in $s$], $4\cdot10^5$ per tube and quadrant; $64$ roots per quadrant & unresolved boxes, full buffer: flag \\
complex solver & offset $10^{-8}\sqrt s$, ramp on $[1,1.02]$, retry $10^{-7}$; $64$ sign changes per step; chart candidates within $10^{-6}$ of a boundary event & lost in both passes, full buffer, candidate not refined: flag \\
identity & copies: $10^{-8}$ in $s$ [$1+s$], $10^{-6}$ in position [$1+|u|$]; torus zero at an axis event: $10^{-6}$ in $s$, $10^{-5}$ in $x$, ambiguous up to $10^{-3}$ & merged; dropped; ambiguous: flag \\
events, clusters & one event: $10^{-9}$ in $s$ [$\max(1,s)$]; cluster: also $10^{-3}$ in position & processed together, order checked; cluster: flag \\
integration & rtol $10^{-8}$, atol $10^{-10}$; residual $10^{-10}$ [terms], projection exit $10^{-12}$, move $\le10^4\times$ step tolerance, $6$ halvings; $2\cdot10^5$ steps per call & step rejected; failure: restart \\
starts & saddle: $\delta=10^{-8}$ [$\max(1,s)$], distance $\le10^{-3}$ in $\xi$, on an invariant axis $10^{-6}$; creation: about $4000$ steps, $\delta_c\ge10^{-8}$, $\le0.05$ [$\max(1,s^*)$], at most $\frac1{20}$ of the distance to the next event; restarts with $100\delta$, $10\delta_c$, at most $3$ & then error \\
near passes & $10^{-3}$ in $\xi$ (boundary saddles: in $\log|y|$, $|z|$ within $e^5$); side below $10^{-6}$ & flag \\
bottom & $s_0=\max(1.25s_1+1,30/\phi,10)$; nearest piece closer than $\frac14$ of the distance to non-adjacent ones and within $\max(10,s_0/20)$; confirmation at $2s_0$ with at most $\max(2\cdot10^5,5ks_0)$ steps, $k$ the steps per unit of $s$ before $s_0$; $s_0\mapsto4s_0$ at most $5$ times & label rejected; then error \\
order of landings & bottom: ties $10^{-6}$, levels $1.04^g$, separation $10^{-8}$; slices: $\delta_m=10^{-7}$ (births), $10^{-4}$ (creations) [$\max(1,s_m)$], $\le\gamma_m/20$, check at $\delta_m/8$, levels $4^j\delta_m\le\gamma_m/4$, $j\le12$, ties $10^{-6}$, compared beyond $10^{-4}$ & unresolved or changed: flag \\
attempts & $3$ on a perturbed lifting & then error \\
\bottomrule
\end{tabular}
\end{table}

\section{Numerical evaluation}\label{app:eval}

This section reports in full the evaluation summarized in Section~\ref{sec:experiments}: of both solvers of \S\ref{ssec:eventcomp}, of the patchwork and of two implementations of the projection method \cite{Col75,GVEK96,EKW07}, written for this purpose and not tuned. The \emph{exact} one computes $\operatorname{Res}_{x_2}(f,\partial_2f)$ modulo primes up to the Hadamard bound, after a shear if needed, and isolates its real roots by the Descartes method \cite{CA76}; the \emph{numerical} one computes the $d(d-1)$ critical points of the projection by a total-degree homotopy in double precision, without error control or warnings, so its wrong answers say nothing about projection methods in general. We did not compare with a mature certified implementation such as \cite{BEKS13}: the comparisons show how the times grow with $d$, $|\cA|$ and $\sigma$, not that our method is faster than such implementations.

\subsection{Inputs and references}\label{app:inputs}
For a degree $d$ and $n\ge3$, $\cA$ consists of $(0,0),(d,0),(0,d)$ and the first $n-3$ other lattice points of their convex hull in a random order, such that an edge of $\Delta$ with a further point of $\cA$ has one at lattice distance one from it (so singular points on the axes are not generic, \S\ref{ssec:eventcomp}), and $c_a=\varsigma_ae^{\sigma u_a}$ with independent uniform signs and standard Gaussian $u_a$. For each $d$ six draws (seeds $0$--$5$) are shared by all $n$ and $\sigma$, so that the polynomial with spread $\sigma$ is $f^{(\sigma)}$ for the one with spread $1$. The \emph{benchmark} consists of the $228$ polynomials with $\sigma=1$, $d=6,10,15,20$ and $n=6,8,11,16,23,32,45,64,90,128,181$ below the dense size $(d+1)(d+2)/2$, and the dense ones; the family with $|\cA|=8$ of the first three draws for $d=10,15,20,30,40,80$; the spread family of three draws each with $d=10$, $|\cA|=11,23,45$ and $d=15$, $|\cA|=23$, with $\sigma=1,4,16,64$; and the \emph{held-out} family of ten further draws (seeds $7001$--$7010$), fixed before the final runs and not used in the development, with $d=6$, $|\cA|=8,16,28$, with $d=10,15$, $|\cA|=8,16,32$ and dense, and with $d=30$, $|\cA|=8$: $120$ polynomials. \emph{Harnack curves}: $f_{d,t}=\sum_{i+j\le d}\varsigma_{ij}t^{2(i^2+ij+j^2)}x_1^ix_2^j$ with $\varsigma_{ij}=1$ if $i$ and $j$ are even and $-1$ otherwise; its dequantization path is the patchworking family $f_{d,t^s}$, whose patchwork is Harnack's M-curve \cite{Vir84,IMS07,Mik00}, with $(d-1)(d-2)/2$ ovals and $d$ unbounded components; $d=6,8,10,12$, $t=0.7,0.5,0.3,0.2$ (coefficients down to $10^{-201}$). \emph{Near-singular cubics}: $1+x_1^3+x_2^3-\lambda x_1x_2$, $\lambda=3(1\pm\varepsilon)$, $\varepsilon=10^{-3},10^{-6},10^{-9},10^{-12},10^{-14}$; for $\lambda=3(1+\varepsilon)$, $C$ has an oval of diameter about $\sqrt\varepsilon$, already in the patchwork, and for $\lambda=3(1-\varepsilon)$ the oval of the patchwork dies at $s-1\approx\varepsilon/\log3$. \emph{Equal coefficients}: $120$ polynomials with $d=4,6,8$, $|\cA|\le16$ and coefficients $\pm1$, resp.\ in $\{\pm1,\pm2,\pm3\}$, for which $-\log|c|$ is often not generic; the implementation takes, among a few random directions $\eta$, one with the fewest circuits that produce an event. Three of these $120$ polynomials violate (A1).

\emph{References.} The inputs are the binary doubles of the coefficients (polynomials printed with three digits are roundings). The reference for $(n_o,n_u)$ is the answer of the exact method on these doubles, computed with the final code for every input, also where it had been recorded before (none changed). Where it did not finish, only $n_u$ is known exactly, from the points at infinity (\S\ref{app:checks}). For the pairing there are three kinds of evidence: an exact reference for perturbed line arrangements (\S\ref{app:lines}); a numerical tracing of $C$ by marching squares on three test families of \S\ref{app:special}, which is not certified (on the line arrangements it disagrees with the exact pairing on $43$ inputs), so that agreement with it is not an independent reference; and, on the families above, only the tests of \S\ref{app:checks}. We call an answer correct or wrong only against the exact reference and report flags separately; without a reference we report agreement, and agreement of our two solvers, which share Sections~\ref{sec:tropical}--\ref{sec:global} and the integration, is not a confirmation.

\subsection{Run protocol and denominators}\label{app:runs}
All runs were made with one frozen version of the code, in one session, on two cores of an Intel Xeon processor at $2.1$\,GHz; the tests of two families of \S\ref{app:special}, the comparison of the event logs and the reference of the line arrangements, were corrected afterwards, and these families were run again with the same code. On the families of \S\ref{app:inputs} every run of a method on an input was a separate process, two at a time, with $20$ minutes of wall-clock time and $6$\,GB of address space; the two shared $6.3$\,GB of memory, and a run killed for lack of memory counts as out of memory (all seven such runs of the exact method were killed by the shared limit, with return code $-9$, and were not repeated alone). The software was Python 3.13.16, NumPy 2.5.3, SciPy 1.18.1 and gcc 13.3.0. Within a chain of inputs (the growing supports of a draw, the cubics, the polynomials with equal coefficients of one degree), a method that exceeded the time limit was not run on the later inputs (\emph{skipped}). The test families of \S\ref{app:special} were run one family at a time, each input and both solvers in one process, without limits, and the sensitivity runs one process per variant. Times are CPU seconds of the whole computation, including the genericity test, the perturbation of the lifting and failed attempts, and excluding the import of the modules; in medians, runs without an answer count as infinitely long.

\emph{Denominators.} Without the held-out family the families contain $419$ distinct polynomials: $416$ are counted in the six families of Table~\ref{tab:outcomes}(a) and in Table~\ref{tab:noref}, $3$ violate (A1), and $21$ benchmark polynomials also belong to the family with $|\cA|=8$ ($d\le20$) or to the spread family ($\sigma=1$). Each solver had $440$ scheduled runs and $443$ attempts. The complex solver returned $397$ complete, $40$ flagged and $1$ counts-only answers, the real solver $426$, $11$ and $1$, in each case $2$ of them after a failed attempt; each solver stopped with an error, with partial output, on $2$ inputs, and $5$ of its answers skipped a test. Every method ran twice on the $21$ shared polynomials: the times of these $105$ pairs of runs differ by a median of $5\%$ and at most a factor of $2.2$ (runs of at least $0.05$\,s), and their outcomes, answers, sets of flags and pairings agree.

\subsection{Performance}\label{app:perf}

\emph{Sparse inputs of high degree} (Figures~\ref{fig:times}(a) and~\ref{fig:timesSM}(a)). The event search of the real solver runs over the $O(|\cA|^3)$ triples of terms (\S\ref{ssec:eventcomp}), whereas the projection methods handle a resultant of degree $d(d-1)$, resp.\ $d(d-1)$ paths; other costs of the real solver, such as the depth of the subdivision, the conditioning and the root finding for the edge polynomials, can grow with $d$. For $|\cA|=8$, with at most five detected event points, its median time stayed below $0.12$\,s up to $d=80$ (three inputs per degree: an observation, not a complexity bound), while the exact method needs $144$\,s for $d=30$ and exceeded the time limit for $d=40$, and the numerical method and the complex solver need $71$\,s and $31$\,s for $d=80$. The advantage disappears as $|\cA|$ grows: the real solver beats the exact method only up to $|\cA|=8,16,32,64$ for $d=6,10,15,20$, and it is the slowest method on dense polynomials ($968$\,s for $d=20$, against $25$\,s and $0.28$\,s for the exact and the numerical method); the numerical method had the smallest median time in $33$ of the $38$ cells $(d,|\cA|)$, the real solver in the other five, all with $|\cA|\le8$.

\emph{Spread} (Figure~\ref{fig:times}(b)). In $\log|x|$ the tropical curve of $f^{(\sigma)}$ is $\sigma$ times that of $f$, whereas the exact method handles coefficients of length $O(\sigma)$. On the twelve polynomials of the spread family the solvers detect $30$, $8$, $1$ and $0$ event points for $\sigma=1,4,16,64$; the patchwork is correct on $6$, $8$ and $11$ of them for $\sigma\le16$ and on the $9$ with an exact reference for $\sigma=64$. On the inputs with a reference the real solver is always correct without a flag, the complex solver is correct and flags $9$ and $7$ answers for $\sigma=16$ and $64$, and the numerical method is wrong on $3$ and fails on $2$ for $\sigma=16$ and is wrong on $6$ for $\sigma=64$; it also violates the exact $n_u$ on the other three (Table~\ref{tab:noref}). The median time of the real solver drops from $1.0$\,s for $\sigma=1$ to $0.059$\,s for $\sigma=16$, as the window shrinks, and stays at $0.051$\,s for $\sigma=64$; for three polynomials with $\sigma=64$ Corollary~\ref{cor:window} alone sufficed. That of the exact method grows from $0.48$\,s to $59$\,s for $d=10$, and for $d=15$, $\sigma=64$ it ran out of memory; the start of the complex solver found $28\%$ and $74\%$ fewer critical points than the mixed volume for $\sigma=16$ and $64$ (for $\sigma=64$ on the nine inputs not settled by Corollary~\ref{cor:window}).

\emph{Harnack curves, cubics and equal coefficients} (Figure~\ref{fig:timesSM}(b)). For $t\le0.3$ the real solver detects no events, so, if its search is complete, $C$ has the topology of the patchwork, which the exact method confirms where it finished; for $t\ge0.5$ all ovals die along the path. The real solver took at most $29$\,s, and $10$\,s for $t=0.2$; the exact method ran out of memory for $d=10$, $t=0.2$ and for $d=12$, $t\le0.5$, and was faster than the real solver only for $d=6$, $t\ge0.5$ and $d=8$, $t=0.7$. For the cubics with $\lambda=3(1+\varepsilon)$, $g(\ell)=\log\lambda>\log3=r$, and Corollary~\ref{cor:window}, applied for $\varepsilon\ge10^{-6}$ (outside the enlargement of the bounds), shows that there are no events; both solvers take at most $13$\,ms, whereas the numerical method slows from $4.5$\,ms to $0.56$\,s as $\varepsilon$ decreases and is wrong for $\varepsilon=10^{-14}$. Our method has no advantage for $\lambda=3(1-\varepsilon)$, nor for equal coefficients (median times $0.015$, $0.0073$, $0.045$ and $0.066$\,s for the exact method, the numerical method, the complex and the real solver).

\subsection{Counts, pairing and flags}\label{app:reliability}

\begin{table}[htbp]
\centering\small
\caption{The inputs on which the exact method did not answer (t: time limit, m: out of memory, s: skipped): the exact number $n_u$ of lines, from the points at infinity, and all answers $(n_o,n_u)$; f: flagged, e: no answer.}\label{tab:noref}
\setlength{\tabcolsep}{4pt}
\begin{tabular}{@{}llcrcccc@{}}
\toprule
family & input & exact & $n_u$ & numerical & complex & real & patchwork \\
\midrule
$|\cA|=8$ & $d=40$, draw 0 & t & 0 & (1,0) & (1,0) & (1,0) & (1,0) \\
 & $d=40$, draw 1 & t & 2 & (0,2) & (0,2) & (0,2) & (1,2) \\
 & $d=40$, draw 2 & t & 2 & (0,2) & (0,2) & (0,2) & (0,2) \\
 & $d=80$, draw 0 & s & 0 & (1,0) & (1,0) & (1,0) & (4,0) \\
 & $d=80$, draw 1 & s & 0 & (0,0) & (0,0)\,f & (0,0) & (4,0) \\
 & $d=80$, draw 2 & s & 2 & e & (1,2) & (1,2) & (0,2) \\
spread & $d=15$, $\sigma=64$, draw 0 & m & 1 & (0,3) & (2,1)\,f & (2,1) & (2,1) \\
 & $d=15$, $\sigma=64$, draw 1 & m & 1 & (0,3) & (2,1)\,f & (2,1) & (2,1) \\
 & $d=15$, $\sigma=64$, draw 2 & m & 1 & (0,3) & (0,1) & (0,1) & (0,1) \\
Harnack & $d=10$, $t=0.2$ & m & 10 & (0,2) & (36,10)\,f & (36,10) & (36,10) \\
 & $d=12$, $t=0.2$ & m & 12 & (0,2) & (55,12)\,f & (55,12) & (55,12) \\
 & $d=12$, $t=0.3$ & m & 12 & (0,2) & (55,12)\,f & (55,12) & (55,12) \\
 & $d=12$, $t=0.5$ & m & 12 & (0,4) & (54,12)\,f & (0,12) & (55,12) \\
\bottomrule
\end{tabular}
\end{table}

\begin{table}[htbp]
\centering\small
\caption{Outcomes on the inputs with an exact reference for $(n_o,n_u)$, by family (totals in Table~\ref{tab:outcomes}(a)). Projection methods and patchwork: numbers of correct answers (w: wrong, e: no answer; patchwork not computed for equal coefficients). Our solvers: c correct without a flag, cf correct with a flag, wf wrong with a flag, w wrong without a flag, -- no answer.}\label{tab:outcomesfull}
\setlength{\tabcolsep}{2.6pt}
\begin{tabular}{@{}lrrrrrrrrrrrrrr@{}}
\toprule
 & & \multicolumn{2}{c}{projection} & \multicolumn{5}{c}{complex solver} & \multicolumn{5}{c}{real solver} & \\
\cmidrule(lr){3-4}\cmidrule(lr){5-9}\cmidrule(lr){10-14}
family & inputs & exact & numerical & c & cf & wf & w & -- & c & cf & wf & w & -- & patchwork \\
\midrule
benchmark, $\sigma=1$ & 228 & 228 & 228 & 224 & 1 & 1 & 0 & 2 & 217 & 7 & 2 & 0 & 2 & 95 \\
$|\cA|=8$, $d=30$ & 3 & 3 & 3 & 2 & 1 & 0 & 0 & 0 & 3 & 0 & 0 & 0 & 0 & 1 \\
spread, $\sigma=4,16,64$ & 33 & 33 & 22 (9w, 2e) & 17 & 16 & 0 & 0 & 0 & 33 & 0 & 0 & 0 & 0 & 28 \\
Harnack curves & 12 & 12 & 1 (11w) & 1 & 5 & 6 & 0 & 0 & 12 & 0 & 0 & 0 & 0 & 5 \\
near-singular cubics & 10 & 10 & 9 (1w) & 10 & 0 & 0 & 0 & 0 & 10 & 0 & 0 & 0 & 0 & 5 \\
equal coefficients & 117 & 117 & 117 & 116 & 1 & 0 & 0 & 0 & 117 & 0 & 0 & 0 & 0 & -- \\
\midrule
all six families & 403 & 403 & 380 & 370 & 24 & 7 & 0 & 2 & 392 & 7 & 2 & 0 & 2 & 134 \\
\bottomrule
\end{tabular}
\end{table}

\begin{figure}[htbp]
\centering
\definecolor{mExact}{HTML}{0072B2}\definecolor{mNum}{HTML}{E69F00}\definecolor{mCplx}{HTML}{009E73}\definecolor{mReal}{HTML}{D55E00}
\pgfplotsset{every axis/.append style={width=0.5\textwidth,height=0.3\textwidth,font=\scriptsize,
  log basis y=10,grid=major,grid style={gray!20,line width=1pt},axis line style={line width=1pt},tick style={line width=1pt,black},title style={yshift=-3pt},
  legend style={font=\scriptsize,draw=none,fill=white,fill opacity=0.8,text opacity=1}},
  pExact/.style={mExact,mark=square,line width=1.2pt,mark options={solid}},
  pNum/.style={mNum,mark=o,line width=1.2pt,densely dashed,mark options={solid}},
  pCplx/.style={mCplx,mark=triangle,line width=1.2pt,densely dotted,mark options={solid}},
  pReal/.style={mReal,mark=diamond*,line width=1.2pt},
  pWrong/.style={only marks,mark=x,mark size=3.2pt,ultra thick,black},
  pFlag/.style={only marks,mark=o,mark size=4pt,line width=1pt,black}}
\begin{tikzpicture}[baseline=(current axis.north)]
\begin{loglogaxis}[xlabel={$|\cA|$},ylabel={seconds},title={(a) benchmark, $d=20$},xmin=5,xmax=260,ymin=0.01,ymax=2000,
  legend to name=timeslegSM,legend columns=6,legend style={draw=none,fill=none,font=\small,
  /tikz/every even column/.append style={column sep=3pt}}]
\addplot[pExact] coordinates {(6,12.7) (8,10.6) (11,13) (16,15) (23,18.6) (32,24.7) (45,25.5) (64,30.8) (90,29.9) (128,33.9) (181,27.4) (231,25.3)};
\addplot[pNum] coordinates {(6,0.0778) (8,0.0727) (11,0.0818) (16,0.0717) (23,0.0763) (32,0.0986) (45,0.128) (64,0.137) (90,0.143) (128,0.216) (181,0.221) (231,0.278)};
\addplot[pCplx] coordinates {(6,0.178) (8,0.173) (11,0.254) (16,0.286) (23,0.346) (32,0.521) (45,0.687) (64,0.85) (90,1.23) (128,1.85) (181,3.15) (231,3.32)};
\addplot[pReal] coordinates {(6,0.0147) (8,0.0298) (11,0.252) (16,0.808) (23,2.56) (32,4.27) (45,12.5) (64,27.6) (90,62.7) (128,157) (181,455) (231,968)};
\addplot[pFlag] coordinates {(16,0.808) (23,2.56) (32,4.27) (45,12.5) (64,27.6) (90,62.7) (231,968)};
\addplot[pWrong] coordinates {(231,968)};
\legend{exact method,numerical method,complex solver,real solver,\hspace{2pt}flagged,\hspace{2pt}wrong}
\end{loglogaxis}
\end{tikzpicture}\hfill
\begin{tikzpicture}[baseline=(current axis.north)]
\begin{semilogyaxis}[xlabel={$d$},title={(b) Harnack curves, $t=0.2$},xtick={6,8,10,12},xmin=5.5,xmax=12.5,ymin=0.01,ymax=2000]
\addplot[pExact] coordinates {(6,2.49) (8,39.9)};
\addplot[pNum] coordinates {(6,3.66) (8,9.06) (10,15.9) (12,27)};
\addplot[pCplx] coordinates {(6,0.963) (8,2.49) (10,4.84) (12,8.76)};
\addplot[pReal] coordinates {(6,0.409) (8,1.45) (10,4.24) (12,10.3)};
\addplot[pFlag] coordinates {(6,0.963) (8,2.49) (10,4.84) (12,8.76)};
\addplot[pWrong] coordinates {(6,3.66) (8,9.06) (10,15.9) (12,27)};
\node[font=\scriptsize,anchor=north east] at (axis cs:12.4,1500) {exact: out of memory for $d\ge10$};
\end{semilogyaxis}
\end{tikzpicture}\\[4pt]
\pgfplotslegendfromname{timeslegSM}
\caption{CPU times in seconds (log scales), with the marks of Figure~\ref{fig:times}. (a) Medians over the six benchmark polynomials of degree $20$ for each $|\cA|$. (b) Single runs on Harnack's curves $f_{d,0.2}$ (\S\ref{app:inputs}).}\label{fig:timesSM}
\end{figure}
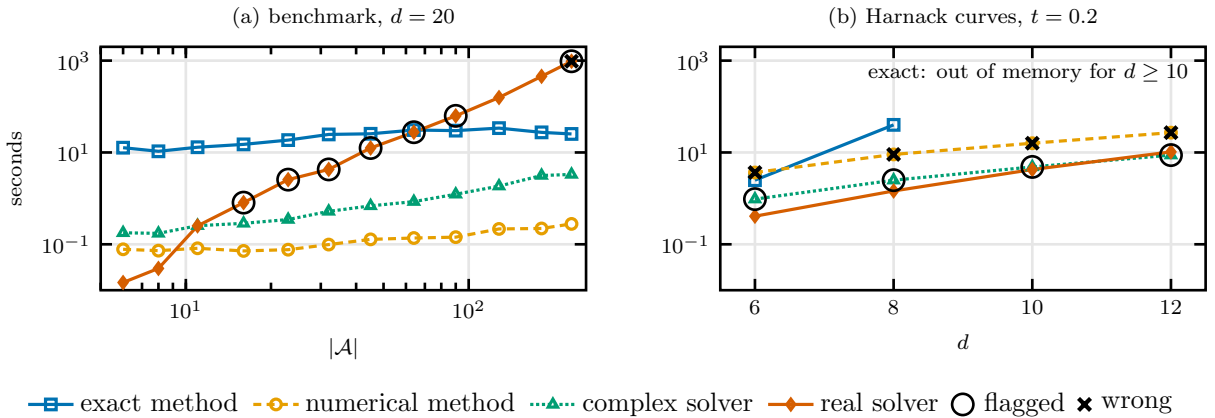

\emph{Counts} (Tables~\ref{tab:outcomes}(a), \ref{tab:outcomesfull} and~\ref{tab:noref}). On the $403$ inputs of the six families with an exact reference every wrong answer of our solvers was flagged ($7$ of the complex and $2$ of the real solver), each solver stopped with an error in the integration of a separatrix on two benchmark polynomials with $d=20$, $|\cA|=181$, and the numerical method returned $21$ wrong answers without a warning; $31$ of the $40$ flagged answers were correct. Both solvers are wrong on the benchmark polynomial with $d=15$, $|\cA|=128$, where a collinear circuit dominates near its discriminant and the events are not resolved in double precision, and the real solver on a dense one with $d=20$, where it missed a saddle in $15$ unresolved boxes; the complex solver is wrong on six Harnack curves with $t\ge0.5$, where its start missed most critical points of $f$ in the torus. On the $13$ inputs of these families without an exact reference the numerical method violates the exact $n_u$ on $7$; the real solver agrees with it on $5$, only with the complex solver on $7$ ($5$ of them flagged by the latter), and with no other method on Harnack's curve with $d=12$, $t=0.5$, so on $8$ inputs the counts are not independently confirmed.

\emph{Pairing} (Table~\ref{tab:outcomes}(b)). An exact reference for the pairing exists only for the $72$ line arrangements (\S\ref{app:lines}). The tracing was attempted on all inputs of Table~\ref{tab:regression}, of the axis stress test and of the rescaled connections; it was inconclusive on inputs 8 and 15 of Table~\ref{tab:regression}, whose pairing was checked by hand, resp.\ inherited from input 12 by an exact change of coordinates, and conclusive on all other inputs, where it agreed with both solvers on all $184$ answers with lines; this is agreement, not an independent reference. None of the $414$ answers of each solver on the six families of Table~\ref{tab:outcomes}(a) and on the inputs of Table~\ref{tab:noref} violated an executed test of \S\ref{app:checks} ($2$ skipped a test), and on all $303$ inputs where both solvers returned the same counts with lines they paired the ends in the same way; this checks necessary conditions and agreement, not the pairing itself.

\emph{Flags.} All $47$ flagged answers on the $416$ inputs of the six families and of Table~\ref{tab:noref} ($38$ of the complex and $9$ of the real solver) carry only numerical flags: missed start points ($37$), an isolated collinear circuit near its discriminant ($1$) and unresolved boxes ($9$); the three polynomials that violate (A1) are flagged by both solvers with the input class. The flags are heuristic: $31$ of the $40$ flagged answers on the six families of Table~\ref{tab:outcomes}(a) are correct, earlier versions of the code gave wrong answers without a flag (\S\ref{app:special}), and some numerical decisions on the axes raise no flag.

\emph{Spread and the number of events.} By Remark~\ref{prop:nesting} the events of $f^{(\sigma)}$ are those of $f$ with $s\ge\sigma$, so the event lists of the benchmark paths give the numbers $K(\sigma)$ and $N(\sigma)$ of Theorem~\ref{thm:prob} for all $\sigma\ge1$, as far as the lists are complete; we call them the \emph{detected} events. On the $226$ paths on which the complex solver answered ($2$ with a flag; its $2$ errors are left out), the fraction of paths with an event drops from $65\%$ for $\sigma=1$ to $12\%$ for $\sigma=16$ and $3.5\%$ for $\sigma=64$, the mean number of detected event points from $1.92$ to $0.15$ and $0.062$, and $\sigma$ times this mean stays between about $2$ and $4$ for $\sigma=1,2,4,\dots,64$. The real solver has the same event points on $225$ of these paths; on the dense polynomial with $d=20$ where its answer is wrong it missed a saddle with $1<s<2$. This is consistent with the bound $O(1/\sigma)$ of Theorem~\ref{thm:prob}, whose constant $C_1$ ranges from about $50$ for $|\cA|=6$ to $4\cdot10^6$ for $|\cA|=64$; it does not measure a rate. (Event lists computed on 15 October with an earlier complex solver have the same event points on all $226$ paths.)

\emph{Held-out inputs} (last row of Table~\ref{tab:outcomes}(a)). On the $120$ held-out polynomials both solvers answered every input at the first attempt, and no test of \S\ref{app:checks} failed or was skipped. On the $118$ on which the exact method finished, the numerical method and the complex solver were always correct, the latter with a flag on $3$, and the real solver on all but one: on the polynomial with $d=15$, $|\cA|=16$ of seed $7009$ two event points are numerically equal in parameter and position, and it returned $(2,1)$ instead of $(1,1)$, with a flag of class unclassified, while the complex solver flagged the same cluster and answered correctly. On the other two ($d=30$, exact method over the time limit) the four other methods agree, with the exact $n_u$. The median times grew from $0.069$\,s for $d=6$ to $221$\,s for $d=30$ for the exact method, and were at most $0.17$\,s for the numerical method, $0.39$\,s for the complex and $1.2$\,s for the real solver.

\emph{Sensitivity to the parameters.} For the $218$ pairs of a solver and one of the $90$ benchmark polynomials with $d=6,10$ or the $19$ inputs of Table~\ref{tab:regression}, we repeated the run with each of eight changes of one parameter ($1744$ runs, besides the $218$ with the frozen values): the tolerances of the integration by the factors $10^{\pm1}$, the start offset $\delta$ at saddles by $10^{\pm2}$, the start offset $\delta_c$ at creations by about $4^{\pm1}$, and $s_0$ by the factors $2$ and $4$. Under none of these changes did an answer, a pairing or a set of flags change: every run answered correctly, with the same $6$ flagged answers (those of Table~\ref{tab:regression}), and the total CPU time changed by at most $9\%$, although the labels of the landing points changed in their last digits. So the frozen values are not at the edge of a failure on these inputs; this does not show that they are optimal.

\subsection{Limitations}\label{ssec:limits}

The implementation is not certified, and its flags are heuristic. A failed step or test returns an error and no answer (on a perturbed lifting after up to three attempts), a flagged answer is unresolved, and no exact method is called automatically. Separatrices can be stiff, and only the start at creations is adapted to this; a near pass is only flagged; connections forced by a symmetry other than an invariant axis, such as the diagonal for $f(x_1,x_2)=f(x_2,x_1)$ with $\bar r_1=\bar r_2$, are not treated. The real solver is not competitive for dense polynomials with small spread, the start of the complex solver loses critical points for large spreads, which a polyhedral start \cite{HuSt95} at the tropical end of the path might reduce (we have not tried this), and an event near $s=1$ has to be resolved like a near-singularity of $C$. Our projection methods are not tuned, so the growth of their times matters more than their size. The test families have triangular Newton polygons, and the inputs of Table~\ref{tab:regression} triangles and quadrilaterals. The output is $(n_o,n_u)$ and the pairing of the ends; the isotopy type of $(\R^2,C)$ would need the region tree of Remark~\ref{rem:isotopy}, which is not implemented.

\FloatBarrier
\section{Special configurations}\label{app:special}

The inputs of Table~\ref{tab:outcomes}(a) rarely or never produce the situations of Sections~\ref{sec:events}--\ref{sec:global} that need care: saddle connections, several event points at one parameter, event points on an axis, events at almost the same parameter or within the rounding of $s=1$, and pairings of the ends that differ from the obvious one. The test families of Table~\ref{tab:coverage} were made to produce them and served as regression and stress tests during the development. Most entered the suite because an earlier version failed on them, often without a flag: on the inputs of Table~\ref{tab:regression} (the complex solver flagged none of these answers, the real solver all but one), on the endpoint families ($171$ to $412$ wrong answers of a solver on a family before the treatment of the endpoint was added, all but $5$ unflagged) and on the rescaled connections (an intermediate rule for $\iota$ paired $16$ of $48$ runs wrongly, $12$ of them unflagged; these $12$ pairings fail two tests of \S\ref{app:checks}). So the flags of one version say little about errors that it does not anticipate. The results below are those of the frozen code; every family has its own reference, and the runs without an answer are counted.

\begin{table}[htbp]
\centering\footnotesize
\caption{Event-state regression, with $h=-\log|c|$ unless a lifting is given (for 1, height $1$ on the constant term and standard Gaussian heights times $0.01$ on the others; for 5c, height $1$ on $x_1^2x_2^2$ plus such noise); random coefficients for 3, 4, 5, 7b. Event points as designed and found; $(n_o,n_u)$ of the exact method; reference of the pairing (T: tracing, H: by hand, I: from 12); answers of the real and the complex solver: \checkmark/\checkmark{} if the counts agree with the exact reference and the pairing with the listed comparison, \checkmark{} if the pairing has no comparison, f: flagged.}\label{tab:regression}
\setlength{\tabcolsep}{3pt}
\begin{tabular}{@{}l>{\raggedright\arraybackslash}p{5.6cm}>{\raggedright\arraybackslash}p{4.5cm}cccc@{}}
\toprule
& input & event points & $(n_o,n_u)$ & ref. & real & complex\\
\midrule
1 & Example~\ref{ex:connection}, lifting & 2 saddles connected along the axis, 1 creation & $(0,3)$ & T & \checkmark/\checkmark & \checkmark/\checkmark\\
1d & the same, $h=-\log|c|$ & 1 saddle on the axis & $(0,3)$ & T & \checkmark/\checkmark & \checkmark/\checkmark\\
2 & Example~\ref{ex:simultaneous} & 2 deaths on the axis at one parameter & $(0,1)$ & T & \checkmark/\checkmark & \checkmark/\checkmark\\
3 & even exponents, $d=6$ & 4 saddles at one parameter & $(2,2)$ & T & \checkmark/\checkmark & \checkmark/\checkmark\\
4 & even in $x_1$, $d=4$ & 2 saddles on one oval at one parameter, 2 deaths & $(1,0)$ & -- & \checkmark & \checkmark\\
5 & even edge polynomial, $d=4$ & 2 annihilations at one parameter & $(1,0)$ & -- & \checkmark & \checkmark\\
5c & $1+\frac12x_1^3+x_1^4-3x_1^2x_2^2+x_2^4$, lifting & 2 births, then 2 creations, pairwise simultaneous & $(0,4)$ & T & \checkmark/\checkmark & \checkmark/\checkmark\\
5d & ten terms on $[0,4]^2$, lifting & 2 creations and 2 annihilations at one parameter, 4 saddles & $(0,4)$ & T & \checkmark/\checkmark & \checkmark/\checkmark\\
6 & $x_1^2-x_1+2x_2^2+3x_2^3+\frac12$ & 1 death on the axis, 1 saddle & $(0,1)$ & T & \checkmark/\checkmark & \checkmark/\checkmark\\
7a & input 4 $+\,10^{-4}x_1^3$ & 2 saddles $1.3\cdot10^{-10}$ apart, 2 deaths & $(1,0)$ & -- & \checkmark & \checkmark\\
7b & $d=6$ & 2 annihilations, 2 deaths next to the divisor & $(0,0)$ & -- & \checkmark, f & \checkmark, f\\
8 & $-x_1^2+\sqrt2\,x_1-1+(x_1-1)^2(x_1+1)x_2+8x_2^2$ & 1 saddle on the axis, separatrices leave it & $(0,2)$ & H & \checkmark/\checkmark, f & \checkmark/\checkmark\\
9 & $x_1-x_1^2+x_2^2-\frac12$ & 1 saddle on the axis, separatrices land on it & $(0,2)$ & T & \checkmark/\checkmark & \checkmark/\checkmark\\
10 & $1+x_1^4+\frac32x_1^2x_2^2+x_2^4$ & none (non-real double roots) & $(0,0)$ & -- & \checkmark & \checkmark\\
11 & $1+x_1+x_2+x_1x_2+2x_1^2+3x_2^2$ & none & $(0,0)$ & -- & \checkmark & \checkmark\\
12 & $-x_1^2+\sqrt2\,x_1-1+(x_1^2-\sqrt2\,x_1+1)x_2+8x_2^2$ & 1 saddle on the axis, separatrices on it & $(0,2)$ & T & \checkmark/\checkmark, f & \checkmark/\checkmark\\
13 & $\frac13x_1^3-x_1^2+\frac32x_1+x_2^2-2x_2+8$ & 1 death (a collision is not an event) & $(0,1)$ & T & \checkmark/\checkmark & \checkmark/\checkmark\\
14 & $1+x_1^3+x_2^3-\lambda x_1x_2$, $\lambda=3-2^{-51}$ & 1 death at $s-1\approx1.3\cdot10^{-16}$ & $(0,1)$ & T & \checkmark/\checkmark & \checkmark/\checkmark\\
15 & input 12 with $x\mapsto(2^{16}x_1,2^{-16}x_2)$ & as 12 & $(0,2)$ & I & \checkmark/\checkmark, f & \checkmark/\checkmark, f\\
\bottomrule
\end{tabular}
\end{table}

\begin{table}[htbp]
\centering\footnotesize
\caption{Test families of the special configurations: what they test, inputs and how often the configuration occurred, references for the counts and the pairing (T: tracing of $C$, H: by hand, I: inherited by an exact change of coordinates, S: smoothing of a line arrangement), and outcomes of the real and the complex solver: the numbers of answers whose counts agree with the exact reference and whose pairing agrees with the listed comparison, f: of them flagged, wf: wrong and flagged, e: no answer.}\label{tab:coverage}
\setlength{\tabcolsep}{3pt}
\begin{tabular}{@{}>{\raggedright\arraybackslash}p{4.6cm}>{\raggedright\arraybackslash}p{3.6cm}>{\raggedright\arraybackslash}p{2.6cm}>{\raggedright\arraybackslash}p{4.0cm}@{}}
\toprule
family (\S): what it tests & inputs; occurrences & reference & outcome (real; complex) \\
\midrule
Table~\ref{tab:regression} (\ref{app:regression}): one configuration each & $19$ & exact; T, H for 8, I for 15 & 19 (4 f); 19 (2 f) \\
event logs (\ref{app:regression}): order within an event, collisions & $10$; $8$ orders on each with several points at one event & Theorem~\ref{thm:update} & 56; 56 orders: same answer, pairing and counts \\
connections (\ref{app:regression}): side $\iota$ of Lemma~\ref{lem:axisconn} & $4$ mirror images, $24$ rescalings; input 1 under $6$ rescalings & exact; T & 28; 24, 4 wf; input 1: 1, 5 e each \\
rescalings (\ref{app:regression}): exact changes of coordinates and scale & $19$ & exact; same pairing & 19 (5 f); 19 (1 f) \\
axis stress (\ref{app:axisstress}): event points on the axes & $240$; 19 paths, 26 points & exact; T & 237 (7 f), 3 e; 237 (15 f), 3 e \\
new liftings (\ref{app:axisstress}): line~\ref{alg:restart} of Algorithm~\ref{alg:main} & $3$, $3$ directions each & exact & all correct; 2 of 3 inputs without flag \\
endpoint circuits (\ref{app:endpoint}): events within $6$ ulp of $s=1$ & $3900$ & exact sign of $\lambda^3-27c_0c_1c_2$ & 3900; 3900 \\
endpoint, random (\ref{app:endpoint}): events within $2\cdot10^{-14}$ of $s=1$ & $511$: $7$ perturbations of each of $73$ & exact & 511; 504, 7 wf \\
endpoint cubics (\ref{app:endpoint}): $\lambda=3+k$ ulp & $10$, and $\lambda=3$ & exact & 10; 10; $\lambda=3$ flagged by both \\
line arrangements (\ref{app:lines}): pairing & $72$; $2d$ ends each & exact; S & 70 (9 f), 2 wf; 64 (13 f), 8 wf \\
\bottomrule
\end{tabular}
\end{table}

\subsection{Event-state regression}\label{app:regression}

Table~\ref{tab:regression} lists small inputs, each made to produce one configuration, and eight that test statements of \S\ref{app:impl} and Proposition~\ref{prop:circuits}. The pairing is compared with a numerical tracing of $C$, by hand for 8, where the tracing does not resolve the ends next to the axis, and for 15 by the change of coordinates from 12, which preserves the labels of the ends. On all inputs the counts of both solvers agree with the exact reference and their pairings with the listed comparison. They flag 7b, where two deaths next to a divisor and two annihilations on it are closer than double precision resolves; the real solver flags 8, 12 and 15, where its subdivision leaves boxes next to the axis unresolved, and the complex solver 15, where its start misses a critical point. The answers and the pairing do not change under the exact rescalings $x\mapsto(2^kx_1,2^{-k}x_2)$, $k=0,8,12,16$, of input 12, under the exchange of its axes, and under multiplication of the coefficients of inputs 1d, 3, 4, 5, 6, 9 and 13 by $2^{\pm200}$ ($19$ runs of each solver; the resultant of the exact method also passed a test against its closed form).

\emph{Event logs.} For the inputs 1, 2, 3, 4, 5, 5c, 5d, 7a, 13 and Example~\ref{ex:infinity} every event point was logged (type, parameter, chart and position, the labels used by its rule, the states before and after it), since the final counts could hide two errors that cancel. On the inputs with several event points at one parameter the computation was repeated with the points of each event in eight random admissible orders; as Theorem~\ref{thm:update} says, the answer, the pairing and the counts after every event did not depend on the order (all $56$ repetitions of each solver), also on 5d, where two creations and two annihilations are exactly simultaneous. On 13 no event point lies within $10^{-3}$ of the collision at $s=\log3/\log2$, so the collision triggers no rule.

\emph{Connections along an axis.} On 1, 1d, 9, 12 and 15, ten separatrices lie on an invariant part of the axis $x_2=0$, and on 1 one of them runs into the lower saddle and is labelled by Lemma~\ref{lem:axisconn}. Reversing the side of that label, on 1 and its three mirror images in the axes, makes the pairing wrong, and the tests of \S\ref{app:checks} reject each of these pairings. On the rescalings $K=-4,-2,\dots,6$ (Example~\ref{ex:connection}) of input 1 with the opposite perturbation and their mirror images, $24$ inputs, the pairings of both solvers agree with the tracing; for $K=6$ the complex solver misses a saddle on all four and returns a wrong count, flagged because an event candidate was not refined. On input 1 itself, whose creation at $s\approx292$ is too deep for the rescaled coordinates, both solvers stop with an error for every $K\neq0$.

\subsection{Events on the axes}\label{app:axisstress}

Since the supports of the families of \S\ref{app:inputs} make singular points on the axes non-generic, both solvers were also run on $240$ random polynomials of degree $4$, $6$ and $8$ whose supports have further points on an axis edge and no point off the axes at lattice distance one from such an edge, with coefficients $\pm e^{u}$, $u$ standard Gaussian (replayed from the recorded doubles). $19$ of their paths have event points on the axes ($26$ in all), and $2$ have critical zeros exponentially close to an axis or a divisor. Neither solver returned a wrong answer; on all $237$ inputs they answered, $149$ of them with lines, the pairing agrees with the tracing, and $22$ separatrices of each solver lie on an invariant axis. Both stopped with an error, after raising flags, on three inputs: two have boundary events on a detached outer edge with $m=2$, singular points of $\overline{C_t}$ on $D_e$ (Remark~\ref{rem:detached}), where the rule of an annihilation found no line to delete, and one has a cluster like input 7b. Besides, the real solver flagged $7$ correct answers and the complex solver $15$, $6$ of them for lost paths. On $5$ inputs the flagged points lie on a detached axis edge. On $3$ others they are Morse points next to a double root of $f_t(x_1,0)$ at almost the same parameter (within $10^{-24}$ in $s$ on two of them), where the second derivative of $f_t$ in $\log|x_2|$ is below $10^{-23}$ of that in $\log|x_1|$ and the test takes this anisotropy for a degeneracy; the complex solver flags all three, the real solver one, and on another it finds two event points that it cannot separate. With three random liftings $-\log|c|+\varepsilon\eta$ each these answers were correct again, without a flag on two of the inputs; on the third the flags remained, so an implementation of line~\ref{alg:restart} needs a bound on the number of new liftings.

\subsection{Events near the endpoint}\label{app:endpoint}

The first endpoint family consists of $3900$ circuits $c_0+c_1x_1^3+c_2x_2^3-\lambda x_1x_2$, $c_i=e^{2u_i}$ with standard Gaussian $u_i$, $\lambda$ within $6$ units in the last place of $(27c_0c_1c_2)^{1/3}$, and random reflections in the axes ($300$ draws with $13$ values of $\lambda$ each). Their topology follows from the sign of $\lambda^3-27c_0c_1c_2$, computed exactly from the binary values, and the oval dies at $s^*=\log3/\log\rho$, $\rho^3=\lambda^3/(c_0c_1c_2)$, which equals the bound of Theorem~\ref{thm:window}, so these inputs also test the rounding of the window. Both solvers answered all of them correctly, without a flag; without the enlargement of the window bounds, Corollary~\ref{cor:window} would have been applied to $206$ circuits with an event above $s=1$. The second family consists of $511$ inputs: for each of $73$ random polynomials of degree $4$ and $6$ with $3$ to $6$ interior exponents, the constant term is set to the nearest double at which a real critical point becomes singular and moved by up to $3$ units in the last place, so that the event lies within $2\cdot10^{-14}$ of $s=1$ (median $10^{-16}$), on either side, and the topology changes within these seven inputs; the reference is the exact method. The real solver answered all correctly, without a flag; the complex solver all but the $7$ inputs of one polynomial, on which it failed to refine an event candidate away from the endpoint and returned flagged answers with one oval too many. On the cubics $1+x_1^3+x_2^3-\lambda x_1x_2$ with $\lambda=3+k$ ulp, $k=\pm1,\pm2,\pm4,\pm8,\pm16$, both solvers are correct without a flag, and both flag the singular control $\lambda=3$.

\subsection{Line arrangements}\label{app:lines}

The pairing has an exact reference for $f=L_1\cdots L_d+\varepsilon$, $L_i=a_ix_1+b_ix_2+c_i$ with small non-zero integers, distinct slopes and no three lines through a point, so that $\Delta=\conv\{0,de_1,de_2\}$, (H0) holds and $C$ has $d$ lines for small $|\varepsilon|$. Then $C$ is the smoothing of the arrangement: near the crossing $p$ of $L_i$ and $L_j$, $f=0$ is $L_iL_j=-\varepsilon/P_{ij}$ with $P_{ij}=\prod_{k\neq i,j}L_k(p)$, whose two branches lie in the two sectors where $L_iL_j$ has the sign of $-\varepsilon P_{ij}$; walking along the lines and switching at every crossing gives the components and the pairing of the $2d$ ends in rational arithmetic. The pairing does not depend on $|\varepsilon|$. Indeed, $P=L_1\cdots L_d$ vanishes on the lines, so every component of $C=\{P=-\varepsilon\}$ lies in one cell of the arrangement, on which $P$ has the sign of $-\varepsilon$, and $\log|P|$ is strictly concave on each cell, its Hessian $-\sum_i\nabla L_i\nabla L_i^T/L_i^2$ being negative definite since the slopes are distinct. In a bounded cell, $C$ is empty or one oval, the boundary of a compact convex superlevel set of $\log|P|$. In an unbounded cell $U$, along every ray contained in $U$ all factors but at most one grow linearly, so $|P|\to\infty$, and $C\cap U$ is the boundary of an unbounded closed convex set that contains no line, i.e.\ one line, whose ends lie along the two unbounded edges of $U$, since there $|L_i|=|\varepsilon|/\prod_{k\neq i}|L_k|\to0$. Hence for every $\varepsilon\neq0$ of the given sign for which $C$ is smooth, each unbounded cell on which $P$ has the sign of $-\varepsilon$ pairs the ends at its two unbounded edges, as the walk above does. The smoothing count of ovals, the number of bounded cells with this sign, holds if $|\varepsilon|$ is below the smallest critical value of $|L_1\cdots L_d|$ off the lines. Since $\log|L_1\cdots L_d|$ is strictly concave on every cell, these critical points are one maximum per bounded cell, which we find by Newton's method from a grid and from the centroids of all triples of crossings, check that there are $(d-1)(d-2)/2$ of them, and take $\varepsilon=\pm2^{-k}$ below $10^{-3}\min(1,v)$, $v$ the smallest value. We took $12$ arrangements for each $d=2,3,4$ and both signs of $\varepsilon$, $72$ inputs; the reference for the counts is the exact method, which agrees with the counts of the smoothing on all of them. The tracing is not reliable here: the curve comes within about $|\varepsilon/P_{ij}|^{1/2}$ of itself near the crossings, and on $43$ inputs the traced pairing differs from the exact one. Both solvers answered all inputs. The real solver returned the exact counts and pairing on $70$, $9$ of them with a flag, and flagged wrong ones on two arrangements with $d=4$, where it could not separate the landing points on a slice. The complex solver was correct on $64$, $13$ of them with a flag; its $8$ wrong answers are flagged: on $4$ its start missed a critical point, on $2$ a candidate was not refined, and on $2$ the landing points on a slice were not separated, which also made the pairing wrong.

This family was added after the code had been prepared for the final runs, and it exposed two defects, which were corrected before them; it is therefore not a held-out test of these corrections. The exact method lost a root of the resultant when an isolating interval ended at an exact dyadic root of the resultant, which gave wrong references on $2$ inputs with $d=2$; and both solvers ordered the landing points on the slice of a birth by their angles at the slice itself, where two separatrices arriving from the same side along the slow direction of the Hessian were only about $10^{-8}$ apart (\S\ref{ssec:labels}), which on the first six inputs gave one wrong answer without a flag and one pairing that the tests rejected.

\FloatBarrier
\section{Details of Example~\ref{ex:infinity}}\label{app:worked}

Here $f_t=-2t+4t^2x_2+3x_1^2+2x_2^2+\tfrac2t\,x_1x_2^2-4t\,x_1^2x_2+x_1^2x_2^2$ and $\Delta=[0,2]^2$. The outer edges are $e_R=[(2,0),(2,2)]$, with $D_{e_R}=\{x_1=\infty\}$, $y=x_2$ and $\psi^{e_R}_t=y^2-4ty+3$, and $e_T=[(0,2),(2,2)]$, with $D_{e_T}=\{x_2=\infty\}$ and, in the coordinate $x_1=1/y$, $\psi^{e_T}_t=x_1^2+\frac2tx_1+2$. The metric $\sum_id\xi_i^2$, $\xi_i=\operatorname{asinh}x_i$, is admissible (\S\ref{ssec:sepcomp}). For $t_0=\frac15$ the bottom $C_{t_0}$ consists of two lines, graphs over the $x_2$-axis: $A$ from $e_1=(r_1,\cO_{--})$ to $e_2=(r_1,\cO_{-+})$ and $B$ from $e_3=(r_2,\cO_{-+})$ to $e_4=(r_2,\cO_{--})$, where $r_1<r_2<0$ are the roots $-5\mp\sqrt{23}$ of $\psi^{e_T}_{t_0}$; landing points are labelled by their positions along $A$ and $B$.

The descending separatrices of the saddle $q_1\approx(-0.488,0.681)$ land at $c_1\approx(-0.680,0.454)$ and $c_2\approx(-0.174,0.828)$, on $A$ and $B$, so rule (a) of Theorem~\ref{thm:update} joins them into $L_1$ and $L_2$. Those of $q_2\approx(-0.679,-2.235)$ land at $d_1\approx(-1.757,-0.940)$ and $d_2\approx(-0.193,-2.412)$, in two intervals of $L_2$, which splits into $L_2'$ and $O_1$. At $t=1/\sqrt2$, $\psi^{e_T}_t=(x_1+\sqrt2)^2$: the point $p$, $x_1=-\sqrt2$, of $D_{e_T}$ is of type (B) with $\alpha\beta>0$ and $Q_m=\cO_{--}$, so rule (e) deletes $L_2'$, whose ends are in $\cO_{--}$, and joins the ends $e_2,e_3$ of $L_1$ into the oval $O_2$. At $t=\sqrt3/2$, $\psi^{e_R}_t=(y-\sqrt3)^2$: the point $p'$, $x_2=\sqrt3$, of $D_{e_R}$ has $\alpha\beta<0$, $Q_m=\cO_{++}$ and $Q_o=\cO_{-+}$, and the separatrix of its boundary saddle lands at $b\approx(-6.594,1.813)$, in the interval $(c_1\,e_2)$ of $O_2$; rule (d) cuts $O_2$ at $b$ into the line $L_3$, and rule (b) adds the line $\lambda$ in $\cO_{++}$. At $t=1$ the lines $L_3$ and $\lambda$ end at the points $x_2=1$ and $x_2=3$ of $D_{e_R}$, $L_3$ on the side $\cO_{-+}$ and $\lambda$ on the side $\cO_{++}$. Figure~\ref{fig:sep} shows the levels and the bottom intervals; the events and the landing points were computed with both solvers.

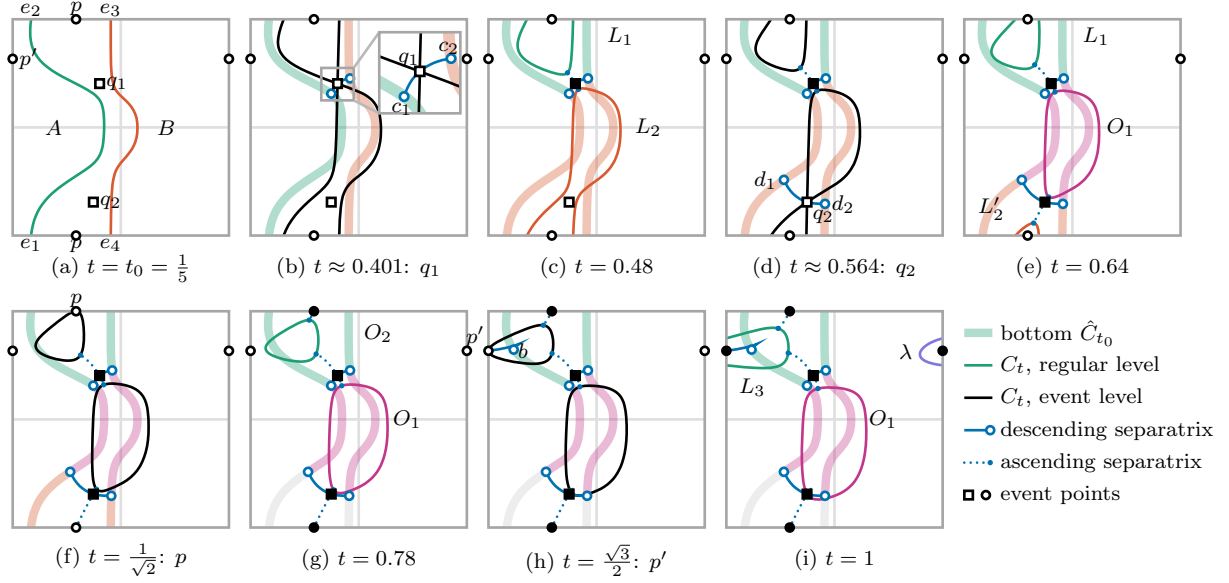
\begin{figure}[tp]
\centering
\definecolor{cL}{HTML}{1D9E75}\definecolor{cO}{HTML}{D85A30}\definecolor{cM}{HTML}{C2378A}\definecolor{cP}{HTML}{7F77DD}\definecolor{sepc}{HTML}{0072B2}
\tikzset{ffr/.style={gray!70,line width=1pt},fax/.style={gray!25,line width=1pt},
  band/.style={line width=3.2pt,opacity=0.35,line cap=butt,line join=round},
  bandin/.style={line width=4.4pt,opacity=0.35,line cap=butt,line join=round},
  cur/.style={line width=1.0pt},curev/.style={black,line width=1.0pt},sep/.style={sepc,line width=1pt},
  asc/.style={sepc,line width=1pt,dash pattern=on 0.01pt off 2.4pt,line cap=round},
  land/.style={sepc,fill=white,line width=1pt},evo/.style={black,fill=white,line width=1pt},
  evf/.style={black,fill=black,line width=1pt},lab/.style={font=\scriptsize,inner sep=1.2pt}}
\begin{tikzpicture}[x=1.43cm,y=1.43cm]
\begin{scope}[shift={(0.00,0.00)}]
\draw[fax] (-1,0) -- (1,0); \draw[fax] (0,-1) -- (0,1);
\draw[cur,cL] (-0.830,-1.000) -- (-0.823,-0.950) -- (-0.813,-0.903) -- (-0.801,-0.859) -- (-0.786,-0.819) -- (-0.768,-0.779) -- (-0.747,-0.742) -- (-0.722,-0.706) -- (-0.694,-0.671) -- (-0.667,-0.641) -- (-0.637,-0.612) -- (-0.603,-0.583) -- (-0.565,-0.553) -- (-0.491,-0.500) -- (-0.336,-0.398) -- (-0.282,-0.359) -- (-0.260,-0.340) -- (-0.240,-0.321) -- (-0.224,-0.302) -- (-0.210,-0.283) -- (-0.196,-0.259) -- (-0.184,-0.232) -- (-0.175,-0.203) -- (-0.167,-0.170) -- (-0.162,-0.132) -- (-0.158,-0.089) -- (-0.155,-0.038) -- (-0.154,0.018) -- (-0.155,0.071) -- (-0.158,0.117) -- (-0.162,0.155) -- (-0.169,0.188) -- (-0.177,0.216) -- (-0.188,0.240) -- (-0.202,0.262) -- (-0.220,0.282) -- (-0.236,0.298) -- (-0.255,0.313) -- (-0.305,0.345) -- (-0.358,0.375) -- (-0.537,0.467) -- (-0.590,0.497) -- (-0.636,0.526) -- (-0.680,0.558) -- (-0.718,0.590) -- (-0.750,0.623) -- (-0.776,0.657) -- (-0.796,0.690) -- (-0.812,0.724) -- (-0.824,0.760) -- (-0.832,0.800) -- (-0.837,0.842) -- (-0.838,0.889) -- (-0.836,0.941) -- (-0.830,1.000);
\draw[cur,cO] (-0.093,-1.000) -- (-0.090,-0.781) -- (-0.085,-0.635) -- (-0.081,-0.579) -- (-0.077,-0.531) -- (-0.071,-0.490) -- (-0.063,-0.455) -- (-0.056,-0.428) -- (-0.047,-0.404) -- (-0.036,-0.382) -- (-0.024,-0.361) -- (-0.002,-0.331) -- (0.055,-0.265) -- (0.079,-0.234) -- (0.102,-0.202) -- (0.120,-0.169) -- (0.135,-0.132) -- (0.146,-0.092) -- (0.152,-0.050) -- (0.154,-0.004) -- (0.152,0.042) -- (0.146,0.084) -- (0.134,0.122) -- (0.118,0.157) -- (0.099,0.188) -- (0.075,0.218) -- (0.048,0.246) -- (-0.009,0.300) -- (-0.034,0.328) -- (-0.047,0.347) -- (-0.057,0.367) -- (-0.065,0.388) -- (-0.072,0.413) -- (-0.079,0.446) -- (-0.084,0.484) -- (-0.088,0.531) -- (-0.091,0.588) -- (-0.093,0.742) -- (-0.093,1.000);
\draw[ffr] (-1,-1) rectangle (1,1);
\draw[evo] (-0.196,0.405) +(-1.6pt,-1.6pt) rectangle +(1.6pt,1.6pt);
\draw[evo] (-0.254,-0.691) +(-1.6pt,-1.6pt) rectangle +(1.6pt,1.6pt);
\draw[evo] (-0.414,1.000) circle (1.5pt);
\draw[evo] (-0.414,-1.000) circle (1.5pt);
\draw[evo] (1.000,0.634) circle (1.5pt);
\draw[evo] (-1.000,0.634) circle (1.5pt);
\node[lab,below] at (-0.830,-1) {$e_1$};
\node[lab,above] at (-0.830,1) {$e_2$};
\node[lab,above] at (-0.093,1) {$e_3$};
\node[lab,below] at (-0.093,-1) {$e_4$};
\node[lab] at (-0.62,0.0) {$A$};
\node[lab] at (0.42,0.0) {$B$};
\node[lab,right] at (-0.166,0.405) {$q_1$};
\node[lab,right] at (-0.224,-0.691) {$q_2$};
\node[lab,below] at (-0.414,-1) {$p$};
\node[lab,above] at (-0.414,1) {$p$};
\node[lab,right] at (-0.97,0.634) {$p\smash{\mathstrut}^{\prime}$};
\node[font=\scriptsize] at (0,-1.30) {(a) $t=t_0=\frac15$};
\end{scope}
\begin{scope}[shift={(2.20,0.00)}]
\draw[fax] (-1,0) -- (1,0); \draw[fax] (0,-1) -- (0,1);
\draw[band,cL] (-0.830,-1.000) -- (-0.823,-0.950) -- (-0.813,-0.903) -- (-0.801,-0.859) -- (-0.786,-0.819) -- (-0.768,-0.779) -- (-0.747,-0.742) -- (-0.722,-0.706) -- (-0.694,-0.671) -- (-0.667,-0.641) -- (-0.637,-0.612) -- (-0.603,-0.583) -- (-0.565,-0.553) -- (-0.491,-0.500) -- (-0.336,-0.398) -- (-0.282,-0.359) -- (-0.260,-0.340) -- (-0.240,-0.321) -- (-0.224,-0.302) -- (-0.210,-0.283) -- (-0.196,-0.259) -- (-0.184,-0.232) -- (-0.175,-0.203) -- (-0.167,-0.170) -- (-0.162,-0.132) -- (-0.158,-0.089) -- (-0.155,-0.038) -- (-0.154,0.018) -- (-0.155,0.071) -- (-0.158,0.117) -- (-0.162,0.155) -- (-0.169,0.188) -- (-0.177,0.216) -- (-0.188,0.240) -- (-0.202,0.262) -- (-0.220,0.282) -- (-0.236,0.298) -- (-0.255,0.313) -- (-0.305,0.345) -- (-0.358,0.375) -- (-0.537,0.467) -- (-0.590,0.497) -- (-0.636,0.526) -- (-0.680,0.558) -- (-0.718,0.590) -- (-0.750,0.623) -- (-0.776,0.657) -- (-0.796,0.690) -- (-0.812,0.724) -- (-0.824,0.760) -- (-0.832,0.800) -- (-0.837,0.842) -- (-0.838,0.889) -- (-0.836,0.941) -- (-0.830,1.000);
\draw[band,cO] (-0.093,-1.000) -- (-0.090,-0.781) -- (-0.085,-0.635) -- (-0.081,-0.579) -- (-0.077,-0.531) -- (-0.071,-0.490) -- (-0.063,-0.455) -- (-0.056,-0.428) -- (-0.047,-0.404) -- (-0.036,-0.382) -- (-0.024,-0.361) -- (-0.002,-0.331) -- (0.055,-0.265) -- (0.079,-0.234) -- (0.102,-0.202) -- (0.120,-0.169) -- (0.135,-0.132) -- (0.146,-0.092) -- (0.152,-0.050) -- (0.154,-0.004) -- (0.152,0.042) -- (0.146,0.084) -- (0.134,0.122) -- (0.118,0.157) -- (0.099,0.188) -- (0.075,0.218) -- (0.048,0.246) -- (-0.009,0.300) -- (-0.034,0.328) -- (-0.047,0.347) -- (-0.057,0.367) -- (-0.065,0.388) -- (-0.072,0.413) -- (-0.079,0.446) -- (-0.084,0.484) -- (-0.088,0.531) -- (-0.091,0.588) -- (-0.093,0.742) -- (-0.093,1.000);
\draw[curev] (-0.694,-1.000) -- (-0.674,-0.951) -- (-0.650,-0.904) -- (-0.621,-0.859) -- (-0.588,-0.814) -- (-0.550,-0.770) -- (-0.505,-0.724) -- (-0.459,-0.681) -- (-0.356,-0.591) -- (-0.312,-0.547) -- (-0.291,-0.522) -- (-0.273,-0.496) -- (-0.259,-0.471) -- (-0.247,-0.442) -- (-0.235,-0.402) -- (-0.226,-0.355) -- (-0.219,-0.297) -- (-0.214,-0.225) -- (-0.207,-0.078) -- (-0.196,0.397) -- (-0.195,0.405) -- (-0.057,0.353) -- (0.027,0.318) -- (0.056,0.303) -- (0.082,0.286) -- (0.104,0.269) -- (0.124,0.251) -- (0.144,0.227) -- (0.160,0.203) -- (0.174,0.175) -- (0.185,0.145) -- (0.194,0.112) -- (0.200,0.075) -- (0.204,0.034) -- (0.206,-0.012) -- (0.205,-0.055) -- (0.203,-0.095) -- (0.198,-0.133) -- (0.191,-0.168) -- (0.183,-0.203) -- (0.172,-0.235) -- (0.159,-0.265) -- (0.144,-0.294) -- (0.129,-0.318) -- (0.112,-0.342) -- (0.093,-0.364) -- (0.073,-0.386) -- (0.033,-0.423) -- (-0.047,-0.487) -- (-0.068,-0.506) -- (-0.086,-0.525) -- (-0.102,-0.545) -- (-0.116,-0.565) -- (-0.127,-0.586) -- (-0.137,-0.609) -- (-0.147,-0.636) -- (-0.154,-0.667) -- (-0.161,-0.704) -- (-0.166,-0.745) -- (-0.175,-0.851) -- (-0.180,-1.000);
\draw[curev] (-0.180,1.000) -- (-0.196,0.405) -- (-0.322,0.451) -- (-0.405,0.484) -- (-0.476,0.514) -- (-0.536,0.542) -- (-0.589,0.571) -- (-0.633,0.599) -- (-0.669,0.627) -- (-0.698,0.657) -- (-0.720,0.688) -- (-0.728,0.703) -- (-0.735,0.720) -- (-0.740,0.736) -- (-0.744,0.754) -- (-0.746,0.773) -- (-0.747,0.792) -- (-0.743,0.832) -- (-0.734,0.878) -- (-0.719,0.931) -- (-0.694,1.000);
\draw[sep] (-0.196,0.405) -- (-0.215,0.386) -- (-0.230,0.365) -- (-0.243,0.341) -- (-0.254,0.312);
\draw[sep] (-0.196,0.405) -- (-0.171,0.423) -- (-0.144,0.437) -- (-0.113,0.447) -- (-0.080,0.453);
\draw[land] (-0.254,0.312) circle (1.5pt);
\draw[land] (-0.080,0.453) circle (1.5pt);
\draw[ffr] (-1,-1) rectangle (1,1);
\draw[evo] (-0.196,0.405) +(-1.6pt,-1.6pt) rectangle +(1.6pt,1.6pt);
\draw[evo] (-0.254,-0.691) +(-1.6pt,-1.6pt) rectangle +(1.6pt,1.6pt);
\draw[evo] (-0.414,1.000) circle (1.5pt);
\draw[evo] (-0.414,-1.000) circle (1.5pt);
\draw[evo] (1.000,0.634) circle (1.5pt);
\draw[evo] (-1.000,0.634) circle (1.5pt);
\draw[gray!70,line width=1pt] (-0.346,0.250) rectangle (-0.046,0.550);
\draw[gray!55,line width=1pt] (-0.046,0.550) -- (0.190,0.880) (-0.046,0.250) -- (0.190,0.130);
\fill[white] (0.190,0.130) rectangle (0.940,0.880);
\begin{scope}
\clip (0.190,0.130) rectangle (0.940,0.880);
\draw[bandin,cL] (0.624,0.006) -- (0.609,0.050) -- (0.591,0.091) -- (0.571,0.128) -- (0.547,0.163) -- (0.519,0.197) -- (0.487,0.230) -- (0.451,0.261) -- (0.409,0.293) -- (0.370,0.320) -- (0.329,0.346) -- (0.229,0.403) -- (0.123,0.460) -- (-0.118,0.582) -- (-0.226,0.638) -- (-0.336,0.698) -- (-0.432,0.755) -- (-0.531,0.817) -- (-0.619,0.880) -- (-0.697,0.942) -- (-0.732,0.973) -- (-0.765,1.005);
\draw[bandin,cO] (1.279,0.006) -- (1.242,0.050) -- (1.198,0.096) -- (1.150,0.142) -- (1.046,0.238) -- (1.002,0.285) -- (0.975,0.317) -- (0.951,0.351) -- (0.931,0.385) -- (0.913,0.421) -- (0.894,0.468) -- (0.878,0.520) -- (0.865,0.577) -- (0.854,0.641) -- (0.845,0.713) -- (0.838,0.797) -- (0.832,0.893) -- (0.828,1.005);
\draw[curev] (0.554,0.055) -- (0.566,0.517) -- (0.822,0.422) -- (1.015,0.346);
\draw[curev] (0.576,0.955) -- (0.564,0.518) -- (0.293,0.618) -- (0.115,0.685);
\draw[sep] (0.565,0.517) -- (0.540,0.495) -- (0.518,0.471) -- (0.497,0.445) -- (0.479,0.418) -- (0.462,0.388) -- (0.447,0.357) -- (0.433,0.323) -- (0.420,0.285);
\draw[sep] (0.565,0.518) -- (0.595,0.541) -- (0.626,0.562) -- (0.660,0.580) -- (0.696,0.597) -- (0.733,0.610) -- (0.772,0.622) -- (0.812,0.630) -- (0.854,0.637);
\draw[land] (0.420,0.285) circle (1.6pt);
\draw[land] (0.854,0.637) circle (1.6pt);
\draw[evo] (0.565,0.518) +(-1.7pt,-1.7pt) rectangle +(1.7pt,1.7pt);
\end{scope}
\draw[gray!70,line width=1pt] (0.190,0.130) rectangle (0.940,0.880);
\node[lab] at (0.400,0.175) {$c_1$};
\node[lab] at (0.834,0.737) {$c_2$};
\node[lab] at (0.465,0.628) {$q_1$};
\node[font=\scriptsize] at (0,-1.30) {(b) $t\approx0.401$: $q_1$};
\end{scope}
\begin{scope}[shift={(4.40,0.00)}]
\draw[fax] (-1,0) -- (1,0); \draw[fax] (0,-1) -- (0,1);
\draw[band,cO] (-0.830,-1.000) -- (-0.823,-0.951) -- (-0.814,-0.904) -- (-0.802,-0.861) -- (-0.787,-0.821) -- (-0.769,-0.782) -- (-0.749,-0.745) -- (-0.725,-0.710) -- (-0.698,-0.675) -- (-0.670,-0.645) -- (-0.640,-0.615) -- (-0.606,-0.585) -- (-0.568,-0.555) -- (-0.492,-0.501) -- (-0.332,-0.395) -- (-0.299,-0.371) -- (-0.272,-0.350) -- (-0.244,-0.325) -- (-0.222,-0.300) -- (-0.204,-0.274) -- (-0.190,-0.246) -- (-0.178,-0.216) -- (-0.170,-0.182) -- (-0.163,-0.143) -- (-0.158,-0.098) -- (-0.155,-0.038) -- (-0.154,0.030) -- (-0.156,0.088) -- (-0.160,0.136) -- (-0.164,0.167) -- (-0.171,0.195) -- (-0.178,0.218) -- (-0.188,0.240) -- (-0.200,0.260) -- (-0.215,0.278) -- (-0.233,0.295) -- (-0.254,0.312);
\draw[band,cL] (-0.254,0.312) -- (-0.285,0.334) -- (-0.325,0.357) -- (-0.523,0.459) -- (-0.571,0.486) -- (-0.614,0.512) -- (-0.657,0.541) -- (-0.695,0.570) -- (-0.728,0.600) -- (-0.755,0.629) -- (-0.781,0.665) -- (-0.802,0.702) -- (-0.818,0.740) -- (-0.829,0.782) -- (-0.836,0.828) -- (-0.838,0.878) -- (-0.836,0.935) -- (-0.830,1.000);
\draw[band,cO] (-0.093,-1.000) -- (-0.090,-0.781) -- (-0.085,-0.635) -- (-0.081,-0.578) -- (-0.076,-0.530) -- (-0.071,-0.490) -- (-0.063,-0.455) -- (-0.056,-0.428) -- (-0.047,-0.404) -- (-0.036,-0.382) -- (-0.024,-0.361) -- (-0.002,-0.331) -- (0.055,-0.265) -- (0.079,-0.235) -- (0.101,-0.203) -- (0.119,-0.171) -- (0.134,-0.134) -- (0.145,-0.095) -- (0.152,-0.053) -- (0.154,-0.009) -- (0.153,0.036) -- (0.147,0.078) -- (0.136,0.115) -- (0.122,0.151) -- (0.103,0.183) -- (0.078,0.214) -- (0.050,0.244) -- (-0.012,0.302) -- (-0.028,0.320) -- (-0.040,0.337) -- (-0.054,0.361) -- (-0.065,0.387) -- (-0.074,0.417) -- (-0.080,0.453);
\draw[band,cL] (-0.080,0.453) -- (-0.088,0.526) -- (-0.092,0.627) -- (-0.093,0.771) -- (-0.093,1.000);
\draw[cur,cO] (-0.644,-1.000) -- (-0.624,-0.964) -- (-0.602,-0.928) -- (-0.577,-0.892) -- (-0.550,-0.858) -- (-0.516,-0.820) -- (-0.479,-0.781) -- (-0.355,-0.666) -- (-0.320,-0.631) -- (-0.302,-0.610) -- (-0.287,-0.589) -- (-0.275,-0.567) -- (-0.265,-0.545) -- (-0.255,-0.513) -- (-0.247,-0.475) -- (-0.241,-0.429) -- (-0.235,-0.371) -- (-0.227,-0.220) -- (-0.216,0.133) -- (-0.211,0.211) -- (-0.205,0.266) -- (-0.200,0.292) -- (-0.195,0.312) -- (-0.188,0.327) -- (-0.181,0.340) -- (-0.171,0.350) -- (-0.160,0.356) -- (-0.146,0.359) -- (-0.129,0.361) -- (-0.096,0.358) -- (-0.052,0.348) -- (0.006,0.331) -- (0.048,0.315) -- (0.082,0.298) -- (0.111,0.278) -- (0.137,0.256) -- (0.158,0.232) -- (0.176,0.205) -- (0.190,0.176) -- (0.202,0.143) -- (0.210,0.106) -- (0.216,0.064) -- (0.220,0.016) -- (0.221,-0.041) -- (0.219,-0.093) -- (0.215,-0.134) -- (0.210,-0.172) -- (0.202,-0.207) -- (0.193,-0.242) -- (0.181,-0.274) -- (0.168,-0.305) -- (0.152,-0.334) -- (0.134,-0.362) -- (0.118,-0.384) -- (0.099,-0.405) -- (0.079,-0.426) -- (0.057,-0.446) -- (0.021,-0.475) -- (-0.071,-0.541) -- (-0.117,-0.580) -- (-0.135,-0.597) -- (-0.150,-0.615) -- (-0.162,-0.633) -- (-0.172,-0.652) -- (-0.182,-0.675) -- (-0.190,-0.702) -- (-0.197,-0.733) -- (-0.202,-0.769) -- (-0.210,-0.862) -- (-0.217,-1.000);
\draw[cur,cL] (-0.217,1.000) -- (-0.223,0.816) -- (-0.229,0.702) -- (-0.235,0.622) -- (-0.243,0.567) -- (-0.247,0.547) -- (-0.253,0.530) -- (-0.260,0.517) -- (-0.267,0.506) -- (-0.276,0.499) -- (-0.287,0.493) -- (-0.300,0.490) -- (-0.315,0.489) -- (-0.335,0.491) -- (-0.360,0.494) -- (-0.421,0.510) -- (-0.492,0.535) -- (-0.559,0.564) -- (-0.610,0.591) -- (-0.651,0.618) -- (-0.668,0.632) -- (-0.683,0.646) -- (-0.696,0.661) -- (-0.707,0.676) -- (-0.721,0.704) -- (-0.729,0.733) -- (-0.731,0.764) -- (-0.727,0.799) -- (-0.717,0.835) -- (-0.702,0.878) -- (-0.680,0.926) -- (-0.644,1.000);
\draw[sep] (-0.196,0.405) -- (-0.215,0.386) -- (-0.230,0.365) -- (-0.243,0.341) -- (-0.254,0.312);
\draw[sep] (-0.196,0.405) -- (-0.171,0.423) -- (-0.144,0.437) -- (-0.113,0.447) -- (-0.080,0.453);
\draw[asc] (-0.196,0.405) -- (-0.230,0.456) -- (-0.269,0.505);
\draw[asc] (-0.196,0.405) -- (-0.180,0.376) -- (-0.168,0.351);
\fill[sepc] (-0.269,0.505) circle (1.0pt);
\fill[sepc] (-0.168,0.351) circle (1.0pt);
\draw[land] (-0.254,0.312) circle (1.5pt);
\draw[land] (-0.080,0.453) circle (1.5pt);
\draw[ffr] (-1,-1) rectangle (1,1);
\draw[evf] (-0.196,0.405) +(-1.6pt,-1.6pt) rectangle +(1.6pt,1.6pt);
\draw[evo] (-0.254,-0.691) +(-1.6pt,-1.6pt) rectangle +(1.6pt,1.6pt);
\draw[evo] (-0.414,1.000) circle (1.5pt);
\draw[evo] (-0.414,-1.000) circle (1.5pt);
\draw[evo] (1.000,0.634) circle (1.5pt);
\draw[evo] (-1.000,0.634) circle (1.5pt);
\node[lab] at (0.22,0.84) {$L_1$};
\node[lab] at (0.48,0.0) {$L_2$};
\node[font=\scriptsize] at (0,-1.30) {(c) $t=0.48$};
\end{scope}
\begin{scope}[shift={(6.60,0.00)}]
\draw[fax] (-1,0) -- (1,0); \draw[fax] (0,-1) -- (0,1);
\draw[band,cO] (-0.830,-1.000) -- (-0.823,-0.951) -- (-0.814,-0.904) -- (-0.802,-0.861) -- (-0.787,-0.821) -- (-0.769,-0.782) -- (-0.749,-0.745) -- (-0.725,-0.710) -- (-0.698,-0.675) -- (-0.670,-0.645) -- (-0.640,-0.615) -- (-0.606,-0.585) -- (-0.568,-0.555) -- (-0.492,-0.501) -- (-0.332,-0.395) -- (-0.299,-0.371) -- (-0.272,-0.350) -- (-0.244,-0.325) -- (-0.222,-0.300) -- (-0.204,-0.274) -- (-0.190,-0.246) -- (-0.178,-0.216) -- (-0.170,-0.182) -- (-0.163,-0.143) -- (-0.158,-0.098) -- (-0.155,-0.038) -- (-0.154,0.030) -- (-0.156,0.088) -- (-0.160,0.136) -- (-0.164,0.167) -- (-0.171,0.195) -- (-0.178,0.218) -- (-0.188,0.240) -- (-0.200,0.260) -- (-0.215,0.278) -- (-0.233,0.295) -- (-0.254,0.312);
\draw[band,cL] (-0.254,0.312) -- (-0.285,0.334) -- (-0.325,0.357) -- (-0.523,0.459) -- (-0.571,0.486) -- (-0.614,0.512) -- (-0.657,0.541) -- (-0.695,0.570) -- (-0.728,0.600) -- (-0.755,0.629) -- (-0.781,0.665) -- (-0.802,0.702) -- (-0.818,0.740) -- (-0.829,0.782) -- (-0.836,0.828) -- (-0.838,0.878) -- (-0.836,0.935) -- (-0.830,1.000);
\draw[band,cO] (-0.093,-1.000) -- (-0.090,-0.781) -- (-0.085,-0.635) -- (-0.081,-0.578) -- (-0.076,-0.530) -- (-0.071,-0.490) -- (-0.063,-0.455) -- (-0.056,-0.428) -- (-0.047,-0.404) -- (-0.036,-0.382) -- (-0.024,-0.361) -- (-0.002,-0.331) -- (0.055,-0.265) -- (0.079,-0.235) -- (0.101,-0.203) -- (0.119,-0.171) -- (0.134,-0.134) -- (0.145,-0.095) -- (0.152,-0.053) -- (0.154,-0.009) -- (0.153,0.036) -- (0.147,0.078) -- (0.136,0.115) -- (0.122,0.151) -- (0.103,0.183) -- (0.078,0.214) -- (0.050,0.244) -- (-0.012,0.302) -- (-0.028,0.320) -- (-0.040,0.337) -- (-0.054,0.361) -- (-0.065,0.387) -- (-0.074,0.417) -- (-0.080,0.453);
\draw[band,cL] (-0.080,0.453) -- (-0.088,0.526) -- (-0.092,0.627) -- (-0.093,0.771) -- (-0.093,1.000);
\draw[curev] (-0.587,-1.000) -- (-0.559,-0.961) -- (-0.527,-0.922) -- (-0.492,-0.884) -- (-0.452,-0.845) -- (-0.410,-0.807) -- (-0.363,-0.770) -- (-0.311,-0.731) -- (-0.254,-0.691) -- (-0.260,-1.000);
\draw[curev] (-0.260,1.000) -- (-0.264,0.851) -- (-0.270,0.742) -- (-0.278,0.663) -- (-0.282,0.633) -- (-0.288,0.607) -- (-0.295,0.586) -- (-0.303,0.569) -- (-0.312,0.555) -- (-0.322,0.544) -- (-0.335,0.536) -- (-0.350,0.530) -- (-0.366,0.526) -- (-0.386,0.525) -- (-0.407,0.526) -- (-0.432,0.530) -- (-0.487,0.542) -- (-0.547,0.561) -- (-0.603,0.584) -- (-0.646,0.607) -- (-0.680,0.632) -- (-0.693,0.644) -- (-0.705,0.657) -- (-0.714,0.671) -- (-0.722,0.685) -- (-0.729,0.710) -- (-0.730,0.736) -- (-0.725,0.765) -- (-0.714,0.798) -- (-0.698,0.830) -- (-0.675,0.869) -- (-0.587,1.000);
\draw[curev] (-0.253,-0.690) -- (-0.245,-0.331) -- (-0.233,0.053) -- (-0.229,0.124) -- (-0.224,0.180) -- (-0.218,0.225) -- (-0.211,0.262) -- (-0.201,0.293) -- (-0.195,0.305) -- (-0.187,0.316) -- (-0.180,0.325) -- (-0.171,0.333) -- (-0.161,0.338) -- (-0.150,0.343) -- (-0.123,0.348) -- (-0.088,0.347) -- (-0.045,0.342) -- (0.006,0.331) -- (0.045,0.319) -- (0.079,0.305) -- (0.110,0.288) -- (0.137,0.270) -- (0.160,0.248) -- (0.179,0.224) -- (0.195,0.197) -- (0.209,0.166) -- (0.220,0.127) -- (0.228,0.083) -- (0.233,0.032) -- (0.235,-0.029) -- (0.235,-0.088) -- (0.231,-0.144) -- (0.224,-0.196) -- (0.213,-0.244) -- (0.204,-0.276) -- (0.193,-0.307) -- (0.181,-0.337) -- (0.166,-0.364) -- (0.150,-0.391) -- (0.132,-0.416) -- (0.112,-0.440) -- (0.089,-0.463) -- (0.056,-0.492) -- (0.018,-0.520) -- (-0.158,-0.629) -- (-0.253,-0.690);
\draw[sep] (-0.196,0.405) -- (-0.215,0.386) -- (-0.230,0.365) -- (-0.243,0.341) -- (-0.254,0.312);
\draw[sep] (-0.196,0.405) -- (-0.171,0.423) -- (-0.144,0.437) -- (-0.113,0.447) -- (-0.080,0.453);
\draw[sep] (-0.254,-0.691) -- (-0.288,-0.678) -- (-0.321,-0.661) -- (-0.352,-0.640) -- (-0.382,-0.615) -- (-0.407,-0.587) -- (-0.430,-0.556) -- (-0.450,-0.522) -- (-0.468,-0.485);
\draw[sep] (-0.254,-0.691) -- (-0.223,-0.698) -- (-0.189,-0.703) -- (-0.147,-0.706) -- (-0.088,-0.707);
\draw[asc] (-0.196,0.405) -- (-0.222,0.446) -- (-0.252,0.485) -- (-0.283,0.521) -- (-0.315,0.551);
\draw[asc] (-0.196,0.405) -- (-0.176,0.368) -- (-0.163,0.337);
\fill[sepc] (-0.315,0.551) circle (1.0pt);
\fill[sepc] (-0.163,0.337) circle (1.0pt);
\draw[land] (-0.254,0.312) circle (1.5pt);
\draw[land] (-0.080,0.453) circle (1.5pt);
\draw[land] (-0.468,-0.485) circle (1.5pt);
\draw[land] (-0.088,-0.707) circle (1.5pt);
\draw[ffr] (-1,-1) rectangle (1,1);
\draw[evf] (-0.196,0.405) +(-1.6pt,-1.6pt) rectangle +(1.6pt,1.6pt);
\draw[evo] (-0.254,-0.691) +(-1.6pt,-1.6pt) rectangle +(1.6pt,1.6pt);
\draw[evo] (-0.414,1.000) circle (1.5pt);
\draw[evo] (-0.414,-1.000) circle (1.5pt);
\draw[evo] (1.000,0.634) circle (1.5pt);
\draw[evo] (-1.000,0.634) circle (1.5pt);
\node[lab] at (-0.104,-0.826) {$q_2$};
\node[lab,left] at (-0.498,-0.485) {$d_1$};
\node[lab,right] at (-0.058,-0.707) {$d_2$};
\node[font=\scriptsize] at (0,-1.30) {(d) $t\approx0.564$: $q_2$};
\end{scope}
\begin{scope}[shift={(8.80,0.00)}]
\draw[fax] (-1,0) -- (1,0); \draw[fax] (0,-1) -- (0,1);
\draw[band,cO] (-0.830,-1.000) -- (-0.825,-0.958) -- (-0.817,-0.918) -- (-0.808,-0.880) -- (-0.796,-0.845) -- (-0.783,-0.811) -- (-0.768,-0.778) -- (-0.750,-0.747) -- (-0.730,-0.716) -- (-0.708,-0.687) -- (-0.683,-0.658) -- (-0.656,-0.631) -- (-0.626,-0.603) -- (-0.593,-0.575) -- (-0.557,-0.546) -- (-0.468,-0.485);
\draw[band,cM] (-0.468,-0.485) -- (-0.322,-0.388) -- (-0.282,-0.359) -- (-0.251,-0.332) -- (-0.224,-0.302) -- (-0.203,-0.272) -- (-0.187,-0.240) -- (-0.175,-0.203) -- (-0.167,-0.170) -- (-0.162,-0.134) -- (-0.158,-0.092) -- (-0.155,-0.044) -- (-0.154,0.015) -- (-0.155,0.070) -- (-0.158,0.116) -- (-0.162,0.155) -- (-0.167,0.182) -- (-0.174,0.206) -- (-0.182,0.226) -- (-0.191,0.246) -- (-0.203,0.264) -- (-0.217,0.280) -- (-0.234,0.296) -- (-0.254,0.312);
\draw[band,cL] (-0.254,0.312) -- (-0.285,0.334) -- (-0.325,0.357) -- (-0.523,0.459) -- (-0.571,0.486) -- (-0.614,0.512) -- (-0.657,0.541) -- (-0.695,0.570) -- (-0.728,0.600) -- (-0.755,0.629) -- (-0.781,0.665) -- (-0.802,0.702) -- (-0.818,0.740) -- (-0.829,0.782) -- (-0.836,0.828) -- (-0.838,0.878) -- (-0.836,0.935) -- (-0.830,1.000);
\draw[band,cO] (-0.093,-1.000) -- (-0.091,-0.832) -- (-0.088,-0.707);
\draw[band,cM] (-0.088,-0.707) -- (-0.083,-0.607) -- (-0.077,-0.531) -- (-0.067,-0.471) -- (-0.053,-0.422) -- (-0.038,-0.386) -- (-0.019,-0.353) -- (-0.000,-0.329) -- (0.054,-0.266) -- (0.077,-0.237) -- (0.099,-0.206) -- (0.117,-0.174) -- (0.134,-0.136) -- (0.145,-0.097) -- (0.152,-0.054) -- (0.154,-0.008) -- (0.153,0.038) -- (0.147,0.079) -- (0.136,0.116) -- (0.121,0.151) -- (0.102,0.183) -- (0.078,0.214) -- (0.050,0.244) -- (-0.012,0.302) -- (-0.028,0.320) -- (-0.040,0.337) -- (-0.054,0.361) -- (-0.065,0.387) -- (-0.074,0.417) -- (-0.080,0.453);
\draw[band,cL] (-0.080,0.453) -- (-0.088,0.526) -- (-0.092,0.627) -- (-0.093,0.771) -- (-0.093,1.000);
\draw[cur,cO] (-0.527,-1.000) -- (-0.504,-0.977) -- (-0.480,-0.954) -- (-0.455,-0.933) -- (-0.432,-0.916) -- (-0.411,-0.903) -- (-0.392,-0.893) -- (-0.375,-0.887) -- (-0.361,-0.885) -- (-0.350,-0.887) -- (-0.341,-0.892) -- (-0.333,-0.900) -- (-0.326,-0.912) -- (-0.321,-0.927) -- (-0.316,-0.947) -- (-0.310,-1.000);
\draw[cur,cL] (-0.310,1.000) -- (-0.307,0.948) -- (-0.305,0.891) -- (-0.307,0.779) -- (-0.310,0.732) -- (-0.314,0.692) -- (-0.319,0.657) -- (-0.326,0.629) -- (-0.334,0.606) -- (-0.343,0.588) -- (-0.353,0.573) -- (-0.366,0.561) -- (-0.381,0.552) -- (-0.398,0.546) -- (-0.419,0.542) -- (-0.442,0.541) -- (-0.465,0.542) -- (-0.490,0.545) -- (-0.545,0.556) -- (-0.602,0.573) -- (-0.653,0.594) -- (-0.691,0.615) -- (-0.706,0.626) -- (-0.719,0.637) -- (-0.730,0.650) -- (-0.738,0.662) -- (-0.744,0.674) -- (-0.748,0.688) -- (-0.749,0.699) -- (-0.749,0.711) -- (-0.743,0.735) -- (-0.731,0.763) -- (-0.712,0.793) -- (-0.689,0.823) -- (-0.658,0.860) -- (-0.527,1.000);
\draw[cur,cM] (-0.202,-0.651) -- (-0.211,-0.649) -- (-0.218,-0.645) -- (-0.226,-0.639) -- (-0.232,-0.630) -- (-0.237,-0.620) -- (-0.241,-0.607) -- (-0.248,-0.573) -- (-0.253,-0.526) -- (-0.255,-0.463) -- (-0.256,-0.380) -- (-0.255,-0.275) -- (-0.251,-0.131) -- (-0.245,0.025) -- (-0.241,0.104) -- (-0.235,0.166) -- (-0.226,0.220) -- (-0.215,0.260) -- (-0.208,0.277) -- (-0.201,0.291) -- (-0.192,0.303) -- (-0.183,0.313) -- (-0.168,0.323) -- (-0.151,0.331) -- (-0.131,0.336) -- (-0.107,0.339) -- (-0.079,0.339) -- (-0.047,0.336) -- (-0.009,0.330) -- (0.026,0.323) -- (0.054,0.315) -- (0.079,0.306) -- (0.103,0.296) -- (0.125,0.284) -- (0.145,0.271) -- (0.162,0.257) -- (0.177,0.241) -- (0.191,0.223) -- (0.202,0.205) -- (0.212,0.186) -- (0.220,0.165) -- (0.227,0.143) -- (0.238,0.092) -- (0.245,0.031) -- (0.247,-0.054) -- (0.245,-0.133) -- (0.237,-0.204) -- (0.232,-0.237) -- (0.224,-0.270) -- (0.214,-0.305) -- (0.203,-0.338) -- (0.188,-0.371) -- (0.173,-0.401) -- (0.155,-0.429) -- (0.135,-0.455) -- (0.113,-0.481) -- (0.088,-0.504) -- (0.057,-0.530) -- (0.023,-0.553) -- (-0.023,-0.579) -- (-0.082,-0.611) -- (-0.124,-0.631) -- (-0.156,-0.643) -- (-0.182,-0.650) -- (-0.202,-0.651);
\draw[sep] (-0.196,0.405) -- (-0.215,0.386) -- (-0.230,0.365) -- (-0.243,0.341) -- (-0.254,0.312);
\draw[sep] (-0.196,0.405) -- (-0.171,0.423) -- (-0.144,0.437) -- (-0.113,0.447) -- (-0.080,0.453);
\draw[sep] (-0.254,-0.691) -- (-0.288,-0.678) -- (-0.321,-0.661) -- (-0.352,-0.640) -- (-0.382,-0.615) -- (-0.407,-0.587) -- (-0.430,-0.556) -- (-0.450,-0.522) -- (-0.468,-0.485);
\draw[sep] (-0.254,-0.691) -- (-0.223,-0.698) -- (-0.189,-0.703) -- (-0.147,-0.706) -- (-0.088,-0.707);
\draw[asc] (-0.196,0.405) -- (-0.230,0.457) -- (-0.269,0.505) -- (-0.309,0.546) -- (-0.329,0.564) -- (-0.349,0.579);
\draw[asc] (-0.196,0.405) -- (-0.173,0.363) -- (-0.166,0.345) -- (-0.161,0.327);
\draw[asc] (-0.254,-0.691) -- (-0.234,-0.661) -- (-0.219,-0.644);
\draw[asc] (-0.254,-0.691) -- (-0.299,-0.771) -- (-0.359,-0.885);
\fill[sepc] (-0.349,0.579) circle (1.0pt);
\fill[sepc] (-0.161,0.327) circle (1.0pt);
\fill[sepc] (-0.219,-0.644) circle (1.0pt);
\fill[sepc] (-0.359,-0.885) circle (1.0pt);
\draw[land] (-0.254,0.312) circle (1.5pt);
\draw[land] (-0.080,0.453) circle (1.5pt);
\draw[land] (-0.468,-0.485) circle (1.5pt);
\draw[land] (-0.088,-0.707) circle (1.5pt);
\draw[ffr] (-1,-1) rectangle (1,1);
\draw[evf] (-0.196,0.405) +(-1.6pt,-1.6pt) rectangle +(1.6pt,1.6pt);
\draw[evf] (-0.254,-0.691) +(-1.6pt,-1.6pt) rectangle +(1.6pt,1.6pt);
\draw[evo] (-0.414,1.000) circle (1.5pt);
\draw[evo] (-0.414,-1.000) circle (1.5pt);
\draw[evo] (1.000,0.634) circle (1.5pt);
\draw[evo] (-1.000,0.634) circle (1.5pt);
\node[lab] at (0.22,0.84) {$L_1$};
\node[lab] at (0.45,0.0) {$O_1$};
\node[lab] at (-0.75,-0.75) {$L_2^{\prime}$};
\node[font=\scriptsize] at (0,-1.30) {(e) $t=0.64$};
\end{scope}
\begin{scope}[shift={(0.00,-2.70)}]
\draw[fax] (-1,0) -- (1,0); \draw[fax] (0,-1) -- (0,1);
\draw[band,cO] (-0.830,-1.000) -- (-0.825,-0.958) -- (-0.817,-0.918) -- (-0.808,-0.880) -- (-0.796,-0.845) -- (-0.783,-0.811) -- (-0.768,-0.778) -- (-0.750,-0.747) -- (-0.730,-0.716) -- (-0.708,-0.687) -- (-0.683,-0.658) -- (-0.656,-0.631) -- (-0.626,-0.603) -- (-0.593,-0.575) -- (-0.557,-0.546) -- (-0.468,-0.485);
\draw[band,cM] (-0.468,-0.485) -- (-0.322,-0.388) -- (-0.282,-0.359) -- (-0.251,-0.332) -- (-0.224,-0.302) -- (-0.203,-0.272) -- (-0.187,-0.240) -- (-0.175,-0.203) -- (-0.167,-0.170) -- (-0.162,-0.134) -- (-0.158,-0.092) -- (-0.155,-0.044) -- (-0.154,0.015) -- (-0.155,0.070) -- (-0.158,0.116) -- (-0.162,0.155) -- (-0.167,0.182) -- (-0.174,0.206) -- (-0.182,0.226) -- (-0.191,0.246) -- (-0.203,0.264) -- (-0.217,0.280) -- (-0.234,0.296) -- (-0.254,0.312);
\draw[band,cL] (-0.254,0.312) -- (-0.285,0.334) -- (-0.325,0.357) -- (-0.523,0.459) -- (-0.571,0.486) -- (-0.614,0.512) -- (-0.657,0.541) -- (-0.695,0.570) -- (-0.728,0.600) -- (-0.755,0.629) -- (-0.781,0.665) -- (-0.802,0.702) -- (-0.818,0.740) -- (-0.829,0.782) -- (-0.836,0.828) -- (-0.838,0.878) -- (-0.836,0.935) -- (-0.830,1.000);
\draw[band,cO] (-0.093,-1.000) -- (-0.091,-0.832) -- (-0.088,-0.707);
\draw[band,cM] (-0.088,-0.707) -- (-0.083,-0.607) -- (-0.077,-0.531) -- (-0.067,-0.471) -- (-0.053,-0.422) -- (-0.038,-0.386) -- (-0.019,-0.353) -- (-0.000,-0.329) -- (0.054,-0.266) -- (0.077,-0.237) -- (0.099,-0.206) -- (0.117,-0.174) -- (0.134,-0.136) -- (0.145,-0.097) -- (0.152,-0.054) -- (0.154,-0.008) -- (0.153,0.038) -- (0.147,0.079) -- (0.136,0.116) -- (0.121,0.151) -- (0.102,0.183) -- (0.078,0.214) -- (0.050,0.244) -- (-0.012,0.302) -- (-0.028,0.320) -- (-0.040,0.337) -- (-0.054,0.361) -- (-0.065,0.387) -- (-0.074,0.417) -- (-0.080,0.453);
\draw[band,cL] (-0.080,0.453) -- (-0.088,0.526) -- (-0.092,0.627) -- (-0.093,0.771) -- (-0.093,1.000);
\draw[curev] (-0.410,1.000) -- (-0.400,0.998) -- (-0.390,0.993) -- (-0.382,0.986) -- (-0.374,0.976) -- (-0.367,0.964) -- (-0.361,0.949) -- (-0.351,0.910) -- (-0.344,0.854) -- (-0.341,0.790) -- (-0.342,0.726) -- (-0.348,0.671) -- (-0.355,0.639) -- (-0.365,0.611) -- (-0.377,0.589) -- (-0.391,0.572) -- (-0.410,0.559) -- (-0.431,0.550) -- (-0.456,0.545) -- (-0.486,0.543) -- (-0.512,0.544) -- (-0.540,0.547) -- (-0.601,0.558) -- (-0.661,0.575) -- (-0.688,0.585) -- (-0.712,0.596) -- (-0.731,0.606) -- (-0.748,0.617) -- (-0.762,0.629) -- (-0.772,0.640) -- (-0.780,0.652) -- (-0.785,0.664) -- (-0.788,0.676) -- (-0.787,0.690) -- (-0.784,0.704) -- (-0.777,0.719) -- (-0.768,0.734) -- (-0.756,0.751) -- (-0.721,0.788) -- (-0.670,0.833) -- (-0.580,0.905) -- (-0.502,0.962) -- (-0.476,0.978) -- (-0.454,0.989) -- (-0.435,0.997) -- (-0.418,1.000);
\draw[curev] (-0.192,-0.663) -- (-0.202,-0.661) -- (-0.211,-0.657) -- (-0.218,-0.652) -- (-0.226,-0.645) -- (-0.238,-0.626) -- (-0.247,-0.600) -- (-0.255,-0.564) -- (-0.260,-0.520) -- (-0.263,-0.464) -- (-0.265,-0.395) -- (-0.264,-0.281) -- (-0.261,-0.142) -- (-0.255,0.018) -- (-0.249,0.108) -- (-0.241,0.172) -- (-0.231,0.221) -- (-0.218,0.258) -- (-0.211,0.274) -- (-0.202,0.286) -- (-0.192,0.298) -- (-0.178,0.309) -- (-0.164,0.317) -- (-0.148,0.323) -- (-0.131,0.327) -- (-0.109,0.330) -- (-0.086,0.331) -- (-0.059,0.330) -- (-0.026,0.327) -- (0.011,0.322) -- (0.041,0.316) -- (0.068,0.308) -- (0.094,0.300) -- (0.118,0.289) -- (0.140,0.277) -- (0.159,0.264) -- (0.184,0.241) -- (0.205,0.214) -- (0.222,0.183) -- (0.235,0.147) -- (0.245,0.106) -- (0.251,0.057) -- (0.256,0.001) -- (0.257,-0.067) -- (0.256,-0.132) -- (0.251,-0.193) -- (0.243,-0.248) -- (0.231,-0.301) -- (0.216,-0.348) -- (0.198,-0.393) -- (0.177,-0.434) -- (0.152,-0.471) -- (0.124,-0.504) -- (0.092,-0.535) -- (0.056,-0.563) -- (0.016,-0.588) -- (-0.057,-0.625) -- (-0.116,-0.649) -- (-0.140,-0.657) -- (-0.160,-0.662) -- (-0.177,-0.664) -- (-0.192,-0.663);
\draw[sep] (-0.196,0.405) -- (-0.215,0.386) -- (-0.230,0.365) -- (-0.243,0.341) -- (-0.254,0.312);
\draw[sep] (-0.196,0.405) -- (-0.171,0.423) -- (-0.144,0.437) -- (-0.113,0.447) -- (-0.080,0.453);
\draw[sep] (-0.254,-0.691) -- (-0.288,-0.678) -- (-0.321,-0.661) -- (-0.352,-0.640) -- (-0.382,-0.615) -- (-0.407,-0.587) -- (-0.430,-0.556) -- (-0.450,-0.522) -- (-0.468,-0.485);
\draw[sep] (-0.254,-0.691) -- (-0.223,-0.698) -- (-0.189,-0.703) -- (-0.147,-0.706) -- (-0.088,-0.707);
\draw[asc] (-0.196,0.405) -- (-0.215,0.435) -- (-0.237,0.465) -- (-0.259,0.493) -- (-0.282,0.520) -- (-0.306,0.543) -- (-0.329,0.563) -- (-0.351,0.580) -- (-0.373,0.594);
\draw[asc] (-0.196,0.405) -- (-0.182,0.380) -- (-0.171,0.358) -- (-0.164,0.339) -- (-0.159,0.320);
\draw[asc] (-0.254,-0.691) -- (-0.234,-0.661) -- (-0.219,-0.644) -- (-0.224,-0.647);
\draw[asc] (-0.254,-0.691) -- (-0.286,-0.747) -- (-0.328,-0.825) -- (-0.373,-0.915) -- (-0.414,-1.000) -- (-0.414,-1.000);
\fill[sepc] (-0.373,0.594) circle (1.0pt);
\fill[sepc] (-0.159,0.320) circle (1.0pt);
\fill[sepc] (-0.224,-0.647) circle (1.0pt);
\draw[land] (-0.254,0.312) circle (1.5pt);
\draw[land] (-0.080,0.453) circle (1.5pt);
\draw[land] (-0.468,-0.485) circle (1.5pt);
\draw[land] (-0.088,-0.707) circle (1.5pt);
\draw[ffr] (-1,-1) rectangle (1,1);
\draw[evf] (-0.196,0.405) +(-1.6pt,-1.6pt) rectangle +(1.6pt,1.6pt);
\draw[evf] (-0.254,-0.691) +(-1.6pt,-1.6pt) rectangle +(1.6pt,1.6pt);
\draw[evo] (-0.414,1.000) circle (1.5pt);
\draw[evo] (-0.414,-1.000) circle (1.5pt);
\draw[evo] (1.000,0.634) circle (1.5pt);
\draw[evo] (-1.000,0.634) circle (1.5pt);
\node[lab,above] at (-0.414,1) {$p$};
\node[font=\scriptsize] at (0,-1.30) {(f) $t=\frac1{\sqrt2}$: $p$};
\end{scope}
\begin{scope}[shift={(2.20,-2.70)}]
\draw[fax] (-1,0) -- (1,0); \draw[fax] (0,-1) -- (0,1);
\draw[band,gray!40] (-0.830,-1.000) -- (-0.825,-0.958) -- (-0.817,-0.918) -- (-0.808,-0.880) -- (-0.796,-0.845) -- (-0.783,-0.811) -- (-0.768,-0.778) -- (-0.750,-0.747) -- (-0.730,-0.716) -- (-0.708,-0.687) -- (-0.683,-0.658) -- (-0.656,-0.631) -- (-0.626,-0.603) -- (-0.593,-0.575) -- (-0.557,-0.546) -- (-0.468,-0.485);
\draw[band,cM] (-0.468,-0.485) -- (-0.322,-0.388) -- (-0.282,-0.359) -- (-0.251,-0.332) -- (-0.224,-0.302) -- (-0.203,-0.272) -- (-0.187,-0.240) -- (-0.175,-0.203) -- (-0.167,-0.170) -- (-0.162,-0.134) -- (-0.158,-0.092) -- (-0.155,-0.044) -- (-0.154,0.015) -- (-0.155,0.070) -- (-0.158,0.116) -- (-0.162,0.155) -- (-0.167,0.182) -- (-0.174,0.206) -- (-0.182,0.226) -- (-0.191,0.246) -- (-0.203,0.264) -- (-0.217,0.280) -- (-0.234,0.296) -- (-0.254,0.312);
\draw[band,cL] (-0.254,0.312) -- (-0.285,0.334) -- (-0.325,0.357) -- (-0.523,0.459) -- (-0.571,0.486) -- (-0.614,0.512) -- (-0.657,0.541) -- (-0.695,0.570) -- (-0.728,0.600) -- (-0.755,0.629) -- (-0.781,0.665) -- (-0.802,0.702) -- (-0.818,0.740) -- (-0.829,0.782) -- (-0.836,0.828) -- (-0.838,0.878) -- (-0.836,0.935) -- (-0.830,1.000);
\draw[band,gray!40] (-0.093,-1.000) -- (-0.091,-0.832) -- (-0.088,-0.707);
\draw[band,cM] (-0.088,-0.707) -- (-0.083,-0.607) -- (-0.077,-0.531) -- (-0.067,-0.471) -- (-0.053,-0.422) -- (-0.038,-0.386) -- (-0.019,-0.353) -- (-0.000,-0.329) -- (0.054,-0.266) -- (0.077,-0.237) -- (0.099,-0.206) -- (0.117,-0.174) -- (0.134,-0.136) -- (0.145,-0.097) -- (0.152,-0.054) -- (0.154,-0.008) -- (0.153,0.038) -- (0.147,0.079) -- (0.136,0.116) -- (0.121,0.151) -- (0.102,0.183) -- (0.078,0.214) -- (0.050,0.244) -- (-0.012,0.302) -- (-0.028,0.320) -- (-0.040,0.337) -- (-0.054,0.361) -- (-0.065,0.387) -- (-0.074,0.417) -- (-0.080,0.453);
\draw[band,cL] (-0.080,0.453) -- (-0.088,0.526) -- (-0.092,0.627) -- (-0.093,0.771) -- (-0.093,1.000);
\draw[cur,cM] (-0.182,-0.682) -- (-0.202,-0.677) -- (-0.218,-0.668) -- (-0.232,-0.654) -- (-0.243,-0.634) -- (-0.253,-0.609) -- (-0.260,-0.576) -- (-0.266,-0.537) -- (-0.271,-0.490) -- (-0.274,-0.413) -- (-0.275,-0.320) -- (-0.273,-0.207) -- (-0.269,-0.082) -- (-0.264,0.024) -- (-0.258,0.094) -- (-0.251,0.153) -- (-0.242,0.200) -- (-0.236,0.222) -- (-0.228,0.242) -- (-0.220,0.259) -- (-0.210,0.274) -- (-0.198,0.286) -- (-0.186,0.296) -- (-0.172,0.305) -- (-0.155,0.312) -- (-0.136,0.317) -- (-0.114,0.320) -- (-0.089,0.322) -- (-0.063,0.322) -- (-0.029,0.320) -- (0.007,0.316) -- (0.037,0.311) -- (0.066,0.305) -- (0.103,0.294) -- (0.134,0.281) -- (0.161,0.265) -- (0.185,0.246) -- (0.204,0.225) -- (0.221,0.201) -- (0.235,0.172) -- (0.246,0.140) -- (0.255,0.096) -- (0.262,0.044) -- (0.266,-0.023) -- (0.267,-0.094) -- (0.265,-0.156) -- (0.260,-0.215) -- (0.252,-0.270) -- (0.241,-0.321) -- (0.230,-0.360) -- (0.217,-0.395) -- (0.203,-0.429) -- (0.186,-0.460) -- (0.168,-0.489) -- (0.147,-0.515) -- (0.124,-0.541) -- (0.099,-0.563) -- (0.068,-0.588) -- (0.035,-0.609) -- (0.003,-0.626) -- (-0.048,-0.649) -- (-0.092,-0.665) -- (-0.127,-0.676) -- (-0.157,-0.682) -- (-0.182,-0.682);
\draw[cur,cL] (-0.562,0.535) -- (-0.623,0.542) -- (-0.686,0.554) -- (-0.747,0.571) -- (-0.795,0.590) -- (-0.813,0.599) -- (-0.829,0.609) -- (-0.841,0.619) -- (-0.851,0.630) -- (-0.857,0.640) -- (-0.860,0.651) -- (-0.861,0.662) -- (-0.858,0.674) -- (-0.851,0.690) -- (-0.839,0.706) -- (-0.821,0.724) -- (-0.799,0.743) -- (-0.771,0.764) -- (-0.737,0.788) -- (-0.651,0.840) -- (-0.573,0.882) -- (-0.542,0.897) -- (-0.515,0.907) -- (-0.493,0.914) -- (-0.473,0.918) -- (-0.455,0.918) -- (-0.440,0.915) -- (-0.430,0.910) -- (-0.420,0.903) -- (-0.411,0.894) -- (-0.403,0.883) -- (-0.396,0.870) -- (-0.390,0.854) -- (-0.380,0.815) -- (-0.374,0.765) -- (-0.373,0.712) -- (-0.378,0.663) -- (-0.382,0.642) -- (-0.387,0.623) -- (-0.397,0.599) -- (-0.410,0.579) -- (-0.425,0.563) -- (-0.444,0.551) -- (-0.467,0.542) -- (-0.494,0.537) -- (-0.525,0.534) -- (-0.562,0.535);
\draw[sep] (-0.196,0.405) -- (-0.215,0.386) -- (-0.230,0.365) -- (-0.243,0.341) -- (-0.254,0.312);
\draw[sep] (-0.196,0.405) -- (-0.171,0.423) -- (-0.144,0.437) -- (-0.113,0.447) -- (-0.080,0.453);
\draw[sep] (-0.254,-0.691) -- (-0.288,-0.678) -- (-0.321,-0.661) -- (-0.352,-0.640) -- (-0.382,-0.615) -- (-0.407,-0.587) -- (-0.430,-0.556) -- (-0.450,-0.522) -- (-0.468,-0.485);
\draw[sep] (-0.254,-0.691) -- (-0.223,-0.698) -- (-0.189,-0.703) -- (-0.147,-0.706) -- (-0.088,-0.707);
\draw[asc] (-0.196,0.405) -- (-0.218,0.439) -- (-0.242,0.472) -- (-0.268,0.503) -- (-0.294,0.532) -- (-0.320,0.556) -- (-0.345,0.576) -- (-0.370,0.592) -- (-0.394,0.604);
\draw[asc] (-0.196,0.405) -- (-0.181,0.378) -- (-0.170,0.354) -- (-0.162,0.334) -- (-0.157,0.311);
\draw[asc] (-0.254,-0.691) -- (-0.234,-0.661) -- (-0.219,-0.644) -- (-0.234,-0.651);
\draw[asc] (-0.254,-0.691) -- (-0.286,-0.747) -- (-0.328,-0.825) -- (-0.373,-0.915) -- (-0.414,-1.000) -- (-0.414,-1.000);
\draw[asc] (-0.414,1.000) -- (-0.414,1.000) -- (-0.436,0.956) -- (-0.458,0.918);
\fill[sepc] (-0.394,0.604) circle (1.0pt);
\fill[sepc] (-0.157,0.311) circle (1.0pt);
\fill[sepc] (-0.234,-0.651) circle (1.0pt);
\fill[sepc] (-0.458,0.918) circle (1.0pt);
\draw[land] (-0.254,0.312) circle (1.5pt);
\draw[land] (-0.080,0.453) circle (1.5pt);
\draw[land] (-0.468,-0.485) circle (1.5pt);
\draw[land] (-0.088,-0.707) circle (1.5pt);
\draw[ffr] (-1,-1) rectangle (1,1);
\draw[evf] (-0.196,0.405) +(-1.6pt,-1.6pt) rectangle +(1.6pt,1.6pt);
\draw[evf] (-0.254,-0.691) +(-1.6pt,-1.6pt) rectangle +(1.6pt,1.6pt);
\draw[evf] (-0.414,1.000) circle (1.5pt);
\draw[evf] (-0.414,-1.000) circle (1.5pt);
\draw[evo] (1.000,0.634) circle (1.5pt);
\draw[evo] (-1.000,0.634) circle (1.5pt);
\node[lab] at (0.18,0.80) {$O_2$};
\node[lab] at (0.45,0.0) {$O_1$};
\node[font=\scriptsize] at (0,-1.30) {(g) $t=0.78$};
\end{scope}
\begin{scope}[shift={(4.40,-2.70)}]
\draw[fax] (-1,0) -- (1,0); \draw[fax] (0,-1) -- (0,1);
\draw[band,gray!40] (-0.830,-1.000) -- (-0.825,-0.958) -- (-0.817,-0.918) -- (-0.808,-0.880) -- (-0.796,-0.845) -- (-0.783,-0.811) -- (-0.768,-0.778) -- (-0.750,-0.747) -- (-0.730,-0.716) -- (-0.708,-0.687) -- (-0.683,-0.658) -- (-0.656,-0.631) -- (-0.626,-0.603) -- (-0.593,-0.575) -- (-0.557,-0.546) -- (-0.468,-0.485);
\draw[band,cM] (-0.468,-0.485) -- (-0.322,-0.388) -- (-0.282,-0.359) -- (-0.251,-0.332) -- (-0.224,-0.302) -- (-0.203,-0.272) -- (-0.187,-0.240) -- (-0.175,-0.203) -- (-0.167,-0.170) -- (-0.162,-0.134) -- (-0.158,-0.092) -- (-0.155,-0.044) -- (-0.154,0.015) -- (-0.155,0.070) -- (-0.158,0.116) -- (-0.162,0.155) -- (-0.167,0.182) -- (-0.174,0.206) -- (-0.182,0.226) -- (-0.191,0.246) -- (-0.203,0.264) -- (-0.217,0.280) -- (-0.234,0.296) -- (-0.254,0.312);
\draw[band,cL] (-0.254,0.312) -- (-0.285,0.334) -- (-0.325,0.357) -- (-0.523,0.459) -- (-0.571,0.486) -- (-0.614,0.512) -- (-0.657,0.541) -- (-0.695,0.570) -- (-0.728,0.600) -- (-0.755,0.629) -- (-0.781,0.665) -- (-0.802,0.702) -- (-0.818,0.740) -- (-0.829,0.782) -- (-0.836,0.828) -- (-0.838,0.878) -- (-0.836,0.935) -- (-0.830,1.000);
\draw[band,gray!40] (-0.093,-1.000) -- (-0.091,-0.832) -- (-0.088,-0.707);
\draw[band,cM] (-0.088,-0.707) -- (-0.083,-0.607) -- (-0.077,-0.531) -- (-0.067,-0.471) -- (-0.053,-0.422) -- (-0.038,-0.386) -- (-0.019,-0.353) -- (-0.000,-0.329) -- (0.054,-0.266) -- (0.077,-0.237) -- (0.099,-0.206) -- (0.117,-0.174) -- (0.134,-0.136) -- (0.145,-0.097) -- (0.152,-0.054) -- (0.154,-0.008) -- (0.153,0.038) -- (0.147,0.079) -- (0.136,0.116) -- (0.121,0.151) -- (0.102,0.183) -- (0.078,0.214) -- (0.050,0.244) -- (-0.012,0.302) -- (-0.028,0.320) -- (-0.040,0.337) -- (-0.054,0.361) -- (-0.065,0.387) -- (-0.074,0.417) -- (-0.080,0.453);
\draw[band,cL] (-0.080,0.453) -- (-0.088,0.526) -- (-0.092,0.627) -- (-0.093,0.771) -- (-0.093,1.000);
\draw[curev] (-1.000,0.637) -- (-0.997,0.649) -- (-0.990,0.660) -- (-0.980,0.672) -- (-0.964,0.685) -- (-0.944,0.700) -- (-0.921,0.714) -- (-0.860,0.746) -- (-0.756,0.794) -- (-0.652,0.836) -- (-0.613,0.850) -- (-0.580,0.860) -- (-0.551,0.867) -- (-0.526,0.871) -- (-0.496,0.871) -- (-0.484,0.869) -- (-0.472,0.865) -- (-0.461,0.860) -- (-0.451,0.853) -- (-0.442,0.844) -- (-0.434,0.835) -- (-0.426,0.822) -- (-0.420,0.809) -- (-0.410,0.776) -- (-0.403,0.738) -- (-0.401,0.695) -- (-0.403,0.655) -- (-0.409,0.620) -- (-0.419,0.590) -- (-0.426,0.577) -- (-0.434,0.565) -- (-0.443,0.555) -- (-0.453,0.546) -- (-0.476,0.532) -- (-0.505,0.522) -- (-0.540,0.516) -- (-0.569,0.514) -- (-0.602,0.514) -- (-0.639,0.516) -- (-0.680,0.520) -- (-0.768,0.532) -- (-0.852,0.550) -- (-0.913,0.568) -- (-0.939,0.577) -- (-0.960,0.587) -- (-0.976,0.597) -- (-0.989,0.609) -- (-0.997,0.620) -- (-1.000,0.631);
\draw[curev] (-0.176,-0.705) -- (-0.198,-0.701) -- (-0.217,-0.691) -- (-0.234,-0.675) -- (-0.247,-0.655) -- (-0.258,-0.628) -- (-0.267,-0.595) -- (-0.275,-0.554) -- (-0.280,-0.505) -- (-0.284,-0.431) -- (-0.286,-0.341) -- (-0.285,-0.235) -- (-0.281,-0.118) -- (-0.275,0.004) -- (-0.270,0.075) -- (-0.262,0.137) -- (-0.252,0.186) -- (-0.244,0.210) -- (-0.236,0.230) -- (-0.227,0.247) -- (-0.216,0.262) -- (-0.204,0.274) -- (-0.190,0.285) -- (-0.173,0.294) -- (-0.155,0.301) -- (-0.133,0.306) -- (-0.108,0.310) -- (-0.081,0.311) -- (-0.048,0.311) -- (-0.011,0.309) -- (0.027,0.305) -- (0.061,0.300) -- (0.091,0.293) -- (0.125,0.282) -- (0.155,0.268) -- (0.181,0.252) -- (0.203,0.233) -- (0.222,0.211) -- (0.237,0.185) -- (0.250,0.156) -- (0.259,0.122) -- (0.268,0.075) -- (0.274,0.018) -- (0.278,-0.057) -- (0.278,-0.131) -- (0.275,-0.194) -- (0.269,-0.254) -- (0.260,-0.310) -- (0.248,-0.361) -- (0.237,-0.398) -- (0.224,-0.434) -- (0.208,-0.467) -- (0.192,-0.497) -- (0.173,-0.526) -- (0.152,-0.553) -- (0.128,-0.577) -- (0.104,-0.599) -- (0.074,-0.621) -- (0.041,-0.641) -- (0.007,-0.657) -- (-0.038,-0.676) -- (-0.082,-0.691) -- (-0.118,-0.701) -- (-0.150,-0.705) -- (-0.176,-0.705);
\draw[sep] (-0.196,0.405) -- (-0.215,0.386) -- (-0.230,0.365) -- (-0.243,0.341) -- (-0.254,0.312);
\draw[sep] (-0.196,0.405) -- (-0.171,0.423) -- (-0.144,0.437) -- (-0.113,0.447) -- (-0.080,0.453);
\draw[sep] (-0.254,-0.691) -- (-0.288,-0.678) -- (-0.321,-0.661) -- (-0.352,-0.640) -- (-0.382,-0.615) -- (-0.407,-0.587) -- (-0.430,-0.556) -- (-0.450,-0.522) -- (-0.468,-0.485);
\draw[sep] (-0.254,-0.691) -- (-0.223,-0.698) -- (-0.189,-0.703) -- (-0.147,-0.706) -- (-0.088,-0.707);
\draw[sep] (-1.000,0.634) -- (-1.000,0.634) -- (-0.905,0.648) -- (-0.826,0.664) -- (-0.762,0.680) -- (-0.737,0.689) -- (-0.718,0.697) -- (-0.731,0.688) -- (-0.744,0.675) -- (-0.756,0.661) -- (-0.767,0.645);
\draw[asc] (-0.196,0.405) -- (-0.220,0.442) -- (-0.247,0.478) -- (-0.275,0.512) -- (-0.304,0.541) -- (-0.332,0.566) -- (-0.360,0.586) -- (-0.386,0.601) -- (-0.412,0.610);
\draw[asc] (-0.196,0.405) -- (-0.180,0.376) -- (-0.168,0.350) -- (-0.161,0.327) -- (-0.155,0.301);
\draw[asc] (-0.254,-0.691) -- (-0.234,-0.661) -- (-0.219,-0.644) -- (-0.230,-0.650) -- (-0.247,-0.656);
\draw[asc] (-0.254,-0.691) -- (-0.286,-0.747) -- (-0.328,-0.825) -- (-0.373,-0.915) -- (-0.414,-1.000) -- (-0.414,-1.000);
\draw[asc] (-0.414,1.000) -- (-0.414,1.000) -- (-0.433,0.962) -- (-0.453,0.928) -- (-0.473,0.897) -- (-0.493,0.870);
\fill[sepc] (-0.412,0.610) circle (1.0pt);
\fill[sepc] (-0.155,0.301) circle (1.0pt);
\fill[sepc] (-0.247,-0.656) circle (1.0pt);
\fill[sepc] (-0.493,0.870) circle (1.0pt);
\draw[land] (-0.254,0.312) circle (1.5pt);
\draw[land] (-0.080,0.453) circle (1.5pt);
\draw[land] (-0.468,-0.485) circle (1.5pt);
\draw[land] (-0.088,-0.707) circle (1.5pt);
\draw[land] (-0.767,0.645) circle (1.5pt);
\draw[ffr] (-1,-1) rectangle (1,1);
\draw[evf] (-0.196,0.405) +(-1.6pt,-1.6pt) rectangle +(1.6pt,1.6pt);
\draw[evf] (-0.254,-0.691) +(-1.6pt,-1.6pt) rectangle +(1.6pt,1.6pt);
\draw[evf] (-0.414,1.000) circle (1.5pt);
\draw[evf] (-0.414,-1.000) circle (1.5pt);
\draw[evo] (1.000,0.634) circle (1.5pt);
\draw[evo] (-1.000,0.634) circle (1.5pt);
\node[lab,above left] at (-1,0.634) {$p\smash{\mathstrut}^{\prime}$};
\node[lab] at (-0.682,0.640) {$b$};
\node[font=\scriptsize] at (0,-1.30) {(h) $t=\frac{\sqrt3}2$: $p^{\prime}$};
\end{scope}
\begin{scope}[shift={(6.60,-2.70)}]
\draw[fax] (-1,0) -- (1,0); \draw[fax] (0,-1) -- (0,1);
\draw[band,gray!40] (-0.830,-1.000) -- (-0.825,-0.958) -- (-0.817,-0.918) -- (-0.808,-0.880) -- (-0.796,-0.845) -- (-0.783,-0.811) -- (-0.768,-0.778) -- (-0.750,-0.747) -- (-0.730,-0.716) -- (-0.708,-0.687) -- (-0.683,-0.658) -- (-0.656,-0.631) -- (-0.626,-0.603) -- (-0.593,-0.575) -- (-0.557,-0.546) -- (-0.468,-0.485);
\draw[band,cM] (-0.468,-0.485) -- (-0.322,-0.388) -- (-0.282,-0.359) -- (-0.251,-0.332) -- (-0.224,-0.302) -- (-0.203,-0.272) -- (-0.187,-0.240) -- (-0.175,-0.203) -- (-0.167,-0.170) -- (-0.162,-0.134) -- (-0.158,-0.092) -- (-0.155,-0.044) -- (-0.154,0.015) -- (-0.155,0.070) -- (-0.158,0.116) -- (-0.162,0.155) -- (-0.167,0.182) -- (-0.174,0.206) -- (-0.182,0.226) -- (-0.191,0.246) -- (-0.203,0.264) -- (-0.217,0.280) -- (-0.234,0.296) -- (-0.254,0.312);
\draw[band,cL] (-0.254,0.312) -- (-0.285,0.334) -- (-0.325,0.357) -- (-0.523,0.459) -- (-0.571,0.486) -- (-0.614,0.512) -- (-0.657,0.541) -- (-0.695,0.570) -- (-0.728,0.600) -- (-0.755,0.629) -- (-0.781,0.665) -- (-0.802,0.702) -- (-0.818,0.740) -- (-0.829,0.782) -- (-0.836,0.828) -- (-0.838,0.878) -- (-0.836,0.935) -- (-0.830,1.000);
\draw[band,gray!40] (-0.093,-1.000) -- (-0.091,-0.832) -- (-0.088,-0.707);
\draw[band,cM] (-0.088,-0.707) -- (-0.083,-0.607) -- (-0.077,-0.531) -- (-0.067,-0.471) -- (-0.053,-0.422) -- (-0.038,-0.386) -- (-0.019,-0.353) -- (-0.000,-0.329) -- (0.054,-0.266) -- (0.077,-0.237) -- (0.099,-0.206) -- (0.117,-0.174) -- (0.134,-0.136) -- (0.145,-0.097) -- (0.152,-0.054) -- (0.154,-0.008) -- (0.153,0.038) -- (0.147,0.079) -- (0.136,0.116) -- (0.121,0.151) -- (0.102,0.183) -- (0.078,0.214) -- (0.050,0.244) -- (-0.012,0.302) -- (-0.028,0.320) -- (-0.040,0.337) -- (-0.054,0.361) -- (-0.065,0.387) -- (-0.074,0.417) -- (-0.080,0.453);
\draw[band,cL] (-0.080,0.453) -- (-0.088,0.526) -- (-0.092,0.627) -- (-0.093,0.771) -- (-0.093,1.000);
\draw[cur,cP] (1.000,0.500) -- (0.950,0.507) -- (0.906,0.516) -- (0.871,0.526) -- (0.842,0.538) -- (0.820,0.551) -- (0.803,0.566) -- (0.793,0.583) -- (0.791,0.592) -- (0.790,0.602) -- (0.791,0.612) -- (0.793,0.622) -- (0.805,0.643) -- (0.823,0.664) -- (0.848,0.684) -- (0.878,0.703) -- (0.913,0.720) -- (0.954,0.736) -- (1.000,0.750);
\draw[cur,cL] (-1.000,0.750) -- (-0.797,0.805) -- (-0.679,0.834) -- (-0.634,0.842) -- (-0.596,0.846) -- (-0.564,0.847) -- (-0.537,0.845) -- (-0.512,0.838) -- (-0.490,0.827) -- (-0.471,0.812) -- (-0.455,0.791) -- (-0.442,0.764) -- (-0.433,0.732) -- (-0.427,0.695) -- (-0.425,0.657) -- (-0.427,0.623) -- (-0.432,0.592) -- (-0.440,0.565) -- (-0.452,0.543) -- (-0.467,0.524) -- (-0.485,0.509) -- (-0.507,0.497) -- (-0.534,0.487) -- (-0.565,0.481) -- (-0.604,0.476) -- (-0.649,0.474) -- (-0.702,0.475) -- (-0.765,0.478) -- (-0.838,0.483) -- (-1.000,0.500);
\draw[cur,cM] (-0.168,-0.741) -- (-0.195,-0.736) -- (-0.217,-0.726) -- (-0.236,-0.711) -- (-0.253,-0.689) -- (-0.266,-0.662) -- (-0.278,-0.627) -- (-0.287,-0.585) -- (-0.293,-0.535) -- (-0.299,-0.462) -- (-0.302,-0.374) -- (-0.302,-0.272) -- (-0.298,-0.160) -- (-0.293,-0.042) -- (-0.286,0.052) -- (-0.277,0.116) -- (-0.266,0.167) -- (-0.258,0.191) -- (-0.249,0.212) -- (-0.238,0.229) -- (-0.226,0.244) -- (-0.212,0.257) -- (-0.196,0.268) -- (-0.178,0.277) -- (-0.157,0.284) -- (-0.131,0.289) -- (-0.101,0.293) -- (-0.066,0.295) -- (-0.027,0.294) -- (0.016,0.292) -- (0.055,0.288) -- (0.090,0.282) -- (0.122,0.274) -- (0.154,0.263) -- (0.183,0.249) -- (0.206,0.233) -- (0.227,0.213) -- (0.244,0.190) -- (0.258,0.163) -- (0.269,0.132) -- (0.278,0.096) -- (0.286,0.042) -- (0.292,-0.026) -- (0.294,-0.104) -- (0.293,-0.180) -- (0.289,-0.245) -- (0.282,-0.307) -- (0.272,-0.365) -- (0.259,-0.417) -- (0.247,-0.455) -- (0.233,-0.490) -- (0.218,-0.522) -- (0.201,-0.551) -- (0.182,-0.579) -- (0.161,-0.604) -- (0.137,-0.627) -- (0.113,-0.647) -- (0.084,-0.667) -- (0.052,-0.685) -- (0.018,-0.700) -- (-0.019,-0.714) -- (-0.065,-0.728) -- (-0.105,-0.737) -- (-0.140,-0.741) -- (-0.168,-0.741);
\draw[sep] (-0.196,0.405) -- (-0.215,0.386) -- (-0.230,0.365) -- (-0.243,0.341) -- (-0.254,0.312);
\draw[sep] (-0.196,0.405) -- (-0.171,0.423) -- (-0.144,0.437) -- (-0.113,0.447) -- (-0.080,0.453);
\draw[sep] (-0.254,-0.691) -- (-0.288,-0.678) -- (-0.321,-0.661) -- (-0.352,-0.640) -- (-0.382,-0.615) -- (-0.407,-0.587) -- (-0.430,-0.556) -- (-0.450,-0.522) -- (-0.468,-0.485);
\draw[sep] (-0.254,-0.691) -- (-0.223,-0.698) -- (-0.189,-0.703) -- (-0.147,-0.706) -- (-0.088,-0.707);
\draw[sep] (-1.000,0.634) -- (-1.000,0.634) -- (-0.905,0.648) -- (-0.826,0.664) -- (-0.762,0.680) -- (-0.737,0.689) -- (-0.718,0.697) -- (-0.731,0.688) -- (-0.744,0.675) -- (-0.756,0.661) -- (-0.767,0.645);
\draw[asc] (-0.196,0.405) -- (-0.222,0.445) -- (-0.252,0.484) -- (-0.283,0.520) -- (-0.314,0.550) -- (-0.344,0.575) -- (-0.374,0.594) -- (-0.402,0.607) -- (-0.428,0.614);
\draw[asc] (-0.196,0.405) -- (-0.178,0.372) -- (-0.166,0.344) -- (-0.158,0.318) -- (-0.153,0.285);
\draw[asc] (-0.254,-0.691) -- (-0.234,-0.661) -- (-0.219,-0.644) -- (-0.238,-0.653) -- (-0.267,-0.661);
\draw[asc] (-0.254,-0.691) -- (-0.286,-0.747) -- (-0.328,-0.825) -- (-0.373,-0.915) -- (-0.414,-1.000) -- (-0.414,-1.000);
\draw[asc] (-0.414,1.000) -- (-0.414,1.000) -- (-0.440,0.949) -- (-0.466,0.906) -- (-0.494,0.869) -- (-0.521,0.841);
\fill[sepc] (-0.428,0.614) circle (1.0pt);
\fill[sepc] (-0.153,0.285) circle (1.0pt);
\fill[sepc] (-0.267,-0.661) circle (1.0pt);
\fill[sepc] (-0.521,0.841) circle (1.0pt);
\draw[land] (-0.254,0.312) circle (1.5pt);
\draw[land] (-0.080,0.453) circle (1.5pt);
\draw[land] (-0.468,-0.485) circle (1.5pt);
\draw[land] (-0.088,-0.707) circle (1.5pt);
\draw[land] (-0.767,0.645) circle (1.5pt);
\draw[ffr] (-1,-1) rectangle (1,1);
\draw[evf] (-0.196,0.405) +(-1.6pt,-1.6pt) rectangle +(1.6pt,1.6pt);
\draw[evf] (-0.254,-0.691) +(-1.6pt,-1.6pt) rectangle +(1.6pt,1.6pt);
\draw[evf] (-0.414,1.000) circle (1.5pt);
\draw[evf] (-0.414,-1.000) circle (1.5pt);
\draw[evf] (1.000,0.634) circle (1.5pt);
\draw[evf] (-1.000,0.634) circle (1.5pt);
\node[lab] at (-0.76,0.30) {$L_3$};
\node[lab] at (0.45,0.0) {$O_1$};
\node[lab] at (0.66,0.63) {$\lambda$};
\node[font=\scriptsize] at (0,-1.30) {(i) $t=1$};
\end{scope}
\begin{scope}[shift={(8.80,-2.70)}]
\draw[band,cL] (-1.0,0.80) -- (-0.74,0.80); \node[lab,right] at (-0.70,0.80) {bottom $\hat C_{t_0}$};
\draw[cur,cL] (-1.0,0.50) -- (-0.74,0.50); \node[lab,right] at (-0.70,0.50) {$C_t$, regular level};
\draw[curev] (-1.0,0.20) -- (-0.74,0.20); \node[lab,right] at (-0.70,0.20) {$C_t$, event level};
\draw[sep] (-1.0,-0.10) -- (-0.76,-0.10); \draw[land] (-0.76,-0.10) circle (1.5pt); \node[lab,right] at (-0.70,-0.10) {descending separatrix};
\draw[asc] (-1.0,-0.40) -- (-0.76,-0.40); \fill[sepc] (-0.76,-0.40) circle (1.0pt); \node[lab,right] at (-0.70,-0.40) {ascending separatrix};
\draw[evo] (-0.96,-0.70) +(-1.6pt,-1.6pt) rectangle +(1.6pt,1.6pt); \draw[evo] (-0.80,-0.70) circle (1.5pt); \node[lab,right] at (-0.70,-0.70) {event points};
\end{scope}
\end{tikzpicture}
\caption{Example~\ref{ex:infinity} in the square model $u=\bigl(\frac{x_1}{2+|x_1|},\frac{x_2}{1+|x_2|}\bigr)$ of $X_\Delta=\RP^1\times\RP^1$ (opposite sides identified; left and right: $D_{e_R}$, top and bottom: $D_{e_T}$). (a), (c), (e), (g), (i): regular levels; (b), (d), (f), (h): the events at $q_1$, $q_2$, $p$ and $p'$. The band is coloured by the component whose arcs land on each bottom interval, at an event just before it, grey on dead intervals; circles on the band are landing points, dots on $C_t$ special points, and event points are filled after their event. (b) Inset: near $q_1$. (f) $L_2'$ has shrunk to $p$, and $L_1$ touches $D_{e_T}$ at $p$ from the other side. (h) The separatrix of the boundary saddle at $p'$ lands at $b$. In $\R^2$ separatrices can turn back where $\partial_tf_t=0$ (h, i).}
\label{fig:sep}
\end{figure}
\FloatBarrier
\section{Complements to Section~\ref{sec:lifting}}\label{app:kurtz}

We use the notation of Section~\ref{sec:lifting}: we compare the dequantization path with other liftings and give a criterion for the absence of boundary events.

\subsection{Dequantization versus other liftings}\label{app:otherrays}

Let $D=\Delta_\cA\prod_{e\text{ outer}}\Delta_{\cA\cap e}$, the factor of $E_\cA$ that detects singular points in $(\C^*)^2$ and multiple roots of the edge polynomials (a discriminant whose variety is not a hypersurface is $1$). The singular points of $f_t$ in $(\C^*)^2$, resp.\ the multiple roots of $\psi^e_t$, form orbits of the finite group $G$ of the $\lambda\in(\C^*)^2$ with $\lambda^a=1$ for all $a\in\cA$, resp.\ $G_e$ of the $\mu\in\C^*$ with $\mu^k=\mu^{k'}$ for all exponents $k,k'$ of $\psi^e$, acting freely; if its order is even, it contains a real element of order two, and real event points come in pairs. A zero of $D(c(t))$ need not be an event: $1+by^2+y^4$, $b>0$, has the double roots $\pm i$ at $b=2$, where its discriminant $b^2-4$ changes sign, but no real root. We call the path \emph{real-generic} if every zero of $D(c(t))$ in $(0,1]$ is an event, as when $G$ and the $G_e$ have odd order and at every zero of $D(c(t))$ the singular points, resp.\ multiple roots, form one orbit, since a conjugation-invariant orbit of odd size contains exactly one real point.

\begin{proposition}[dequantization versus other liftings]\label{prop:otherrays}
Assume that $\cA$ contains a circuit; if it has none, no lifting has an event (\S\ref{ssec:window}), and there is nothing to compare. Let $D=\sum_\beta p_\beta c^\beta$, and let $u,v\in\R^{\cA}$, $0<g(u)<\infty$, be such that the ray $\{u+\varrho v\}_{\varrho\ge0}$ meets the codimension-one skeleton of the normal fan of $\operatorname{Newt}(D)$ only in relative interiors of codimension-one cones, transversally, traversing the normal cones of the vertices $\beta^{(0)},\dots,\beta^{(q)}$ of $\operatorname{Newt}(D)$; let $\operatorname{sc}(u,v)$ be the number of indices $j$ with $\sgn(p_{\beta^{(j-1)}}\varsigma^{\beta^{(j-1)}})\neq\sgn(p_{\beta^{(j)}}\varsigma^{\beta^{(j)}})$. For $c=\varsigma e^{\sigma u}$: (i) the dequantization path has no events for $\sigma>r/g(u)$; (ii) there is $\sigma_0$ such that for $\sigma\ge\sigma_0$ the lifting $h=-v$ has at least $\operatorname{sc}(u,v)$ events if its path is real-generic. Consequently, if $u$ is standard Gaussian and $v$ is independent of $u$ with a density (so that $g(u)>0$ and the transversality hold almost surely), and the paths in (ii) are almost surely real-generic for all large $\sigma$, then $\liminf_{\sigma\to\infty}\E\bigl[K_{-v}(\varsigma e^{\sigma u})\bigr]\ge\E[\operatorname{sc}(u,v)]$, where $K_{-v}(c)$ is the number of events of the path of $h=-v$ for the coefficients $c$, whereas $\E[K(\sigma)]\to0$.
\end{proposition}

\begin{proof}
(i) is Corollary~\ref{cor:window}(ii). (ii) Choose $0<\varrho_0<\dots<\varrho_q$ with $u+\varrho_jv$ in the open normal cone of $\beta^{(j)}$. The path of $h=-v$ is $\log|c(t)|=\sigma u+(s-1)v$, and at $s_j=1+\sigma\varrho_j$ the differences $\langle\beta^{(j)}-\beta,\sigma u+\sigma\varrho_jv\rangle$, $\beta\neq\beta^{(j)}$ in the support of $D$, are $\sigma$ times positive numbers. Hence for $\sigma$ larger than some $\sigma_0$ the term of $\beta^{(j)}$ dominates at every $s_j$, i.e.\ exceeds the sum of the absolute values of the others, so $\sgn D(c(s_j))=\sgn(p_{\beta^{(j)}}\varsigma^{\beta^{(j)}})$, and $D(c(s))$ vanishes between $s_{j-1}$ and $s_j$ whenever these signs differ; on a real-generic path its zeros are events. The last statement follows from Fatou's lemma and Theorem~\ref{thm:prob}.
\end{proof}

Such sign changes occur: for the cubic of Example~\ref{ex:cubic}, $D=c_{11}^3+27c_{00}c_{30}c_{03}$, and for $\varsigma=(+,+,+,-)$ its two terms have opposite signs, so $\operatorname{sc}(u,v)=1$ if and only if the ray crosses the wall, i.e.\ if $\langle w,u\rangle$ and $\langle w,v\rangle$ have opposite signs, where $w$ is the difference of the two vertices of $\operatorname{Newt}(D)$; for standard Gaussian $v$ this has probability $\frac12$.

\subsection{A criterion for the absence of boundary events}\label{app:kurtzcrit}

The event window (Theorem~\ref{thm:window}) confines the events of the dequantization path for every support. For an outer edge all of whose lattice points carry terms, a condition on consecutive coefficients excludes boundary events on that edge altogether; Corollary~\ref{cor:window}(i) is a multivariate analogue.

\begin{proposition}[no boundary events]\label{prop:kurtz}
Let $e$ be an outer edge all of whose lattice points lie in $\cA$, and write $\psi^e=y^{k_0}\sum_{k=0}^dc_ky^k$. Suppose that $|c_k|^2\ge5|c_{k-1}c_{k+1}|$ for $0<k<d$, or that $|c_k|^2>4|c_{k-1}c_{k+1}|$ for $0<k<d$ and the signs of the $c_k$ are constant or alternating. Then every polynomial on the dequantization path of $\psi^e$ has $d$ simple real roots; in particular the dequantization path of $f$ has no boundary event on $e$.
\end{proposition}

\begin{proof}
Along the path the ratios $|c_k|^2/|c_{k-1}c_{k+1}|$ are raised to the power $s\ge1$, so the hypotheses persist, and it suffices to treat $s=1$. For constant signs this is Kurtz's theorem \cite{Kur92}, and alternating signs are reduced to it by $y\mapsto-y$. In general put $q_k=|c_k|/|c_{k+1}|$, so that $q_k\ge5q_{k-1}$, and $\rho_k=\sqrt{q_{k-1}q_k}$ for $0<k<d$. On the circle $|y|=\rho_k$ one finds $|c_jy^j|\le5^{-(j-k)^2/2}|c_ky^k|$ for every $j$, hence $\sum_{j\neq k}|c_jy^j|\le2\sum_{i\ge1}5^{-i^2/2}\,|c_ky^k|<0.98\,|c_ky^k|$. By Rouch\'e's theorem, $\sum_kc_ky^k$ has exactly $k$ roots in $|y|<\rho_k$, hence exactly one root in each of the $d$ regions into which the circles $|y|=\rho_k$ divide $\C$. A non-real root would come with its conjugate in the same region, so all roots are real, and they are simple.
\end{proof}

For arbitrary signs the constant $4$ does not suffice: $1+5y+6y^2-y^3$ has the ratios $25/6$ and $36/5$, both larger than $4$, but only one real root.


\end{document}